\documentclass[english,11pt, a4paper]{article}
 
\title{The Loewner--Kufarev Energy and Foliations by Weil--Petersson Quasicircles}

\usepackage[T1]{fontenc}

\usepackage[hypertexnames=false]{hyperref}
\usepackage[utf8]{inputenc}

\usepackage{mathtools}
\usepackage{lmodern}
\usepackage{amsthm,amsmath,amsfonts,amssymb}
\usepackage{graphicx}
\usepackage{enumerate,enumitem}
\usepackage{authblk}
\usepackage{cite}

\let\OLDthebibliography\thebibliography
\renewcommand\thebibliography[1]{
  \OLDthebibliography{#1}
  \setlength{\parskip}{0pt}
  \setlength{\itemsep}{0pt plus 0.3ex}
}

\usepackage[activate={true,nocompatibility},final,tracking=true,kerning=true,spacing=true,factor=1100,stretch=15,shrink=15]{microtype}
\microtypecontext{spacing=nonfrench}

\allowdisplaybreaks

\makeatletter
\numberwithin{equation}{section}
\numberwithin{figure}{section}
\theoremstyle{plain}
\newtheorem{thm}{\protect\theoremname}
  \theoremstyle{plain}
  \newtheorem{lemma}[thm]{\protect\lemmaname}
    \newtheorem{prop}[thm]{\protect\propname}

\newtheorem{df}[thm]{Definition}

\numberwithin{thm}{section}

\newtheorem{cor}[thm]{Corollary}

\theoremstyle{remark}
\newtheorem*{rem}{Remark}

\makeatother

\usepackage{babel}
\providecommand{\propname}{Proposition}

\providecommand{\lemmaname}{Lemma}
\providecommand{\theoremname}{Theorem}

\renewcommand{\Im}{\imag}
\renewcommand{\Re}{\real}

\newcommand{\DD}{\mathbb{D}}
\newcommand{\CC}{\mathbb{C}}

\newcommand{\SLE}{\operatorname{SLE}}

\let \le \leqslant
\let \leq \leqslant
\let \ge \geqslant

\let \epsilon \varepsilon

\let \z \zeta

\usepackage{hyperref}
\hypersetup{
    colorlinks=true, 
    linkcolor=black, 
    urlcolor=black, 
    linktoc=all 
}

\newcommand{\abs}[1]{\left\lvert #1 \right \rvert}

\newcommand{\brac}[1]{\left \langle #1 \right \rangle}
\newcommand{\norm}[1]{\| #1 \|}

\newcommand{\mc}[1]{\mathcal{#1}}
\newcommand{\m}[1]{\mathbb{#1}}

\renewcommand\Re{\operatorname{Re}}
\renewcommand\Im{\operatorname{Im}}

\def\SLE{\operatorname{SLE}}

\def\mob{\text{M\"ob}}
\def\Diff{\operatorname{Diff}}
\def\QS{\operatorname{QS}}

\def\a{\alpha}
\def\b{\beta}
\def\g{\gamma}

\def\d{\delta}
\def\D{\Delta}

\def\z{\zeta}
\def\t{\theta}

\def\i{\iota}

\def\k{\kappa}

\def\s{\sigma}

\def\O{\Omega}
\def\vare{\varepsilon}

\def\Chat{\hat{\m{C}}}
\def\eps{\varepsilon}

\def\SLE{\operatorname{SLE}}

\def\dd{\mathrm{d}}

\newcommand{\ad}[1]{\overline{#1}}

\usepackage[dvipsnames]{xcolor}

 \def \M {\mc M}
 \def \1{\mathbf{1}}
\def\Ker{\operatorname{Ker}}
\def\Id{\operatorname{Id}}

 \author{Fredrik Viklund\thanks{KTH Royal Institute of Technology, Sweden, email: \url{frejo@kth.se} } \, and Yilin Wang\thanks{Institut des Hautes \'Etudes Scientifiques, France, email: \url{yilin@ihes.fr}}}
 \date{\today}

\begin{document}

\maketitle

\begin{abstract}
     We study foliations by chord-arc Jordan curves of the twice punctured Riemann sphere $\mathbb C \smallsetminus \{0\}$ using the Loewner--Kufarev equation. We associate to such a foliation a function on the plane that describes the ``local winding'' along each leaf. Our main theorem is that this function has finite Dirichlet energy if and only if the Loewner driving measure $\rho$ has finite Loewner--Kufarev energy, defined by $$S(\rho) = \frac{1}{2}\iint_{S^1 \times \mathbb{R}} \nu_t'(\theta)^2  \, d \theta d t$$ whenever $\rho$ is of the form $\nu_t(\theta)^2 d \theta d t$, and set to $\infty$ otherwise. Moreover, if either of these two energies is finite they are equal up to a constant factor, and in this case, the foliation leaves are Weil--Petersson quasicircles.
    
    This duality between energies has several consequences. The first is that the Loewner--Kufarev energy is reversible, that is, invariant under inversion and time-reversal of the foliation. Furthermore, the Loewner energy of a Jordan curve can be expressed using the minimal Loewner--Kufarev energy of those measures that generate the curve as a leaf. This provides a new and quantitative characterization of Weil--Petersson quasicircles. Finally, we consider conformal distortion of the foliation and show that the Loewner--Kufarev energy satisfies an exact transformation law involving the Schwarzian derivative.  The proof of our main theorem uses an isometry between the Dirichlet energy space on the unit disc and $L^2(2\rho)$ that we construct using  Hadamard's variational formula expressed by means of the Loewner--Kufarev equation.  
    Our results are related to $\kappa$-parameter duality and large deviations of Schramm-Loewner evolutions coupled with Gaussian random fields.

\end{abstract}
\tableofcontents
\newpage

\section{Introduction}
 Methods involving monotone evolution of simply connected domains in the  plane have been successfully applied to a range of problems in complex analysis and beyond. In recent years notably in the context of probabilistic lattice models, random conformal geometry, and physical growth processes \cite{carleson-makarov, Schramm_LERW_UST, LSW_exponent,Schramm:ICM,Smi:ICM, MilShe2016}.
 The Riemann mapping theorem and the notion of Carath\'eodory convergence allow to encode the dynamics by a continuous family of conformal maps from a reference domain, such as the unit disk $\m D$, onto the evolving domains. Motivated by the Bieberbach conjecture, Loewner proved in 1923 \cite{Loewner23} that the family of maps generated by progressively slitting a domain along a simple curve satisfies a  differential equation with a real-valued control function. Loewner's equation was ultimately instrumental in its resolution by de Branges about sixty years later \cite{debranges}. This method was extended by Kufarev \cite{Kufarev} and further developed by Pommerenke \cite{Pom1965} to cover general evolution families beyond slitted domains. In this case, the dynamics is described by the Loewner--Kufarev equation which is controlled by a family of measures. 
It is a fundamental problem to relate geometric properties of the evolving domains to analytic properties of the family of driving measures.

In this paper, we study the Loewner--Kufarev equation driven by members of the class of measures having finite \emph{Loewner--Kufarev energy}, a quantity that arises as a large deviation rate function for Schramm-Loewner evolutions, SLE$_\kappa$ when $\kappa \to \infty$ \cite{APW}. We show that the Loewner--Kufarev energy
     can be considered dual to the Loewner energy of a single Jordan curve, a quantity that also appears in the context of large deviations for SLE$_\kappa$, but this time when $\kappa \to 0+$  \cite{W1,RW,W2}. (Our methods and conclusions are entirely deterministic, however, and no results from probability are needed in this paper.)  
     More precisely, the boundaries of the evolving domains are Weil--Petersson quasicircles (equivalent to having finite Loewner energy) that foliate the twice punctured Riemann sphere $\mathbb C \smallsetminus \{0\}$. We prove that Loewner--Kufarev energy, now viewed as attached to the foliation, equals the Dirichlet energy of a geometric object --- the winding function. 
     
     This energy duality  leads to a complete characterization of measures generating Weil--Petersson quasicircles. 
     We are not aware of any other   case where there is an equivalent description of the generated non-smooth Jordan curve interfaces in terms of the driving measure. 
     As further consequences of energy duality, we will see that the induced 
    foliation and its energy
    exhibit several remarkable features and symmetries.

    Besides the introduction and analysis of this natural class of non-smooth foliations,
    and the implications to the theory of planar growth processes, another purpose of this paper is to provide large deviation counterparts to recent results from the theory of SLE processes and Liouville quantum gravity, including SLE duality and the mating-of-trees theorem. We will discuss links to random conformal geometry and open problems related to further connections 
    at the end of the paper.  
    
\subsection{Loewner energy and Weil--Petersson quasicircles} \label{intro_wp} 
Let $\g$ be a Jordan curve in $\m C$ which
separates $0$ from $\infty$
and write $D$ and  $D^*$ for the bounded and unbounded component, respectively, and set $\m D^* = \m C \smallsetminus \ad{\m D}$. If $f: \m D \to D$ and $h: \m D^* \to D^*$ are conformal maps fixing $0$ and $\infty$ respectively, then the M\"obius invariant \emph{Loewner energy} of $\gamma$ can be defined as
\begin{equation}\label{eq:loop_LE_def}
    I^L(\gamma) =  \mc D_{\m D} (\log \abs{f'}) + \mc D_{\m D^*} (\log \abs{h'})+4 \log \abs{f'(0)/h'(\infty)},
\end{equation}
where \[\mc{D}_{D}(u)=\frac{1}{\pi}\int_D  |\nabla u|^2 \dd A\] is the Dirichlet integral and $h'(\infty) = \lim_{z \to \infty} h'(z)$. Here and below we write $\dd A$ for two-dimensional Lebesgue measure. The quantity $I^L$ was originally introduced  in a different form \cite{W1, RW}, and the identity \eqref{eq:loop_LE_def} was established in \cite{W2}\footnote{The right-hand side of \eqref{eq:loop_LE_def} was introduced in \cite{TT06} and there referred to as the universal Liouville action (up to a constant factor $\pi$). Since we will discuss connections to random conformal geometry, we choose to use the term Loewner energy.}. See also Section~\ref{sec:further}.  

A Jordan curve $\g$ has finite Loewner energy if and only if it belongs to the set of Weil--Petersson quasicircles \cite{TT06}, a class of non-smooth chord-arc Jordan curves that has a number of equivalent characterizations from various different perspectives, see, e.g.,  
\cite{Cui,TT06,Shen_WP_1,Shen_Grunsky,W3,bishop-WP, johansson2021strong}. One way to characterize Weil--Petersson quasicircles is to say that their welding homeomorphisms $h^{-1} \circ f|_{S^1}$ (more precisely, the equivalence class modulo left-action by the group of M\"obius transformations preserving $S^1$) belong to the Weil--Petersson Teichm\"uller space $T_0(1)$, defined as the completion of $\mob(S^1) \backslash \Diff^\infty(S^1)$ using its unique homogeneous K\"ahler metric. Weil--Petersson Teichm\"uller space appears in 
several different contexts 
\cite{bowick1987holomorphic, Witten, nag1991non,Hong-Rajeev, pekonen, RSS_CFT, 
sharon20062d}, and carries an infinite dimensional K\"ahler-Einstein manifold structure. In fact, the Loewner energy itself is a K\"ahler potential for the Weil--Petersson metric on $T_0(1)$, see \cite{TT06} for definitions and a thorough discussion.

\subsection{Loewner--Kufarev energy}
Let $\mc N$ be the set of Borel measures $\rho$ on the cylinder $S^1 \times \mathbb{R}$ with the property that
$\rho(S^1 \times I)$ equals $|I|$ for any interval $I$.
Each $\rho \in \mc N$ can be disintegrated into a measurable family of probability measures $(\rho_t)_{t \in \m R}$ on $S^1$.

Let $(D_t)_{t \in \m R}$ be a family of simply connected domains such that $ 0 \in D_t \subset D_s$ for all $s \le t$.
We assume that for each $t$, the conformal radius of $D_t$ at $0$ equals $e^{-t}$ which implies $\bigcup_{t\in \m R} D_t = \m C$.  Let $(f_t: \m D \to D_t)_{t \in \m R}$ be the associated  family of conformal maps normalized so that $f_t(0)=0$ and $f'_t(0) = e^{-t}$. (Here and below we write $'$ for $\partial_z$.) 
By a result of Pommerenke \cite{Pom1965},  there exists 
$\rho\in \mc N$ such that $(f_t)_{t \ge s}$ satisfies the \emph{Loewner--Kufarev equation} 
$$    \partial_t f_t(z) = -z f'_t(z)H_t(z),  \qquad  H_t(z) = \int_{S^1} \frac{e^{i\t}+z}{e^{i\t}-z} \dd \rho_t(\t).
$$
 The Herglotz integral $H_t$ is holomorphic in $\m D$ with positive real part. The equation is interpreted in the sense that $t \mapsto f_t(z)$ is absolutely continuous.
Conversely, for any $\rho \in \mc N$ the monotone family of simply connected domains $(D_t)_{t \in \m R}$ and the corresponding family of conformal maps $(f_t)_{t \in \m R}$ can be recovered via the Loewner--Kufarev equation. We say that $(f_t)_{t \in \m R}$ is the \emph{Loewner chain} driven by $\rho$.
It is sometimes convenient to work with the family of inverse maps $(g_t=f_t^{-1} : D_t \to \m D)_{t \in \m R}$, referred to as the \emph{uniformizing Loewner chain}. See Sections~\ref{sect:Loewner--Kufarev-Equation} and \ref{subsec:whole_plane_Loewner_chain} for more details.

        In order to provide intuition, let us discuss a few examples.
    \begin{itemize}[itemsep= -2pt, topsep= -1pt]
    \item The classical definition of Loewner chain \cite{Loewner23} is the one generated by a continuous \emph{driving function} $t \mapsto W_t \in S^1$. This corresponds to $\rho_t$ being the Dirac measure at $W_t$. When $t\mapsto W_t$ is sufficiently regular, the complement of $\bigcap_{t\in \m R} D_t$ is a simple curve in $\mathbb{C}$ connecting $\infty$ with $0$. The Schramm-Loewner evolution, $\SLE_\k$, is the Loewner chain obtained when the driving function is chosen as $W_t=e^{i\sqrt \k B_{t}}$ where $(B_t)$ is one-dimensional Brownian motion on the real line. 
        \item  If $\rho_t = (2 \pi)^{-1}\dd \t$ for all $t \in \m R$, where $\dd \t$ denotes the Lebesgue measure on $S^1$, then $(D_t = e^{-t} \m D)_{t \in \m R}$ is the family of concentric disks centered at $0$.
        
        \item If $\rho \ll \dd t \times \dd \t$ with smooth density, then writing $\rho_t$ as $\rho_t(\t) \dd \t$, we have that the normal velocity (in the direction into $D_t$) of the interface $\partial D_t$ at the point $f_t(e^{i\t})$ equals $2 \pi |f'_t(e^{i\t})|\rho_t(\t)$. See \eqref{eq:normal}.  
        \item (Becker's condition) If there exists $k \in [0,1)$  such that $H_t$ takes values in the same compact set $\{z \in \m C \colon |z - 1| \le k|z +1|\}$ for a.e. $t \ge 0$ and (say) $\rho_t$ is uniform for all $t < 0$, then the interfaces $(\partial D_t)_{t \ge 0}$ are all $k$-quasicircles \cite{becker1972}. However, not all $k$-quasicircles can be generated in this way \cite[Thm.\,3]{GP}.
    \end{itemize}
 
For $\rho \in \mc N$,  the \emph{Loewner--Kufarev energy}  $S(\rho)$ introduced in \cite{APW} is
 defined by \[S(\rho) = \int_{\mathbb{R}} L(\rho_t) \,\dd t, \qquad \textrm{where}  \quad L(\rho_t) = \frac{1}{2}\int_{S^1}\nu_t'(\t)^2 \, \dd \t,\] 
 whenever $\dd \rho_t(\t)  =  \nu_t^2(\t) \dd \t$  
 and 
 $L(\rho_t)$
 is set to $\infty$ otherwise.
   Notice that for fixed $t$, $L(\rho_t)$ is the Dirichlet energy of $\nu_t$ on $S^1$, which is finite for a relatively regular measure. 
    We refer to measures with finite Loewner--Kufarev energy simply as \emph{finite energy measures}. We will discuss how the Loewner--Kufarev energy relates to conformally invariant random systems in Section~\ref{sec:further}. 

\subsection{Main results} \label{sec:main_results}
We call a monotone and continuous\footnote{Monotone is meant in the following sense: if $s < t$ then $D_t \subset D_s$ where $D_t$ is the bounded connected component of $\m C \smallsetminus \g_t$.
Continuity is 
understood with respect to the supremum norm modulo increasing reparametrization.} family $(\gamma_t)_{t \in \m R}$ of chord-arc Jordan curves whose union covers $\m C\smallsetminus \{0\}$ (or $\ad{\m D}\smallsetminus \{0\}$) a \emph{foliation}, see Section~\ref{sect:winding-function}. Individual curves in a foliation are called \emph{leaves}. Our definition allows more than one leaf  passing through a given point $z \in \m C \smallsetminus \{0\}$. If the family of interfaces $(\g_t  = \partial D_t)_{t\in \m R}$ arising from the Loewner--Kufarev equation driven by the measure $\rho$ form a foliation, we say that $\rho$ \emph{generates a foliation}.

\begin{thm}[Weil--Petersson leaves]\label{thm:WP-leaf}
   If $S(\rho) < \infty$, then $\rho$ generates a foliation of $\m C  \smallsetminus \{0\}$ in which every leaf is a Weil--Petersson quasicircle.
\end{thm}
See Corollary~\ref{cor:S_finite_foliation_WP} 
and Section~\ref{subsec:whole_plane}. Theorem~\ref{thm:WP-leaf} shows that any $\rho$ with $S(\rho) < \infty$ gives rise to a dynamical process in $T_0(1)$, see Section~\ref{sec:further}. Moreover, we can view $S(\rho)$ as the energy of the generated foliation. Later in Theorem~\ref{thm:main-jordan-curve} we prove that all Weil--Petersson quasicircles can be generated by a finite energy measure after proper normalization.

The next result is our main theorem. 
It shows that the foliation energy can, in fact, be expressed using a geometrically defined quantity, without reference to the Loewner chain. 

To state the result, we associate to a foliation $(\g_t)_{t \in \m R}$ a real-valued function $\varphi$ as follows. Given the leaf $\g_t$, let $g_t$ be the conformal map  of the bounded component of $\hat{\m C} \smallsetminus \g_t$ onto $\m D$, fixing $0$ and with positive derivative there. Given $z$ at which $\g_t$ has a tangent, we define $\varphi(z)$ to be the non-tangential limit at $z$ of the function 
$w\mapsto \arg [w g'_{t}(w)/g_t(w)]$. (We choose the continuous branch that equals $0$ at $0$.) This defines $\varphi$ arclength-a.e. on $\g_t$. Monotonicity of $(\g_t)_{t \in \m R}$ implies that there is no ambiguity in the definition of $\varphi(z)$ if $z \in \g_t \cap \g_s$. See Section~\ref{sect:winding-function} for more details. Modulo $2\pi$,  $\varphi(z)$ equals the difference of the argument of the tangent of $\gamma_{t}$ at $z$ and that of the tangent to the circle centered at $0$ passing through $z$. See Figure~\ref{fig:intro_fol}. We call $\varphi$ the \emph{winding function} associated with a foliation $(\g_t)_{t \in \m R}$. 
The simplest example of winding function is when the measure $\rho$ has zero energy, namely when $\rho_t$ is the uniform probability measure on $S^1$ for a.e. $t \in \m R$. In this case,  the associated foliation is the family of concentric circles centered at $0$, and the winding function is identically $0$. We discuss additional examples in Section~\ref{sect:examples}. 

\begin{figure}
    \centering
    \includegraphics[width=.4\textwidth]{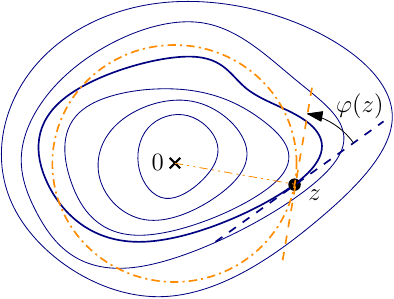}
    \caption{Illustration of the winding function $\varphi$.}
    \label{fig:intro_fol}
\end{figure}

In the present non-smooth foliation setting, a function defined on each leaf arclength-a.e.\ is not necessarily defined Lebesgue-a.e.\ in $\m C$, see Section~\ref{sect:winding-function}. However, we prove that 
if it is possible to extend the winding function $\varphi$ to an element in $W^{1,2}_{\mathrm{loc}}$, then the extension is unique. (This is one main reason why we choose to work in the setting of chord-arc curves.) See Proposition~\ref{prop:unique_extension}.
 Statements about the Dirichlet energy of $\varphi$ will be understood in terms of this extension whose existence is implicitly part of any such statement.

\begin{thm}[Energy duality]\label{thm:main0} Assume that $\rho \in \mc N$ generates a foliation and
   let $\varphi$ be the associated winding function on $\m C$. Then $\mc D_{\m C} (\varphi) <\infty$  if and only if $S(\rho) <\infty$ and 
   \[\mc D_{\m C} (\varphi) = 16 \, S(\rho).
 \]
   \end{thm}
   The proof of Theorem~\ref{thm:main0} is completed in Section~\ref{subsec:whole_plane}. 
 \begin{rem}
  The factor $16$ in Theorem~\ref{thm:main0} is consistent with the SLE duality relation $\k \leftrightarrow 16/\k$  \cite{Zhan_duality, Dub_duality}. See Section~\ref{sec:further}.
 \end{rem}

\subsection{Consequences of energy duality}
Theorem~\ref{thm:main0} has several implications. The first is the reversibility of the Loewner--Kufarev energy. Consider $\rho \in \mc N$ and the corresponding evolution family of domains $(D_t)_{t \in \m R}$. Applying $z\mapsto 1/z$ to the complementary domains $\hat{\m C} \smallsetminus D_t$, we obtain an evolution family of domains $(\tilde D_t)_{t \in \m R}$ upon time-reversal and reparametrization, which may be described by the Loewner--Kufarev equation with an associated driving measure $\tilde \rho$. While there is no known simple description of $\tilde \rho$ in terms of $\rho$, energy duality implies remarkably that the Loewner--Kufarev energy is invariant under this transformation. 
\begin{thm}[Energy reversibility]  \label{thm:main_rev}
We have $S (\rho) = S(\tilde \rho).$
\end{thm} 
See Theorem~\ref{thm:energy_rev}. 
Notice that although the Loewner--Kufarev energy is preserved, the reversal transformation redistributes the energy in a highly non-local manner: for instance, suppose that $\rho_t$ is the uniform measure when 
$t \notin [0,1]$. Then $L(\rho_t)$ is zero 
when $t \notin [0,1]$, yet it is not hard to see that the corresponding reversed foliation may have non-zero $L(\tilde \rho_t)$ for a.e. $t \le 0$ if $D_1 \neq e^{-1} \m D$.

By Theorem~\ref{thm:WP-leaf}, every leaf in the foliation generated by a measure with $S(\rho)< \infty$ satisfies $I^L(\g) < \infty$. 
Conversely, the next result shows that if $I^L(\g) < \infty$, then $\g$ can always be realized as a leaf in a foliation generated by Loewner evolution driven by a measure with $S(\rho) < \infty$ and $I^L(\g)$ is controlled by $S(\rho)$. We obtain a new and quantitative characterization of Weil--Petersson quasicircles.

\begin{thm}[Characterization of Weil--Petersson quasicircles]
\label{thm:main-jordan-curve}  
A Jordan curve $\g$ separating $0$ from $
  \infty$ is a Weil--Petersson quasicircle if and only if $\g$ can be realized as a leaf in the foliation generated by a measure $\rho$ with $S(\rho) < \infty$. Moreover, the Loewner energy of $\gamma$ satisfies the identity
  \[I^L(\g)  = 16 \inf_{\rho} S(\rho) + 2 \log |f'(0)/h'(\infty)|,\]
  where the infimum, which is attained, is taken over all $\rho \in \mc N$ such that $\g$ is a leaf of the generated foliation. The minimal energy foliation is the one obtained by equipotentials on both sides of $\gamma$.
\end{thm}

See Theorem~\ref{thm:dual_Jordan_curve} and Corollary~\ref{cor:WP-characterization}. 
By equipotential we mean the image of a circle 
centered at $0$ under the Riemann map from $\m D$ to a component of $\m C \smallsetminus \g$ fixing $0$ or taking $0$ to $\infty$. In this case, the winding function is harmonic in $\m C \smallsetminus \g$, see Section~\ref{subsec:jordan_curve}. This minimum is zero if and only if $\g$ is a circle centered at $0$, whereas $I^L(\g)$ is zero for all circles. This explains the presence of the derivative terms.
Corollary~\ref{cor:WP-LE-bound} of Theorem~\ref{thm:main-jordan-curve} shows that the Loewner energies of the leaves in a foliation generated by $\rho$ with $S(\rho) < \infty$ are uniformly bounded by $16 \, S(\rho)$.

We also prove the following identity that simultaneously expresses the interplay between Dirichlet energies under ``welding'' and ``flow-line'' operations along a \emph{bounded} curve  in a similar spirit as \cite[Cor.\,1.6]{VW1}. In our setup, given a chord-arc curve $\gamma$ separating $0$ and $\infty$, we define a winding function on $\gamma$  arclength-a.e.\ as above. We say that a Weil--Petersson quasicircle $\g$ is \emph{compatible} with $\varphi \in W^{1,2}_{\textrm{loc}}$, if the winding function of $\g$ coincides with the trace $\varphi|_{\g}$ arclength-a.e.

\begin{prop}[Complex identity] \label{prop:complex_id}
Let $\psi$ be a complex valued function on $\m C$ with $\mc D_{\m C}(\psi) = \mc D_{\m C}(\Re \psi) +  \mc D_{\m C}(\Im \psi)<\infty$ and $\g$ a Weil--Petersson quasicircle separating $0$ from $\infty$ compatible with $\Im \psi$. 
Let 
\begin{equation}\label{eq:complex_transform}
    \zeta (z): = \psi \circ f (z) + \log \frac{f'(z) z}{ f(z)}; \quad \xi(z) : = \psi \circ h(z) + \log \frac{h'(z) z}{ h(z)}.
\end{equation}
Then
$\mc D_{\m C} (\psi) = \mc D_{\m D} (\zeta) + \mc D_{\m D^*} (\xi)$. 
\end{prop}
See Section~\ref{sec:complex}.

Finally, we study the change of the Loewner--Kufarev energy under conformal transformation of the foliation. We consider the following setup. Let $\rho \in \mc N$ be a measure such that  $\rho_t$ is the uniform measure 
for $t < 0$ and
$$S_{[0,1]} (\rho) : = \int_0^1 L(\rho_t)\,\dd t < \infty.$$ 
 (The choice of the upper bound  
 $T = 1$ is only for notational simplicity and the result is easily generalized to other bounded time intervals.)

We have $D_0 = \m D$ and write $K_t = \ad{ \m D} \smallsetminus D_t$.
Let $\psi$ be a conformal map from a neighborhood $U$ of $K_1$ in $\ad {\m D}$ to a neighborhood $\tilde U$ of $\tilde K_1$, another compact hull in $\m D$, 
such that $\psi (K_1) = \tilde K_1$. The family of compact hulls $(\tilde  K_t: = \psi (K_t))_{t \in [0,1]}$ can be generated by a measure $\tilde \rho$ with the associated uniformizing Loewner chain $\tilde g_t: \m D \smallsetminus \tilde K_t \to \m D$.

\begin{thm}[Conformal distortion] \label{thm:conformal-distortion} 
We have the formulas
\[
L(\tilde \rho_t) - L( \rho_t) = \frac{1}{4}  \int_{S^1}  e^{2i\t} \mc S \psi_t (e^{i\t}) \rho_t (\t) \,\dd\t + \frac{1}{8}\left(|\tilde \rho_t| -  | \rho_t| \right)
\]
and
\begin{align*}
   S_{[0,1]} (\tilde \rho) - S_{[0,1]} (\rho) 
   = \frac{1}{4} \int_0^1  \int_{S^1}  e^{2i\t} \mc S \psi_t (e^{i\t}) \rho_t (\t) \,\dd\t\dd t + \frac{1}{8}\left(\log \tilde g_1'(0) - \log g_1'(0) \right),
\end{align*} 
  where $\psi_t  = \tilde g_t \circ \psi \circ g_t^{-1}$, $\mc S \psi = (\psi''/\psi')' - (\psi''/\psi')^2/2$ is the Schwarzian derivative, and $|\tilde \rho_t| = \tilde \rho_t(S^1)$.
\end{thm}

The proof is given in Section~\ref{subsec:variation_LK}. Theorem~\ref{thm:conformal-distortion} is related in spirit to conformal restriction formulas \cite{LSW_CR_chordal}. 
We will connect the expressions in Theorem~\ref{thm:conformal-distortion} to Brownian loop measures (see \cite{LSW_CR_chordal,LW2004loopsoup}) in a forthcoming paper with Lawler.

\subsection{Core argument for energy duality (Theorem~\ref{thm:main0})}\label{sect:core-argument}

We first derive a version of Theorem~\ref{thm:main0} for the unit disk. Let $\mc N_+$ be defined analogously to $\mc N$ but considering measures on $S^1 \times \m R_+$ and we define the corresponding Loewner--Kufarev energy by $S_+(\rho) = \int_0^{\infty} L(\rho_t) \, \dd t$.
The Loewner--Kufarev equation with the initial condition $f_0 (z) = z$ generates a Loewner chain $(f_t: \m D \to D_t)_{t \ge 0}$, where $(D_t)_{t\ge 0}$ is a monotone family of simply connected domains in $\m D = D_0$. It can be viewed as a special case of the whole-plane Loewner chain by extending the measure $\rho$ to a measure on $S^1 \times \m R$, where $\rho_t$ is the uniform probability measure on $S^1$ when $t < 0$, which implies $D_t = e^{-t} \m D$ for all $t \le 0$.

   \begin{thm}[Disk energy duality]\label{thm:main}
   Assume that $\rho \in \mc N_+$ generates a foliation
    of $\ad{\m D} \smallsetminus \{ 0\}$ and
   let $\varphi$ be the associated winding function.
   Then $\mc D_{\m D}(\varphi) < \infty$ if and only if $S_+(\rho) <\infty$ and $\mc D_{\m D} (\varphi) = 16 \, S_+(\rho).$
\end{thm}

The proof of Theorem~\ref{thm:main} is completed in Section~\ref{sec:disk_duality}.

The starting point of the proof is Hadamard's classical formula for the variation of the Dirichlet Green's function. We express Hadamard's formula using the Loewner--Kufarev equation: 
\[
-\partial_t G_{D_t}(z,w) = \int_{S^1} P_{\m D}(g_t(z), e^{i\t})P_{\m D}(g_t(w), e^{i\t}) \dd \rho_t(\t),
\]
where $P_{\m D}$ is the Poisson kernel for $\m D$ and $\rho \in \mc N_+$,
 see Lemma~\ref{loewner-hadamard} and also \cite[Thm.\,4]{Roth-schippers} for a similar result with additional smoothness assumptions. 
 The Sobolev space $\mc E_0 (\m D) = W^{1,2}_0 (\m D)$ is a Hilbert space when endowed with the Dirichlet inner product.
 Hadamard's formula and the orthogonal decomposition with respect to the Dirichlet inner product  along the Loewner evolution lead  to a correspondence between  $\mc E_0(\m D)$ and $L^2$-integrable functions on the cylinder $S^1 \times \m R_+$. 
More precisely, setting $L^2(2\rho) : = L^2(S^1 \times \m R_+, 2\rho)$, we define an operator
\begin{equation}\label{june24}
\iota: C_c^\infty(\m D) \to L^2(2\rho), \qquad \phi \mapsto \frac{1}{2\pi}\int_{\m D} \D ( \phi \circ f_t) (z) P_{\m D} (z, e^{i\t}) \,\dd  A(z).
\end{equation}

We prove the following theorem which is an important step in the proof of our main result and which we also believe to be of independent interest.

\begin{thm}[Foliation disintegration isometry]\label{thm:intro_isom}
  If $\rho \in \mc N_+$ generates a foliation of $\ad{\m D} \smallsetminus\{0\}$, then \eqref{june24} extends to a bijective isometry $\i : \mc E_0(\m D) \to L^2(2\rho)$ with the inverse mapping $\varkappa: L^2(2 \rho) \to \mc E_0 (\m D)$
\[
      \varkappa [u] (w) = 2 \pi \int_0^{\tau(w)} P_{\m D}[u_t\rho_t](g_t(w)) \, \dd t, \qquad u_t(\cdot) := u(\cdot, t). 
  \]
\end{thm}
See~Lemma~\ref{lem:disintegration_general},  Proposition~\ref{prop:kappa_formula}, and Theorem~\ref{thm:bi_isometry}.  
We show that $P_{\m D} [u_t \rho_t]$ is in the harmonic Hardy space $ \mathfrak h^1$ on $\m D$, 
and Theorem~\ref{thm:intro_isom} can be interpreted as a disintegration of finite Dirichlet energy functions into  $\mathfrak h^1$ functions. 
This implies the formula
\begin{align}\label{eq:intro_isom_T}
\phi_t^0 \circ g_t =  \varkappa \big[\iota [\phi] \mathbf{1}_{S^1 \times [t,\infty)}\big]
\end{align}
where $\phi_t^0 \in \mc E_0(\m D)$ is the zero-trace part of the function $\phi_t=\phi \circ f_t$ (so that $\phi_t - \phi_t^0$ is harmonic). See Corollary~\ref{cor:ortho_decomp_formula}.

\begin{rem} If $\rho$ is smooth and $t$ is fixed, the function $\theta \mapsto \iota[\phi](\t,t)$ can be interpreted as the inward pointing normal derivative in $\m D$ at $e^{i\t}$ applied to $\phi_t^0$. 
   The foliation disintegration isometry is closely related to the ``Hadamard operators'' considered in \cite{Haakan_GFF} in a $C^2$-smooth and strictly monotone setting,  see Section~\ref{sec:hadamard} for further discussion.
   Here, we consider chord-arc foliations which in general have leaves that are not $C^1$, and are not even locally Lipschitz graphs. Moreover, $t\mapsto \g_t$ is only continuous and monotone but  
   not strictly monotone. 
  Theorem~\ref{thm:intro_isom} allows us to work under 
  weak regularity assumptions and this level of generality is needed in order to include Weil--Petersson quasicircles and to obtain 
  the equivalence 
  in Theorem~\ref{thm:main0} and Theorem~\ref{thm:main}. 
  \end{rem}

  \begin{rem} By considering the Gaussian measures associated to the Hilbert spaces $\mc E_0(\m D)$ and $L^2(2\rho)$, Theorem~\ref{thm:intro_isom} immediately entails a decomposition of the Dirichlet Gaussian free field on $\m D$ into white-noise on the cylinder $S^1 \times \m R_+$ weighted by $2\rho$,
  generalizing the main result of \cite{Haakan_GFF}.
\end{rem}

The next step is to prove that $\rho \in \mathcal{N}_+$ with $S_+(\rho) < \infty$ generates a foliation so that Theorem~\ref{thm:intro_isom} can be applied. For this, we first derive a preliminary, weak energy duality result: under the 
assumption that $\rho$ is piecewise constant in time and  $\rho_t$ 
is strictly positive and smooth, we prove in Proposition~\ref{prop:weak_duality} that the winding function is defined, continuous  and piecewise smooth in $\overline{\m D}$. By essentially explicit computation, the identity $16 S_+(\rho) = \mathcal{D}_{\m D}(\varphi)$ follows. This result combined with an approximation argument is then used to prove that $\rho$ with $S_+(\rho)<\infty$ generates a foliation by Weil--Petersson quasicircles, see  Corollary~\ref{cor:S_finite_foliation_WP}. 

At this point the proof of Theorem~\ref{thm:WP-leaf} is completed in the case when $\rho \in \mc N_+$. 
  Then in Section~\ref{sec:disk_duality} we prove disk energy duality in full generality.  
  From the work in Section~\ref{sect:weak-WP} we know that $\rho$ with $S_+(\rho)< \infty$ generates a foliation. Using the inverse operator $\varkappa$ we prove that the winding function $\varphi=-\varkappa[2 \nu_t'/\nu_t]$ from which it follows immediately that $\varphi \in \mc E_0 (\m D)$ and that $\mc D_{\m D}(\varphi) = 16\, S_+(\rho)$.

  The final step assumes that $\rho \in \mc{N}_+$ generates a foliation whose winding function can be extended so that $\mc D_{\m D}(\varphi) < \infty$, and we want to show that $S_+(\rho) < \infty$. To point to the difficulty, at this stage we do not know that $\rho_t$ is absolutely continuous. However, applying  Theorem~\ref{thm:intro_isom}, we have that $\i[\varphi]\in L^2(2\rho)$.
Using the integrability information we deduce that $H'_t$ is in the Hardy space $\mathcal{H}^1$ and this implies that $\rho_t$ is absolutely continuous and that the density is differentiable a.e. It then follows that $S_+(\rho) < \infty$ using Theorem~\ref{thm:intro_isom}.

\bigskip 

{\bf Structure of the paper.} We recall basic definitions
in Section~\ref{sec:prelim}. The foliation disintegration isometry, Theorem~\ref{thm:intro_isom},
is established in Section~\ref{sec:hadamard}. In Section~\ref{sect:LK-energy} we define the Loewner--Kufarev energy, derive a few basic properties, and discuss examples of finite energy measures and associated foliations. 
Section~\ref{sect:weak-WP} is devoted to the proof of energy duality in the disk under strong regularity assumptions and here we obtain the important \emph{a priori} result that finite energy measures generate foliations whose leaves are Weil--Petersson quasicircles, see Corollary~\ref{cor:S_finite_foliation_WP}.
The general energy duality result for the unit disk, Theorem~\ref{thm:main}, is proved in Section~\ref{sec:disk_duality}.  We work in the set-up of the  Loewner--Kufarev equation in the unit disk in Sections~\ref{sec:prelim}-\ref{sec:disk_duality} and we introduce the whole-plane Loewner--Kufarev equation only in Section~\ref{sec:whole-plane}, where we prove whole-plane energy duality, see Theorems~\ref{thm:WP-leaf} and \ref{thm:main0}.
In Section~\ref{sec:application} we derive the consequences of energy duality, see Theorem~\ref{thm:main_rev}, \ref{thm:main-jordan-curve}, and Proposition~\ref{prop:complex_id}.  Section~\ref{subsec:variation_LK} is devoted to the proof of the conformal distortion formula, Theorem~\ref{thm:conformal-distortion}.  Section~\ref{sec:further} collects further comments, including a discussion of additional interpretations of our results as well as open problems.

\subsection*{Acknowledgments}
We are very grateful to H\aa kan Hedenmalm for pointing out the reference \cite{Haakan_GFF}, which helped us improve and simplify our argument considerably. We are also very thankful to an anonymous referee for providing many useful comments, and especially for suggesting and allowing us to use the current proof of Lemma~\ref{lem:matching_trace} which gives a stronger statement compared to our previous argument. We thank David Jerison and Greg Lawler for discussions and helpful input, Paul Laurain and Andrea Seppi for discussions on the implication to minimal surfaces, Yan Luo, Curt McMullen, Steffen Rohde, and Wendelin Werner for useful comments on an earlier version of our paper. F.V. acknowledges support by the Knut and Alice Wallenberg Foundation, the Swedish Research Council, and the Ruth and Nils-Erik Stenb\"ack Foundation.  Y.W. acknowledges support by NSF grant DMS-1953945.

\section{Preliminaries} \label{sec:prelim}
\subsection{Basic definitions}
For a bounded simply connected domain $D \subset \m C$, we write $G_D(z,w)$ for the Green's function associated to the positive Laplacian $\Delta = -(\partial_{xx} + \partial_{yy})$ with Dirichlet boundary condition, where $z = x + i y$. It is convenient to normalize the Green's function so that
$G_D(z,w)-\log|z-w|$ is harmonic in each variable, that is, so that $G_D = 2\pi \Delta^{-1}$.

For sufficiently regular domains $D$, the Poisson kernel is defined for $z \in D, \zeta \in \partial D$, by \[P_D(z,\z) :=\partial_{n(\zeta)} G_D(z,\zeta),\] where $\partial_{n(\z)}$ is the inward normal derivative at $\z$. If $D = \m D$ we have \[G_{\m D}(z,w) = -\log \left|\frac{z-w}{1-\overline{w}z}\right|, \qquad P_{\m D}(z,e^{i\t}) = \frac{1-|z|^2}{|z-e^{i\t}|^2}.\] 

If $\sigma$ is a finite measure on the unit circle $S^1$, identified with $[0,2\pi]_{/0\sim 2\pi}$, we write
\[
P_{\m D}[\sigma](z) = \frac{1}{2\pi}  \int_0^{2\pi} P_{\m D} (z,e^{i\t}) \,\dd \sigma(\t)
\]
for its Poisson integral. If $\dd \sigma = u \,\dd\t$ for $u \in L^1(S^1)$ we write simply $P_{\m D}[u](z)$. By Fatou's theorem, $P_{\m D}[u]$ has non-tangential limit $u$ Lebesgue-a.e.\ on $S^1$.

If $D$ is an open set, the Dirichlet energy of an almost everywhere defined  function $\phi : D \to \mathbb{R}$ is given by 
\[
\mc D_{D} (\phi) := \brac{\phi,\phi}_\nabla : = \frac{1}{\pi} \int_{D} |\nabla \phi|^2 \dd A(z)
\]
whenever $\phi$ has distributional first derivatives in $L^2(D)$. 
Let $W^{1,2}(D)$ be the Sobolev space of real-valued functions $\phi$ such that both $\phi$ and its weak first derivatives are in $L^2(D)$ with norm $\|\phi\|_{W^{1,2}(D)} = \|\phi\|_{L^2(D)} + \|\nabla \phi\|_{L^2(D)}$ and write $W^{1,2}_0(D)$ for the closure of smooth and compactly supported functions in $D$, $C_c^\infty(D)$, in $W^{1,2}(D)$. 

When $D$ is bounded, $\mc D_{D}^{1/2}$ is a norm equivalent to $\norm{\cdot}_{W^{1,2}}$ on $W^{1,2}_0(D)$ by the Poincar\'e inequality.
In this case, we write $\mc E_0(D)$ for $W^{1,2}_0(D)$ equipped with the norm $\mc D_{D}^{1/2}$ which is a Hilbert space endowed with the Dirichlet inner product $\brac{\cdot, \cdot}_{\nabla}$. Every $\phi$ with $\mc D_{\m C} (\phi) < \infty$ has a unique decomposition with respect to  $D$ 
    \begin{equation}\label{eq:orthogonal}
        \phi = \phi^0 + \phi^{h} \quad \text{and} \quad \brac{\phi^0, \phi^{h}}_\nabla =  0
   \end{equation}
   such that $\phi^0 \in \mc E_0(D)$ and $\phi^{h}$ with $\mc D_{\m C}(\phi^h)<\infty$ is harmonic in $D$, see, e.g., \cite[P.\,77]{adams} and  \cite[Appx.\,B]{VW1}.

   The Hardy space $\mc{H}^p$ (resp. harmonic Hardy space $\mathfrak h^p$) for $p > 0$ consists of functions $f$ holomorphic (resp. harmonic) in $\m D$ satisfying \[\sup_{0 \le r <1} \int_0^{2\pi}|f(re^{i\t})|^p \,\dd \t < \infty.\]

      Let $\gamma$ be a Jordan curve in $\m C$, and write $D$ and $D^*$ for the connected components of $\CC \smallsetminus \gamma$. Let $f$ be a conformal map from $\DD$ onto the bounded component $D$, and $h$ a conformal map from $\DD^*$ onto the unbounded component $D^*$ fixing~$\infty$. The Loewner energy of $\gamma$ is defined as 
       \begin{equation} \label{eq_disk_energy}
   I^L(\gamma) =  \mc D_{\m D} (\log \abs{f'}) + \mc D_{\m D^*} (\log \abs{h'})+4 \log \abs{f'(0)} - 4 \log \abs{h'(\infty)},
 \end{equation}
 where $h'(\infty):=\lim_{z\to \infty} h'(z) = \tilde h'(0)^{-1}$ and $\tilde h(z) := 1/h(1/z)$.
    The Loewner energy $I^L$ is finite if and only if either  Dirichlet integral on the right-hand side in \eqref{eq_disk_energy} is finite. We summarize this as a lemma.
    
\begin{lemma} [See {\cite[Thm.\,II.1.12]{TT06}}]
\label{thm_TT_equiv_T01}
Suppose $\gamma$ is a bounded Jordan curve. Then the following statements are equivalent:
\begin{enumerate}[itemsep= -2pt, topsep= -1pt]
\item $I^L (\gamma) < \infty;$
\item  $\mc D_{\m D} (\log \abs{f'}) = {\displaystyle \frac{1}{\pi}\int_{\DD} \left|\frac{f''(z)}{f'(z)}\right|^2  \dd  A(z)  < \infty;}$
\item $ \mc D_{\m D^*} (\log \abs{h'})  <\infty$.
\end{enumerate} 
\end{lemma}
There are several additional equivalent ways to define the Loewner energy, and the class of Jordan curves with finite Loewner energy coincides with the class of Weil--Petersson quasicircles, see \cite{W2}. Henceforth we will refer to Jordan curves with finite Loewner energy as Weil--Petersson quasicircles.

\subsection{The Loewner--Kufarev equation}\label{sect:Loewner--Kufarev-Equation}
We consider first the version for the unit disc and then later in Section~\ref{sec:whole-plane} the whole-plane version. 
Let $\M (\O)$ (resp. $\M_1 (\O)$) be the space of Borel measures (resp. probability measures) on $\O$ endowed with the topology induced by weak convergence on compact subsets. 
Let
$$\mc N_+ = \{\rho\in \M (S^1 \times \m R_+): \rho(S^1\times I) = |I| \text{ for any interval } I \}.$$
Any $\rho \in \mc N_+$ can be disintegrated into a family of measures $(\rho_t)_{t \ge 0}$ measurable in $t$, such that $\rho (\dd \t \dd t)= \rho_t (\dd \t) \dd t$ and $\rho_t \in \mathcal M_1(S^1)$. The disintegration is unique in the sense that any two disintegrations $(\rho_t), ( \widetilde \rho_t)$ of $\rho$ must satisfy $\rho_t = \widetilde \rho_t$ for a.e. $t$.
See, e.g., 
\cite[Sec.\,70]{DM_prob_potential}. We will work with the Loewner--Kufarev equation driven by measures $\rho \in \mathcal{N}_+$ and we now review some of the relevant definitions and facts. See, e.g., \cite[Ch.\,8]{rosenblum} for proofs and a detailed discussion in this general setting.

Given $\rho \in \mathcal{N}_+$, consider the associated Herglotz integrals
\begin{equation}\label{def:herglotz}H_t (z) := H[\rho_t](z) =  \int_{0}^{2\pi} \frac{e^{i \theta} + z}{e^{i \theta} -z} \,\dd\rho_t (\t).\end{equation}

The mapping $z \mapsto H_t(z)$ is a holomorphic function in $\mathbb{D}$ with positive real part and $t \mapsto H_t(z)$ is measurable. In this setting, the Loewner--Kufarev partial differential equation for $\m D$ reads
\begin{equation} \label{eq:loewner-pde}
\partial_t f_t (z) =   -z f_t'(z) H_t(z), \quad f_0 (z) = z.
\end{equation}
The equation \eqref{eq:loewner-pde} is understood in the sense that
 $t \mapsto f_t (z)$ is absolutely continuous,
and \eqref{eq:loewner-pde} is required to hold for a.e. $t \ge 0$.  The exceptional set can be taken to be the same for all $z \in \m D$. Throughout the paper, all differential equations are interpreted in a similar manner.
The unique solution to \eqref{eq:loewner-pde} gives rise to
a family of conformal maps $(f_t)_{t \ge 0}$ fixing $0$ such that for each $t$, $f_t:\m D \to D_t = f_t(\m D)$ where if $s < t$, then $D_t \subset D_s$. In the present case where the disintegrated measures are probability measures we have $f'_t(0) = e^{-t}$. (Indeed, $\partial_t \log f_t'(0) = - H_t (0 ) = -|\rho_t| \equiv -1$.)
We refer to the family $(f_t)_{t \ge 0}$ as the \emph{Loewner chain} driven by $\rho$.  We call the family of compact sets $(K_t=\overline{\m D} \smallsetminus D_t)_{t  \ge 0}$ the \emph{hulls} associated to the Loewner chain.

A converse statement is also true. Consider a monotone family of simply connected domains $(D_t)_{t \ge 0}$ containing $0$: $D_0 = \m D$, $0\in D_t$ for all $t$, and if $s < t$ then $D_t \subset D_s$.  Let $f_t : \m D \to D_t$ be the conformal map normalized by $f_t(0)=0, f'_t(0)>0$. According to a theorem of Pommerenke \cite[Satz 4]{Pom1965} (see also \cite[Thm.\,6.2]{Pom_uni} and \cite{rosenblum}), if $t\to f_t'(0)$ is a decreasing homeomorphism of $\m R_+ \to (0,1]$, one can reparametrize $(D_t)_{t\ge 0}$ so that $f_t'(0) = e^{-t}$, and there exists a measurable family of holomorphic functions $(H_t)_{t \ge 0}$ in $\m D$, uniquely defined for a.e. $t \ge 0$, with positive real part such that \eqref{eq:ODE} holds. The measurable family of probability measures $(\rho_t)_{t \ge 0}$ is obtained from the Herglotz-Riesz representation of $(H_t)_{t \ge 0}$. 

The upshot of the discussion is that the  measure
$\rho \in \mc N_+$, Loewner chain $(f_t)_{t\ge 0}$, \emph{uniformizing Loewner chain} $(g_t:=f_t^{-1}: D_t \to \m D)_{t \ge 0}$, the monotone family of simply connected domains $(D_t)_{t\ge 0}$, and the increasing family of hulls $(K_t)_{t \ge 0}$ each determine the others.

Let $\mathcal{L}_+$ be the set of Loewner chains  $(f_t)_{t\ge 0}$, and change notation by writing $f(z,t) = f_t(z)$. We endow $\mathcal{L}_+$ with the topology of uniform convergence of $f$ on compact subsets of $\mathbb{D} \times \m R_+$. In this setting, the continuity of the Loewner transform is well-known.

\begin{lemma}[See e.g., {\cite[Prop.\,6.1]{MilShe2016} and \cite{JohSolTur2012}}]\label{lem:cont-bij-loewner-transf}
	The Loewner transform $\rho \in \mathcal{N}_+ \mapsto f \in \mathcal{L}_+$ is a homeomorphism.
\end{lemma}

The uniformizing Loewner chain $(g_t=f_t^{-1}: D_t \to \m D)_{t \ge 0}$ driven by $\rho$ satisfies a similar equation
\begin{equation} \label{eq:ODE}
\partial_t g_t (z) =   g_t(z) H_t(g_t(z)), \quad g_0 (z) = z,
\end{equation}
for a.e. $t < \tau(z)$ where for $z \in \m D$,
\begin{align} 
\label{eq:tau_def}
\begin{split}
    \tau(z) & = \sup\{ t \ge 0, \, \text{ the solution of \eqref{eq:ODE} exists for all } s < t\}\\ & = \sup\{t \ge 0: z \in D_t\},
\end{split}
\end{align}
and we set by convention $\tau (z) = 0$ for $z \in S^1$.
Further, we can write
\[D_t = \{z \in \m D, \, \tau(z) > t\}.\] 
An important property which follows immediately from \eqref{eq:ODE} is the \emph{domain Markov property} of the Loewner transform.
\begin{lemma}\label{lem:domain_markov}
  Let $\rho \in \mc N_+$ and $(g_t)_{t \ge 0}$ the associated uniformizing Loewner chain. Fix $T \ge 0$, then $(\tilde g_s : = g_{s +T} \circ g_T^{-1})_{s\ge 0}$ is the uniformizing Loewner chain driven by the measure $\tilde \rho \in \mc N_+$ with disintegration $\tilde \rho_s = \rho_{s+T}$.
\end{lemma}

\subsection{Chord-arc foliations and the winding function}\label{sect:winding-function}
Let $\gamma$ be a rectifiable Jordan curve. We say that $\gamma$ is chord-arc if there exists a constant $A < \infty$ such that for all $z,w \in \gamma$, we have the inequality $|\gamma^{z,w}| \le A|z-w|$, where $\gamma^{z,w}$ is the subarc of $\gamma \smallsetminus \{z,w\}$ of smaller length. 
 
Consider a monotone family of simply connected domains $(D_t)_{t \ge 0}$ that can be described by Loewner evolution as in Section~\ref{sect:Loewner--Kufarev-Equation}. We call the family  $(\g_t := \partial D_t)_{t \ge 0}$ a \emph{non-injective chord-arc foliation} of $\ad {\m D} \smallsetminus \{0\}$ if
\begin{enumerate}[itemsep=-2pt]
    \item For all $t \ge 0$, $\g_t$ is a chord-arc Jordan curve.
    \item It is possible to parametrize each curve $\g_t, t \ge 0,$ by $S^1$ so that the mapping $t \mapsto \g_t$ is continuous in the supremum norm.
    \item For all $z \in \ad{\m D} \smallsetminus \{0\}$, $\tau(z) <\infty$, where $\tau$ is defined in \eqref{eq:tau_def}.
\end{enumerate}
For convenience we shall say simply foliation in what follows, but we stress that the chord-arc assumption is always in effect.
We refer to the Jordan curves $\g_t$ as \emph{leaves}. Non-injective here means that we do not require that there is a unique leaf that passes through each point in $\ad{\m D}$.
The following lemma shows that a foliation in the above sense indeed foliates the punctured unit disk.
\begin{lemma}\label{lem:tau_foliates}
  Assume that $(\g_t)_{t\ge 0}$ is a foliation of $\ad{\m D} \smallsetminus \{0\}$. Then for all $z \in \ad{\m D} \smallsetminus \{0\}$, we have  $z \in \gamma_{\tau(z)}$. In particular, $\bigcup_{t \ge 0} \g_t =  \ad{\m D} \smallsetminus \{0\}$. 
\end{lemma}
This lemma shows that the definitions of foliation in $\ad {\m D} \smallsetminus \{0\}$ and $\tau$ coincide with those given in Section~\ref{sec:main_results}.
\begin{proof}
From the monotonicity of $(D_t)$ and the continuity of $t \mapsto \g_t$, we have $\bigcup_{t > \tau(z)} D_t = D_{\tau(z)}$. Since $z\in {D_{t}}$ implies $\tau(z) > t$,  $z\not \in {D_{\tau(z)}}$. Similarly, $z \not\in \overline{{D_{t}}}$ implies $t > \tau(z)$ and it follows that $z \in  \overline{{D_{\tau(z)}}} \setminus {D_{\tau(z)}} = \partial D_{\tau(z)}$. This completes the proof. \end{proof}

If $\rho \in \mathcal{N}_+$ and the family of interfaces $(\g_t = \partial D_t)_{t \ge 0}$ produced by Loewner evolution forms a foliation of $\ad{\m D} \smallsetminus\{0\}$, we say that $\rho\in \mc N_+$ \emph{generates a foliation}. 
We will often choose the conformal parametrization for each $\g_t$, obtained by continuously extending $f_t: \m D \to D_t$ to $\overline{\m D}$ by Carath\'eodory's theorem and then restricting to $S^1$.

We will now define the \emph{winding function} $\varphi$ associated with a foliation $(\g_t = \partial D_t)_{t \ge 0}$. 
  It will be convenient to use the notation
    \begin{equation}\label{eq:vartheta}
    \vartheta[f](z) = \arg \frac{z f'(z)}{f(z)} = \int_0^z \dd \arg \frac{f(w) - f(z)}{w - z},
    \end{equation}
when $f$ is a conformal map (defined on a simply connected domain) fixing $0$. Here and elsewhere, the continuous branch of $z\mapsto \arg (z f'(z)/f(z))$ is chosen to equal $0$ at the origin. The integral is taken along any smooth path from $0$ to $z$. The function $\vartheta$ satisfies a chain rule
    \begin{equation}\label{eq:theta_chain}
        \vartheta[f \circ g] = \vartheta [f] \circ g + \vartheta [g]
    \end{equation}
    if $f,g, f\circ g$ are all conformal maps defined on simply connected neighborhoods of the origin which $f,g$ fix.
Given a curve $\g$ with some parametrization $\g(s)$, we say that $\g$ has a tangent at $\g(s_0)$ if there exists $\alpha$ such that
    \[
    \lim_{s  \to s_0 \pm }\frac{\g(s) - \g(s_0)}{|\g(s) - \g(s_0)|} = \pm e^{i \alpha}.
    \]    
    The existence of a tangent does not depend on the chosen parametrization. 
    For any Jordan curve $\g$ we write $\mathcal{T}_\g \subseteq \g$ for the differentiable points on $\g$, that is, set of points $z \in \g$ at which a tangent exists.
    \begin{lemma}\label{lem:boundary-behavior}
       Suppose $\g = \partial D$ is a chord-arc curve where $D$ contains $0$ and let $f:\m D \to D$ be the conformal map such that $f(0)=0, f'(0)>0$. Then $|\g \smallsetminus \mathcal{T}_\g| = 0$ and the following statements are equivalent. 
       \begin{enumerate}[label= (\roman*), itemsep=-2pt]
           \item $f(e^{i\t}) \in \mathcal{T}_\g$.
         \item The unrestricted limit of $\arg(f(z) - f(e^{i\t}))/(z-e^{i\t})$ in $\m D$ exists at $e^{i\t}$ and is finite.
           \item The non-tangential limit of $\arg(f(z) - f(e^{i\t}))/(z-e^{i\t})$ exists at $e^{i\t}$ and is finite.
           \item The non-tangential limit of $\arg f'(z)$ exists at $e^{i\t}$ and is finite.
       \end{enumerate}
       In this case, the limits are equal modulo $2\pi$.
    \end{lemma}
    \begin{proof}
    The equivalence of $(i)$ and $(ii)$ is the first part of \cite[Thm.\,II.4.2]{GM} and this equivalence holds without the chord-arc assumption. Next, $(ii)$ trivially implies $(iii)$ and the converse implication is given in \cite[Thm.\,5.5]{Pommerenke_boundary} since $\g$ is chord-arc so in particular $D$ is a John domain. The equivalence of $(iii)$ and $(iv)$ is \cite[Prop.\,4.11]{Pommerenke_boundary}.
    Finally, since $\g$ is rectifiable $|\g \smallsetminus \mathcal{T}_\g| = 0$ by \cite[Thm.\,6.8]{Pommerenke_boundary}.
    \end{proof}
     \begin{rem}
In the setting of Lemma~\ref{lem:boundary-behavior},  assuming $\g$ is a general Jordan curve (not necessarily chord-arc), $(i)$ still implies that $(ii), (iii)$, and $(iv)$ hold.
\end{rem}
    \begin{rem}
Note that non-tangential approach regions are preserved under conformal mapping between $\m D$ and a quasidisk, see \cite[Prop.\,1.1]{Jerison_Kenig_1982}. 
\end{rem}
\begin{lemma}\label{lem:goluzin}
  Suppose $f: \m D \to D \subset \m D$ is a conformal map satisfying $f(0)=0$. If $z \in \partial \m D$ is such that the non-tangential limit $\arg f'(z)$ exists and $f(z) \in \partial \m D$, then $\vartheta[f](z)=0$. 
\end{lemma}

\begin{proof}
If $z, f(z) \in \partial \m D$, then $w \mapsto (f(w) -f(z))/(w -z)$ takes values in the slit domain $\m C \smallsetminus \{- t f(z)/z, \, t\ge 0\}$. Indeed, if this were not true, then there would exist $t \ge 0$ and $w \in \mathbb{D}$ such that $f(w) + tw f(z)/z= f(z)(1+t)$. By Schwarz lemma this implies 
$$1+t=|f(w) + tw f(z)/z|\le |f(w)| + t |w| \le |w|(1+t),$$ a contradiction since $|w| < 1$.
 We can therefore define $w \mapsto \arg  (f(w) -f(z))/(w -z)$ continuously and unambiguously for $w \in \m D$ by choosing a branch of $\arg$ on this slit domain. 
Consider the path $s\mapsto sz$, $s \in [0,1]$ and set
\[
v(s)=\arg \frac{f(sz) - f(z)}{sz - z} - \arg \frac{f(z)}{z}, \quad s \in [0,1). 
\]
This is a continuous function such that $v(0)=0$.
Equation~\eqref{eq:vartheta} and Lemma~\ref{lem:boundary-behavior} show that 
$$\vartheta [f] (z) =\lim_{s \to 1-}  v(s).$$

To compute the value, we note that the tangent of $\partial \m D$ at $z$, oriented counterclockwise, has argument $\arg z + \pi/2$ modulo $2\pi$. Moreover, if $f(z) \in \partial \m D$ is a differentiable point of $\partial D$ then since $D \subset \m D$ it follows that $\partial D$ is tangent to $\partial \m D$ at $f(z)$ and this tangent has argument $\arg f(z) + \pi/2$ modulo $2\pi$.   
Using Lemma~\ref{lem:boundary-behavior} we have that $\lim_{s \to 1-} \arg \,  (f(sz) - f(z))/(sz - z)$ coincides with the non-tangential limit of $\arg f' (z)$ modulo $2\pi$, and the latter is equal to the difference of the arguments of the tangents at $f(z)$ and at $z$, namely $\arg f(z) / z$ modulo $2\pi$. This implies that $\vartheta [f] (z) = 0$ modulo $2\pi$.

Since the continuous function $w \mapsto (f(w) -f(z))/(w -z)$ takes values in the slit plane $\m C \smallsetminus \{- t f(z)/z, \, t\ge 0\}$, the function $s \mapsto v(s)$ does not take any two values differing by a non-zero multiple of $2\pi$. Since $v(0)=0$ it follows that $\vartheta [f] (z) = 0$.
\end{proof}

Let $\mathcal{T}_t = \mc T_{\g_t}$ be the set of $z \in \g_t$ at which $\g_t$ has a tangent. By Lemma~\ref{lem:boundary-behavior} and the remark following it, $\g_t \smallsetminus \mathcal{T}_t$ has arclength $0$ and if $z \in \mathcal{T}_t$,  
the harmonic function $\vartheta[g_t](w)$ has a non-tangential limit $\vartheta[g_t](z)$ as $w$ tends to $z$ inside $D_t$. 
We define the \emph{winding function} of $\g_t$ by 
\begin{equation}\label{eq:def_winding}
\varphi_{\g_t}(z) := \vartheta[g_t](z), \quad z \in \mathcal{T}_t, \quad t \ge 0.
\end{equation}

\begin{rem}
 Geometrically, modulo $2 \pi$, the winding function at $z \in \mathcal T_t$ is the angle between the tangent to $\g_{t}$ and that of the circle passing through $z$ centered at $0$, as described in Section~\ref{sec:main_results}. 
In Lemma~\ref{lem:turning_winding} we discuss an equivalent geometric definition of the winding function, suggested to us by a referee, at a differentiable point on a Jordan curve using the turning of a smooth curve approaching the point normally.
 \end{rem}

\begin{lemma}\label{lem:merge}
   Suppose $s< t$ and $z \in \mathcal{T}_s \cap \mathcal{T}_t$. Then $\varphi_{\g_s}(z) =  \varphi_{\g_t}(z)$. 
\end{lemma}
\begin{proof}
 By the chain rule \eqref{eq:theta_chain} and Lemma~\ref{lem:domain_markov}, it is enough to assume that $s=0$ and verify that $\varphi_{\g_t}(z) =0$ if $z\in S^1$. This follows from Lemma~\ref{lem:goluzin} which shows that $\vartheta [g_t] (z) = -\vartheta[f_t] (g_t(z)) = 0$.
\end{proof}
This observation allows us to define the winding function of a foliation.
\begin{df}
     Let $\mathcal{T}:=\cup_{t \ge 0} \mathcal{T}_t \subset \overline{\m D}$. We define the \emph{winding function} of a foliation,  $\varphi: \mathcal{T} \to \mathbb{R}$, by forgetting the leaf, namely  $\varphi(z) := \varphi_{\g_t}(z)$ if $z \in \mathcal{T}_t$. We set by convention $\varphi(0)=0$. 
\end{df}

 \begin{rem}If the leaves are all $C^1$, $\varphi$ is defined everywhere since $\mathcal{T} = \overline{\m D} \smallsetminus \{0\}$ by Lemma~\ref{lem:tau_foliates}.
 In the general case of a chord-arc foliation, a function defined arclength-a.e. on all leaves need not be defined on $\ad{\m D}\smallsetminus\{0\}$ Lebesgue-a.e. To illustrate the subtlety, it is possible to construct a subset $X$ of $\m D$ of full Lebesgue measure and a foliation of piecewise smooth curves such that each leaf intersects $X$ in exactly one point, see, e.g., \cite{FubiniFoiled}.  On the other hand, we shall prove in Proposition~\ref{prop:unique_extension} that a function in $\mc E_0(\m D)$ is determined by its ``values'' (interpreted appropriately) on the leaves if the curves are chord-arc. 
 \end{rem}

\begin{lemma} \label{lem:harmonic}
Suppose $\rho \in \mc N_+$ generates a foliation and let $T \ge 0$. If $\rho_t$ is the uniform measure for all $t \ge T$, then for all $z \in D_T$, 
\begin{equation}\label{eq:harmonic}
\varphi (z) = \vartheta [g_T] (z).
\end{equation}
In particular, $\varphi$ is defined and harmonic in $D_T$. Moreover, $(\g_t)_{t \ge T}$ are the equipotentials in $D_T$.
\end{lemma}

\begin{proof}
If $T = 0$, $\rho_t$ is the uniform measure  for all $t \ge 0$ which implies that $g_t(z) =  e^t z$ and $\varphi = 0 = \vartheta[g_0]$.
 Now let $t > T \ge 0$. By Lemma~\ref{lem:domain_markov}, we have $g_t \circ f_T (z) = e^{t - T} z$.
It follows from 
$g_t =(g_t \circ f_T) \circ g_T$, \eqref{eq:theta_chain} and $\vartheta [z \mapsto e^{t- T} z] = 0$ that
$$\varphi (z) = \vartheta[g_t] (z) = \vartheta [g_T] (z) 
 \qquad \forall z \in \partial D_t,$$
and we have $(\g_t = f_T (e^{T-t} S^1))_{t > T}$ which is by definition the family of equipotentials in $D_T$. We conclude the proof of \eqref{eq:harmonic} using $\cup_{t > T} \g_t = \cup_{t > T} g_T^{-1} (e^{T- t} S^1)= g_T^{-1} (\m D \smallsetminus \{0\}) = D_T \smallsetminus \{0\}$ (and by convention, $\varphi(0) = 0 =\vartheta [g_T] (0)$).
\end{proof}

\begin{lemma}\label{lem:zero_trace_varphi}
Suppose $\rho \in \mc N_+$ generates a foliation and let $ \varphi$ be the corresponding winding function. For every $T \ge 0$,  let $ \varphi^{(T)}$ denote the winding function associated to the truncated measure $\rho^{(T)}$ defined by $\rho^{(T)}_t : = \rho_t$ for $t \le T$ and $\rho_t^{(T)}$ is uniform for $t > T$.
Then
$(\varphi -\varphi^{(T)}) \circ f_T$ 
is the winding function generated by the measures $(\tilde \rho_s:= \rho_{T+s})_{s \ge 0}$.
\end{lemma}
\begin{proof}
From Lemma~\ref{lem:domain_markov}, the winding function $\tilde \varphi$ associated to $\tilde \rho$ is obtained from
$$\vartheta [\tilde g_s] = \vartheta [g_{s +T} \circ f_T] = \vartheta [g_{s+T}] \circ f_T - \vartheta [g_T] \circ f_T$$ 
where
$(\tilde g_s)_{s \ge 0}$ is the Loewner chain of $\tilde \rho$. We conclude the proof using Lemma~\ref{lem:harmonic}.
\end{proof}

In other words, decomposing the driving measure $\rho$ into the truncated (and uniformly extended) measure on $S^1 \times [0,T]$, and the measure $\tilde \rho: s \mapsto \rho_{T+s}$ amounts to decomposing the winding function into $\varphi^{(T)}$ which is defined and harmonic in $D_T$ and $\varphi - \varphi^{(T)}$ which vanishes (arclength-a.e.) on every $\g_t$ for $t \le T$. 

\section{The foliation disintegration isometry}\label{sec:hadamard}

\subsection{Hadamard's formula} \label{subsec:hadamard}
 In this section we use the Loewner--Kufarev equation to derive a formula for the dynamics of the Green's function. 
 The resulting expression is a version of Hadamard's classical formula in a smooth setting. We fix $\rho \in \mc N_+$ and let $(g_t : D_t \to \m D)_{t\ge 0}$ be the associated uniformizing Loewner chain.

For each $t \ge 0$, let $G_t(z,w)$ be the Green's function for $D_t$ with Dirichlet boundary condition. For $z,w$ fixed, we extend the definition of $G_t(z,w)$ to $t \ge \min\{\tau(z), \tau(w)\}$ by $0$. We will show that with this definition, $t \mapsto G_t(z,w)$ is absolutely continuous under the assumption that $\rho$ generates a foliation.  Note that $t \mapsto G_t(z,w)$ is always absolutely continuous for $t \in [0,\min\{\tau(z), \tau(w)\}) $ but need not be continuous at $\min\{\tau(z), \tau(w)\}$,
e.g., in the case when $\partial D_t$ disconnects an open set containing $z, w$ from $0$ as $t \to \tau = \tau(z) = \tau(w)$. 
\begin{lemma} \label{loewner-hadamard}\label{lem:G_to_0}
  Consider the Loewner--Kufarev equation driven by $\rho \in \mathcal{N}_+$. Then for a.e.\ $t < \min\{\tau(z) , \tau(w) \}$ and all $z,w \in D_t$, 
\[
-\partial_t G_{D_t}(z,w)= \int_0^{2\pi} P_{\m D}(g_t(z), e^{i\t})P_{\m D}(g_t(w), e^{i\t})\dd \rho_t(\t).
\]
If in addition $\rho$ is assumed to generate a foliation, then $t \mapsto G_t(z,w)$ is absolutely continuous for all $t$.
\end{lemma}
See also \cite[Thm.\,4]{Roth-schippers} for a similar result.

\begin{proof}
First note that for $z,w \in D_t$, by conformal invariance of the Green's function,
\begin{equation}\label{eq:G_t}
G_t(z,w) = G_{D_t}(z,w)= -\log \left|\frac{g_t(z) - g_t(w)}{1- \overline{g_t(w)} g_t(z)}\right|.
\end{equation}
Setting $z_t=g_t(z), w_t = g_t(w)$, the Loewner--Kufarev equation \eqref{eq:ODE} implies that for a.e. $t < \min\{\tau(z) , \tau(w) \}$, $\partial_t z_t = z_t H_t(z_t)$, $\partial_t w_t = w_t H_t(w_t)$ and so
\begin{align*}
-\partial_t G_t(z,w) & =  \Re \partial_t \log \frac{z_t - w_t}{1- \overline{w_t} z_t}  \\
& = \Re \frac{\overline{w_t}z_t \left(H_t(z_t)  + \overline{H_t(w_t)} \right)}{1- \overline{w_t} z_t} + \Re \frac{z_t H_t(z_t) - w_t H_t(w_t)}{z_t - w_t}.
\end{align*}
Next, note that
\begin{align}
 \Re  \frac{\overline{w_t}z_t \left(H_t(z_t)  + \overline{H_t(w_t)} \right)}{1- \overline{w_t} z_t}& = \int_0^{2\pi}\Re \frac{\overline {w}_t z_t \cdot 2(1-z_t\overline{w_t})}{(1-z_t \overline{w_t} )(e^{i\t} -z_t)(e^{-i\t} - \overline{w}_t)} \dd \rho_t(\t)  \nonumber \\
& = \int_0^{2\pi}\Re \frac{ 2 z_t \overline{w_t} }{(e^{i\t} -z_t)(e^{-i\t} - \overline{w}_t)} \dd \rho_t(\t)  \nonumber \\
& = \int_0^{2\pi} \Re \frac{ 2 z_t \overline{w_t} (e^{-i\t} -\overline{z_t})(e^{i\t} - w_t) }{|e^{i\t}-z_t|^2|e^{i\t}-w_t|^2}  \dd \rho_t(\t)   \label{eq1}.
\end{align}
Moreover,
\begin{align}
\Re  \frac{z_t H_t(z_t) - w_t H_t(w_t)}{z_t - w_t} & = \Re \int_0^{2\pi} \left( \frac{e^{i\t} + z_t}{e^{i\t} - z_t} + \frac{2w e^{i\t}}{(e^{i\t}-z_t)(e^{i\t}-w_t)} \right)  \dd \rho_t(\t)   \nonumber \\
 & = \Re \int_0^{2\pi}  \frac{(e^{i \t} + z)(e^{-i\t} -\overline{z}) |e^{i\t}-w_t|^2}{|e^{i \t}-z_t|^2|e^{i\t}-w_t|^2} \dd \rho_t(\t) \label{eq2} \\
 & \qquad + \Re \int_0^{2\pi}  \frac{ 2w_t e^{i\t}(e^{-i\t}-\overline{z_t})(e^{-i\t}-\overline{w_t})}{|e^{i \t}-z_t|^2|e^{i\t}-w_t|^2} \dd \rho_t(\t) \nonumber. 
\end{align}
Adding \eqref{eq1} and \eqref{eq2} and simplifying,
\begin{align*} -\partial_t G_t(z,w) & =  \int_0^{2\pi} \frac{1-|z_t|^2}{|e^{i\t}-z_t|^2} \cdot \frac{1-|w_t|^2}{|e^{i\t}-w_t|^2} \dd \rho_t(\t) 
= \int_0^{2\pi} P_{\m D}(z_t, e^{i\t})P_{\m D}(w_t, e^{i\t})\dd \rho_t(\t),
\end{align*}
as claimed.

Let us now assume $\rho  \in \mc N_+ $ generates a foliation and prove that $t \mapsto G_t(z,w)$ is absolutely continuous for all $t \ge 0$. For all $z,w \in \m D$, $z \neq w$, we need to show that $G_t(z,w) \to 0$ as $t\nearrow \min\{\tau(z),\tau(w)\}$. Without loss of generality, we assume that $\tau(z) \le \tau(w)$. 
 If $\tau(z) < \tau (w)$, it is clear from \eqref{eq:G_t} that $G_t(z,w) \to 0$ as $t \to \tau(z)$ since $|g_t(w)|$ is bounded away from $1$ for $t \in [0,\tau(z)]$. Assume now that $\tau  = \tau(z) = \tau(w)$.  Lemma~\ref{lem:tau_foliates} implies that $z, w \in \partial D_\tau$. Since $g_\tau : D_\tau \to \m D$ is a conformal map between two Jordan domains, $g_\tau$ extends to a homeomorphism $\ad D_\tau \to \ad {\m D}$ and $g_\tau(z) \neq g_\tau(w) \in S^1$.
 We claim that
 $$ \lim_{t \to \tau-} g_t (z) = g_\tau(z).
 $$ 
 Assuming this, it follows that $|g_t (z) - g_t(w)|$ is bounded away from $0$ on $[0,\tau]$. Therefore $G_t(z,w) \to 0$ as $t\to \tau-$ using \eqref{eq:G_t}. 
 To prove the claim, let $t_n, n=1,2,\ldots,$ be any sequence such that $t_n \nearrow \tau$ and such that $\zeta :=
 \lim_{n \to \infty} g_{t_n}(z)$ exists. Pick $\epsilon > 0$ and let $v \in D_\tau$ be such that the interior distance between $z$ and $v$ in $D_\tau$ is at most $\epsilon$. By the Beurling estimate there is a constant $C$ that does not depend on $z,v$ such that $|g_\tau(z) - g_\tau(v)| \le C \sqrt{\epsilon}$ and by monotonicity of the domains (Schwarz' lemma) the same estimate holds with $\tau$ replaced by any $t_n < \tau$. Let $n_0$ be large such that if $n > n_0$ then $|g_{t_n}(v)- g_{\tau}(v)| + |g_{t_n}(z)-\zeta| < \epsilon$. Such $n_0$ exists by definition of $\zeta$ and since $t \mapsto g_{t}(v)$ is continuous on $[0,\tau(v))$ and $\tau(v) > \tau(z)$.  Combining these estimates with the triangle inequality it follows that $|g_\tau(z) - \zeta| \le 2C \sqrt{\epsilon}$. 
\end{proof}

\subsection{Disintegration isometry}\label{subsec:hadamard_isometry}

In this section we assume that $\rho \in  \mc N_+$ generates a foliation of $\ad{\m D} \smallsetminus\{0\}$. 
We will write
\[
L^2(2\rho) := L^2(S^1 \times \m R_+ , 2\rho)
\]
and $\langle \cdot, \cdot \rangle_{L^2(2\rho)}$ for the corresponding inner product.
Recall also that we write $\mc{E}_0(\m D)$ for $W^{1,2}_0(\m D)$ equipped with the Dirichlet energy norm.
\begin{lemma}
\label{lem:disintegration_general} 
Suppose $\rho \in  \mc N_+$ generates a foliation and that $ \phi \in C^\infty_c (\m D)$.
Then
\begin{equation}\label{eq:iota_general}
 \i [ \phi ] (\t,t)  := \frac{1}{2\pi}\int_{\m D} \D ( \phi \circ f_t) (z) P_{\m D} (z, e^{i\t}) \,\dd  A(z) 
 \end{equation}
 is an element of $L^2 (2\rho)$.
Moreover, for all $T > 0$, 
\begin{align*}
\mc{D}_{\m D}(\phi) - \mc{D}_{\m D}(\phi^0_T) = \norm{\iota[\phi]\mathbf{1}_{S^1 \times [0,T]}}^2_{L^2(2\rho)}
\end{align*}
where $\phi_T^0 \in \mc E_0(\m D)$ is the zero-trace part of $\phi \circ f_T \in W^{1,2}(\m D)$ as in \eqref{eq:orthogonal}.

In particular,  the operator $\iota: \mc E_0(\m D) \cap C^\infty_c (\m D) \to L^2 (2\rho)$
is norm-preserving.
\end{lemma}

\begin{proof}
Let $\phi \in C_c^\infty(\m D)$ set $\mu = \D  \phi$ and write $G_0 = 2\pi \Delta^{-1}$, where we recall that we consider the positive Laplacian. (We write $G_0$ for both the operator and function.) For $z \in \m D$,
 $$ \phi (z)  = \frac{1}{2\pi}G_0 \mu  (z)  = \frac{1}{2\pi}\int_{\m D} G_0(w,z) \mu(w) \,\dd A(w).$$ 
 Using integration by parts, since $\phi$ is smooth and has compact support,
 \begin{align*}    
 \int_{\m D} |\nabla   \phi  (z)|^2 \dd A(z)
 & =  \int_{\m D} \phi  (z) \D  \phi  (z) \,\dd A(z)\\
 & = \frac{1}{2\pi} \int_{\m D} \int_{\m D} G_0(w,z)  \mu(w) \mu(z) \,\dd A (w) \,\dd A(z).
 \end{align*}
 As before we set $G_t(z,w) = G_{D_t}(z,w) = G_0 (z_t, w_t)$ for $z,w \in D_t$ (recall that $z_t = g_t (z)$ and $w_t = g_t(w)$) and $0$ otherwise.
 Let
 \[
  \phi_t :=  \phi \circ f_t , \quad \mu_t : = \Delta  \phi_t= |f_t'|^{2}  \mu \circ f_t
 \]
 and note that the zero-trace part of $\phi_t$ satisfies $\phi_t^0 =  G_0 \mu_t/2\pi$. Moreover, performing a change of variable,
 \begin{align}    
 &\int_{\m D} |\nabla   \phi_T^0  (z)|^2 \dd A(z) \nonumber
 \\ 
 & = \frac{1}{2\pi} \int_{\m D} \int_{\m D} G_0(w,z)  \mu_T(w) \mu_T(z) \,\dd A (w) \,\dd A(z) \nonumber \\
 & = \frac{1}{2\pi} \int_{\m D} \int_{\m D} G_0(w,z)  \mu \circ f_T (w) \mu \circ f_T(z) |f_T'(z)|^2 |f_T'(w)|^2  \dd  A (w) \,\dd A(z) \nonumber \\
 & = \frac{1}{2\pi} \int_{D_T} \int_{D_T} G_T(w,z)  \mu (w) \mu(z)  \dd  A (w) \,\dd A(z). \label{june10}
 \end{align}
 
Next, note that by Lemma~\ref{loewner-hadamard} $$-\partial_t G_t(z,w) = \int_0^{2 \pi} P_{\m D} (g_t(z), e^{i\t}) P_{\m D} (g_t(w),e^{i\t})  \dd \rho_t(\t)  \ge 0$$ 
for $z,w \in \mathbb{D}$ and that $P_{\m D} (g_t(w),e^{i\t}) \ge 0$. 
Since the singularity of the Green's function is logarithmic (and therefore integrable), we obtain
\begin{align*}
  \int_{0}^T \int_{D_t}\int_{D_t} -\partial_t G_t(z,w) \dd A(z) \dd A(w) \dd t  & =\int_0^T \int_{\mathbb{D}}\int_{\mathbb{D}}  -\partial_t G_t(z,w) \,\dd A(z) \,\dd A(w)\,\dd t  \\ & =\int_{\mathbb{D}}\int_{\mathbb{D}} \int_0^T -\partial_t G_t(z,w) \,\dd t \,\dd A(z) \,\dd A(w) \\
 & = \int_{\mathbb{D}}\int_{\mathbb{D}} G_0(z,w)-G_T(z,w) \, \dd A(z) \dd A(w)  < \infty.
\end{align*}
Therefore, using the smoothness of $h$, we can apply Fubini's theorem (repeatedly) and \eqref{june10} to compute
 \begin{align*}
 &2\pi\left(\int_{\m D} |\nabla   \phi  (z)|^2 \dd A(z)  - \int_{\m D} |\nabla   \phi_T^0 (z)|^2 \dd A(z)\right) \\
 & =\int_{\mathbb{D}}\int_{\mathbb{D}} \left(G_0(z,w)-G_T(z,w) \right) \mu(z) \mu(w) \, \dd A(z) \dd A(w) \\
 & = \int_{\m D} \int_{\m D} \int_0^{T} -\partial_t G_t(z,w) \mu(z) \mu(w) \,\dd t \dd A (w) \dd A(z)  \\
 & =\int_0^{T} \int_{D_t} \int_{D_t}  -\partial_t G_t(z,w) \mu(z) \mu(w)\, \dd A (w) \dd A(z) \dd t \\
 & = \int_0^{T}  \int_{D_t} \int_{D_t}  \mu(z) \mu(w)  \int_0^{2 \pi} P_{\m D} (g_t(z), e^{i\t}) P_{\m D} (g_t(w),e^{i\t}) \, \dd \rho_t( \t) \dd A (w) \dd A(z) \dd t \\
  & = \int_0^{T} \int_0^{2 \pi} \int_{D_t} \int_{D_t}  \mu(z) \mu(w)   P_{\m D} (g_t(z), e^{i\t}) P_{\m D} (g_t(w),e^{i\t}) \,  \dd A (w) \dd A(z) \dd \rho_t( \t) \dd t \\
 & =  \iint_{S^1 \times [0,T]}  \left(\int_{\m D} \mu_t(z) P_{\m D}(z,e^{i\t})  \dd A(z)\right) \left(\int_{\m D} \mu_t(w) P_{\m D}(w,e^{i\t})   \dd A(w)\right) \dd \rho(\t,t  )\\
 & = (2\pi)^2\iint_{S^1 \times [0,T]}  \abs{ \i [\phi] (\t,t) }^2 \dd \rho(\t, t),
 \end{align*}
 as claimed.  Letting $T \to \infty$,  by monotone convergence, we obtain that the operator $\iota: \mc E_0(\m D) \cap C^\infty_c (\m D) \to L^2 (2\rho)$
is norm-preserving.
\end{proof}
\begin{rem}
Given a signed measure $\mu$ supported on $\overline{\m D}$, the \emph{balayage} of $\mu$ to $\partial \m D$ is a measure $\nu$ on $\partial \m D$ such that the logarithmic potentials of $\mu$ and $\nu$ agree on $\m D^*$. Viewing $\mu$ as a charge density, $\nu$ represents the optimal way (with respect to logarithmic energy) to redistribute charge to $\partial \m D$ while keeping the potential fixed in the complementary domain. In our setting, for $t$ fixed, one can see that $\iota[\phi](\t,t)$ in fact equals the (density of) the balayage of the ``charge density'' $\D \phi_t \,\dd A$ to $\partial \m D$. See, e.g., \cite[Ch.\,IV]{Landkof}.
\end{rem}

 \begin{rem}
 The disintegration isometry is closely related to the general approach of \cite{Haakan_GFF},  see in particular Section~3 of that paper.
 Hedenmalm and Nieminen consider 
 $C^2$-smooth, strictly monotone deformations of domains, and do not employ the Loewner equation. 
In the smooth, strictly monotone case, a version of our Lemma~\ref{lem:disintegration_general} can be deduced from results of \cite{Haakan_GFF} via a ``polar coordinates'' change of variable which also relies on strong regularity assumptions of the interfaces.
\end{rem}

 Since $C^\infty_c(\m D)$ is dense in $\mc{E}_0(\m D)$, the operator $\iota$ extends to $\mc{E}_0(\m D)$ 
 and we have \[\mc D_{\m D}(\phi) = \norm{\iota[\phi]}^2_{L^2 (2\rho)}, \qquad \phi \in \mc E_0(\m D).\]

 We will now construct 
 an inverse operator $\varkappa : L^2 (2\rho) \to \mc{E}_0(\m D)$ as follows. For $u \in L^2 (2\rho)$ and $\phi \in C^{\infty}_c (\m D)$, we consider
$$\Phi_u : \quad \phi \mapsto 2 \iint_{\m R_+ \times S^1} u (t, \t) \iota[\phi](\t,t) \dd \rho (t, \t) = \brac{u, \iota[\phi]}_{L^2(2\rho)}$$
which  extends to a bounded linear operator $\mc{E}_0(\m D) \to \m R$. Indeed, the Cauchy-Schwarz inequality and Lemma~\ref{lem:disintegration_general} show that 
\begin{equation}\label{eq:Riesz_inequality}
\abs{
\brac{u, \iota[\phi]}_{L^2(2\rho)}} \le \norm {u}_{L^2 (2\rho)} \norm {\iota[\phi]}_{L^2 (2\rho)} = \norm {u}_{L^2 (2\rho)} \mc D_{\m D}(\phi)^{1/2}.    
\end{equation}
By the Riesz representation theorem there exists a unique $\varkappa [u] \in \mc{E}_0(\m D)$ with the property that for all $\phi \in C^\infty_c(\m D)$,
\begin{equation}\label{eq:inverse_char}
\frac{1}{\pi} \int_{\m D}  \varkappa [u] \Delta \phi \,\dd  A (z)  = \frac{1}{\pi} \int_{\m D} \nabla \varkappa [u] \nabla \phi \,\dd  A (z) 
= \Phi_u (\phi). 
\end{equation}
It follows immediately that $\varkappa \circ \iota = \Id_{\mc E_0 (\m D)}$. 
We now give the explicit formula for $\varkappa$. In this statement and below we use the notation $u_t(\cdot) := u(\cdot, t)$ for $u \in L^2 (2\rho)$.

\begin{prop}\label{prop:kappa_formula}
  Suppose $\rho \in \mathcal{N}_+$ generates a foliation. Let $u \in L^2 (2\rho)$. Then for a.e.\ $w \in \m D$,
   $$ \varkappa [u] (w) = 2 \pi \int_0^{\tau(w)} P_{\m D}[u_t\rho_t](g_t(w))  \, \dd t, 
  $$
  and $t\mapsto P_{\m D}[u_t\rho_t](g_t(w))  \in L^1([0,\tau(w)], \dd t)$.
\end{prop}

\begin{proof}
By linearity and splitting $u = u^+ - u^-$ with $u^+ := \max\{ u, 0\}$ and $u^- := \max\{ -u, 0\}$, it suffices to prove the proposition when $ u \ge 0$. 
We let 
\begin{align*}
\tilde \varkappa [u](w) & : =2 \pi \int_0^{\tau(w)} P_{\m D}[u_t\rho_t](g_t(w))  \, \dd t 
=  \int_0^{\tau(w)} \int_{S^1} P_{\m D} (g_t(w), e^{i\t}) u(\t, t) \, \dd \rho_t(\t) \dd t.
\end{align*}

We first prove that 
for all $\phi \in C^\infty_c (\m D)$,
\begin{align}\label{eq:converse_1}
    &2 \int_0^\infty \int_{S^1} u(\t, t) \iota[\phi] (\t, t) \dd \rho_t(\t) \dd t = \frac{1}{\pi}  \int_{\m D} \D \phi(w)  \tilde \varkappa [u](w) \dd A (w). 
\end{align}

For this, we will repeatedly interchange the order of integration in the following computation.
To justify this, we let 
$$\phi^+ := \D^{-1} \max\{\D \phi, 0\} \quad \text{ and } \quad \phi^- : =\D^{-1} \max \{-\D \phi, 0 \}. $$
We have $\phi = \phi^+ - \phi^-$. Since $\phi$ is smooth, $\max\{\D \phi, 0\}$ is Lipschitz continuous. By elliptic regularity (see, e.g., \cite[Sec.\,4.3]{Gilbarg_Trudinger}), $\phi^+, \phi^- \in  \mc E_0 (\m D) \cap C^{2,\a} (\m D)$ for any $\a <1$ and clearly $\D \phi^\pm \ge 0$.
Moreover, recalling that $\phi_t^\pm = \phi^\pm \circ f_t$, we have $\D \phi^\pm_t= \D (\phi^\pm \circ f_t) =  |f_t'|^2(\D \phi^\pm) \circ f_t  \ge 0$. So it follows that
$$\iota [\phi^\pm] (\t, t) = \frac{1}{2\pi} \int_{\m D} \D \phi^\pm_t (z) P_{\m D} (z, e^{i\t}) \dd A(z) \ge 0.$$
Since $u,\iota [\phi^\pm] \in L^2(2\rho)$, the Cauchy-Schwarz inequality shows that $u \cdot \iota [\phi^\pm] \in L^1 (2\rho)$. Using Fubini's theorem, this implies that $u (\t,t) \D \phi_t (z) P_{\m D}(z,e^{i\t})$ is in $L^1 (\m D  \times S^1 \times \m R_+, \dd A \times \dd \rho_t(\t) \times \dd t)$. 
Therefore the left-hand side of
\eqref{eq:converse_1} equals
\begin{align*}
  &   2 \int_0^\infty \int_{S^1} u(\t, t) \left[\frac{1}{2\pi} \int_{\m D} \D \phi_t (z) P_{\m D} (z, e^{i\t}) \dd A(z) \right]\dd \rho_t(\t) \dd t \\
      & = \frac{1}{\pi}  \int_0^\infty \int_{\m D} \D \phi_t (z)  \int_{S^1} u(\t, t) P_{\m D} (z, e^{i\t}) \dd \rho_t(\t) \dd A (z) \dd t  \\
      & = \frac{1}{\pi}  \int_0^\infty \int_{\m D} (\D \phi) \circ f_t(z) |f_t'(z)|^2   \int_{S^1} u(\t, t) P_{\m D} (z, e^{i\t}) \dd \rho_t(\t)  \dd A (z)\dd t\\
      & = \frac{1}{\pi}  \int_0^\infty \int_{D_t} \D \phi(w)    \int_{S^1} u(\t, t) P_{\m D} (g_t(w), e^{i\t}) \dd \rho_t(\t)  \dd A (w)\dd t \\
      & = \frac{1}{\pi}  \int_{\m D} \D \phi(w)  \int_{0}^{\tau(w)}  \int_{S^1} u(\t, t) P_{\m D} (g_t(w), e^{i\t}) \dd \rho_t(\t) \dd t  \dd A (w)
\end{align*}
which proves \eqref{eq:converse_1}. 
This implies that $\tilde \varkappa [u] \in \mc E(\m D)$ since for all $\phi \in C^\infty_c (\m D)$, we have by Cauchy-Schwarz inequality
$$\abs{\brac{\D \phi, \tilde \varkappa [u]}_{L^2 (\m D)}} \le C \norm{u}_{L^2(\rho)} \norm{\nabla \phi}_{L^2 (\m D)}$$
for some universal constant $C$.

Given this, \eqref{eq:inverse_char} implies that $\varkappa [u] - \tilde\varkappa [u]$ is weakly harmonic, and therefore, by Weyl's lemma, it is harmonic in $\m D$. Since we know that $\varkappa[u] \in \mc E_0(\m D)$, to complete the proof, it is therefore enough to show that $\tilde \varkappa[u] \in \mc E_0 (\m D)$ (namely, that it has zero trace). For this, we show that $\tilde \varkappa[u]$ can be extended to a function $\psi \in W^{1,2} (e \m D)$ where $\psi(w) = 0$ for all $1<|w|<e$ (see, e.g., \cite[Prop.\,9.18]{brezis}).
To construct the extension, we define a measure $\hat \rho \in \mc N_+$ with $\hat \rho_{t} = \rho_{t-1}$ for $t \ge 1$ and uniform for $t \in [0,1)$. Similarly, let $\hat u (\t, t) : = u (\t, t-1)$ for $t \ge 1$, and $\hat u(\t,t) = 0$ for $t \in [0,1)$. 

We clearly have $\hat u \in L^2 (2\hat \rho)$, and $ \psi (w): = \tilde \varkappa [\hat u] (e^{-1} w)= 0$ if $1 <|w| < e$. By construction, we have $\psi \in \mc E_0 (e \m D)$, hence $W^{1,2} (e \m D)$, from the proof above. 
Moreover, if $\hat g_t$ is the uniformizing Loewner chain of $\hat \rho$, then for $t \in [0,1]$ we have $\hat g_t(z) = e^t z$ and 
 for $t \ge 1$, $\hat g_t (z) = g_{t - 1} (e z)$. It follows that $\tilde \varkappa [u] (w) = \tilde \varkappa [\hat u] (e^{-1} w) = \psi (w)$ for $w \in \m D$.
This shows that $\varkappa = \tilde \varkappa$ in $\mc E_0 (\m D)$.

Finally, since $u \ge 0$, it follows that 
$P_{\m D} [u_t \rho_t] (g_t(w)) \in L^1([0,\tau(w)], \dd t)$ for a.e. $w \in \m D$.
\end{proof}

\begin{cor} \label{cor:ortho_decomp_formula}
  Let $\phi \in \mathcal{E}_0 (\m D)$ and $T \ge 0$. The orthogonal decomposition of $\phi$ into a function in $\mc E_0 (D_T)$ and a function in $\mc E_0(\m D)$ harmonic in $D_T$ \eqref{eq:orthogonal} is given by 
  $$\phi^{0,T} : = \varkappa \big[\iota[\phi] \mathbf{1}_{S^1 \times [T,\infty)}\big] \in \mc E_0 (D_T) \quad \text{and} \quad \phi^{h,T} := \varkappa \big[\iota[\phi] \mathbf{1}_{S^1 \times [0, T)}\big].$$
\end{cor}

\begin{rem}
Using the notation of Lemma~\ref{lem:disintegration_general}, a consequence of the corollary is that $\phi^{0,T} = \phi^0_T \circ g_T$.
\end{rem}
\begin{proof}
We define  $$\phi^{0,T} : = \varkappa \big[\iota[\phi] \mathbf{1}_{S^1 \times [T,\infty)}\big] \quad \text{and} \quad \phi^{h,T} := \varkappa \big[\iota[\phi] \mathbf{1}_{S^1 \times [0, T)}\big].$$
   Since $\varkappa \circ \iota [\phi] = \phi$, we have $\phi^{0,T} + \phi^{h,T} = \phi$.
   
   We now show that $\phi^{0,T}\in \mc E_0 (D_T)$. For this, consider the Loewner chain driven by $(\hat \rho_t : = \rho_{T+t})_{t \ge 0}$ and the associated operator $\hat \varkappa$. We obtain from the domain Markov property Lemma~\ref{lem:domain_markov} that
   $$
   \phi^{0,T} \circ f_T = \hat \varkappa \big[\iota[\phi](\cdot, \cdot + T)\big].$$ 
   Then Proposition~\ref{prop:kappa_formula} shows that $   \phi^{0,T} \circ f_T \in \mc E_0 (\m D).$
   Using conformal invariance, this yields $\phi^{0,T}
   \in \mc E_0 (D_T)$. 
   
   We show that $\phi^{h,T}$ is harmonic in $D_T$. In fact, for $z \in D_T$, by definition,
   $$\phi^{h,T} (z) = 2 \pi \int_0^{T} P_{\m D}[\iota[\phi]_t \rho_t](g_t(z))  \, \dd t.$$
   Note that $P_{\m D}[\iota[\phi]_t \rho_t](g_t(\cdot))$ is harmonic in $D_T$ for $t < T$. Therefore, by the characterization of harmonic functions in terms of the mean value property and Fubini's theorem we obtain that $\phi^{h,T}$ is harmonic in $D_T$.
\end{proof}

\begin{thm}\label{thm:bi_isometry}
The operator $\i : \mc{E}_0(\m D) \to L^2 (2\rho)$ is a \textnormal{(}bijective\textnormal{)} isometry. 
\end{thm}
\begin{proof}
We already know that $\varkappa \circ \iota = \Id_{\mc{E}_0(\m D)}$.
To check that $\iota \circ \varkappa = \Id_{L^2(2\rho)}$, it suffices to show that $\Ker (\varkappa) =  \{0\}$ in $L^2 (2\rho)$. Indeed, for all $u \in L^2(2\rho)$, we have $\varkappa [u] = \varkappa \circ \iota \circ \varkappa [u]$. Then $\Ker (\varkappa) = \{0\}$ implies that $u = \iota \circ \varkappa [u]$.

Now let $u \in \Ker (\varkappa)$. 
Fix $T \ge 0$. By Corollary~\ref{cor:ortho_decomp_formula},
$\varkappa [u \1_{S^1 \times [0,T)}] \in \mc E_0 (\m D)$ and $\varkappa[ u \1_{S^1 \times [T,\infty)}]$ give the orthogonal decomposition \eqref{eq:orthogonal}  of $\varkappa[u] = 0$ with respect to $D_T$.  Therefore, for all $w \in D_T$,
$$\varkappa[u \1_{S^1 \times [0,T)}] (w) =  2 \pi \int_0^{T} P_{\m D} [u_t \rho_t] (g_t(w)) \dd t= 0.$$
Taking the derivative in $T$, we obtain that for a.e. $T \ge 0$ the function
$P_{\m D} [u_T \rho_T]$ 
vanishes in $\m D$. (One can choose a dense and countable family $J$ of points in $\m D$, we have for a.e. $T \ge 0$ and $w \in D_T \cap J$, 
$P_{\m D} [u_T \rho_T] (g_T(w))= 0$. Since the latter integral is harmonic thus continuous, it vanishes for a.e. $T \ge 0$ and all $w \in D_T$.)
But this implies that the measure $u_T \rho_T$ on $S^1$ is the zero measure.  Therefore, $u(\t, T) = 0$ for $\rho_T$-a.e. $\t \in S^1$. It follows that
$$\iint_{S^1 \times \m R_+} \abs{u(\t,t)}^2 \,\dd \rho(\t, t) = 0,$$
which proves $\Ker(\varkappa) = \{0\}$ and this concludes the proof.
\end{proof}

We will now discuss the relation between $\mc{E}_0(\m D)$ and functions defined on the leaves of a foliation. We will consider a generalization of the standard trace operator for $W^{1,2}$ on Lipschitz domains to chord-arc domains, see \cite{Jonsson-Wallin} and \cite[Appx.\,A]{VW1}.
Suppose $\phi \in W^{1,2}_{\textrm{loc}}(\m C)$ and $\gamma$ is a chord-arc curve in $\mathbb{C}$. The \emph{Jonsson--Wallin trace} of $\phi$ on $\gamma$ is defined for arclength-a.e. $z \in \gamma$ by the following limit of averages on balls $B(z,r) = \{w: |w-z| < r\}$ 
\begin{equation}\label{def:trace}
\phi|_{\gamma}(z):=\lim_{r \to 0+} \frac{1}{|B(z,r)|} \int_{B(z,r)} \phi \,\dd A. 
\end{equation}

A function $\varphi$ defined arclength-a.e.\ on all leaves of a foliation $(\g_t = \partial D_t)$ is said to have an extension $\phi$ in $\mc E_0(\m D)$ if for all $t$, the Jonsson--Wallin trace of $\phi$ on $\g_t$ (denoted $\phi|_{\g_t}$) coincides with $\varphi$ arclength-a.e. (Here and below we identify $\phi \in \mc E_0(\m D)$ with the function in $W^{1,2}(\m C)$ that is equal to $\phi$ in $\m D$ and $0$ in $\m D^*$.)

\begin{prop}\label{prop:unique_extension}
  If a function $\varphi$ defined arclength-a.e. on each leaf  of a foliation of $\m D$ has an extension to $\mc E_0(\m D)$, then the extension is unique.
\end{prop}

\begin{proof}
  Let $\phi$ be an extension of $\varphi$ in $\mc E_0(\m D)$. 
  For a fixed $T \ge 0$, the orthogonal decomposition $\phi = \phi^{0,T} + \phi^{h,T}$ where $\phi^{0,T} \in \mc E_0 (D_T)$ and $\phi^{h,T} \in \mc E_0(\m D)$ is harmonic in $D_T$ is given by 
  $$\phi^{0,T} = \varkappa [\iota[\phi] \mathbf 1_{ S^1 \times [T,\infty) }] \quad \text{and} \quad \phi^{h,T} = \varkappa [\iota[\phi] \mathbf 1_{ S^1 \times [0,T) }],$$
by Corollary~\ref{cor:ortho_decomp_formula}.  
 We have that $\phi|_{\partial D_T} = \phi^{h,T}|_{\partial D_T}$ arclength-a.e. since they coincide in $\m C \smallsetminus D_T$, see  \cite[Lem.\,A.2]{VW1}. Hence, $\varphi$ determines $\phi^{h,T}|_{D_T}$ since arclength and harmonic measure are mutually absolutely continuous on chord-arc curves. Assume that $\tilde \phi$ also extends $\varphi$. Then we have for all $w \in D_T$, 
 \begin{align*}
     0 &= \phi^{h,T}(w) - \tilde \phi^{h,T}(w) = \varkappa\left[(\iota[\phi] - \iota [\tilde \phi])\mathbf 1_{S^1 \times [0, T)}\right] (w) \\
     &=  2\pi \int_0^{T} P_{\m D} [\iota[\phi-\tilde \phi]_t \rho_t] (g_t(w)) \dd t.
 \end{align*}
 The proof of Theorem~\ref{thm:bi_isometry} shows that $\iota[\phi] = \iota [\tilde \phi]$ in $L^2(2\rho)$. Therefore $\phi = \tilde \phi$ in $\mc E_0 (\m D)$ and this completes the proof.
\end{proof}

\section{Loewner--Kufarev energy}\label{sect:LK-energy}

\subsection{Definitions}
For each measure $\sigma \in \M(S^1)$ we define
\begin{equation}\label{eq:ldp_local}
		L(\sigma) =  
		\frac 12\int_{S^1} \nu'(\t)^2  \,\dd\t  
\end{equation}
if $\dd \sigma(\t)=\nu^2 (\t)\,\dd \t$ with the non-negative square-root density $\nu \in W^{1,2} (S^1)$ and, by convention, $L(\sigma)  =  \infty$ otherwise. With this definition, $L(\sigma)$ is the usual Dirichlet energy of $\nu$ on $S^1$.  
\begin{rem}
   Note that $L (\sigma) = 0$ if and only if $\sigma$ is the uniform measure on $S^1$. 
\end{rem}
Then for $\rho \in \mc N_+$ (see Section~\ref{sect:Loewner--Kufarev-Equation}), we define 
\begin{equation}\label{def:dual-loewner-energy}
S_+ (\rho) =\int_0^{\infty} L(\rho_t)\, \dd t,
\end{equation} 
where $(\rho_t)_{t\ge 0}$ is a disintegration of $\rho$. We call $S_+(\rho)$ the \emph{Loewner--Kufarev energy} of the measure $\rho$. 
When $L(\rho_t) < \infty$, we write $\dd \rho_t = \rho_t (\t) \dd \t = \nu_t^2(\t) \dd \t$.

It is also useful to define
\[
S_{[a,b]} (\rho) =\int_a^b L(\rho_t)\,\dd t.
\]

\subsection{First properties}
We now record a few simple properties that will be used in our proofs.
\begin{lemma}\label{lem:time-reparam}
 Suppose $S_+(\rho) < \infty$ and define the time-changed family of measures $(\tilde{\rho}_{s} = t'(s) \rho_{t (s)} )_{s \ge 0}$, where 
 $$ t(s) = \int_0^s |\tilde{\rho}_u| \dd u.$$ Then $
 S_+(\tilde{\rho}) = S_+(\rho)$. 
 Moreover, if $(g^\rho_t)_{t \ge 0}$ and $(g^{\tilde{\rho}}_s)_{s \ge 0}$ are the corresponding uniformizing Loewner chains, 
 then $g_s^{\tilde{\rho}} = g_{t(s)}^\rho$.
 \end{lemma}

\begin{proof}
We have
   $$S_+ (\tilde{\rho})  = \int_0^{\infty} L(\tilde{\rho}_t) \dd t =  \int_0^{\infty} L(t'(s) \rho_{t} ) \dd s = \int_0^{\infty} L(\rho_{t})  t'(s) \dd s  =  S_+ (\rho). $$
   On the other hand, the Loewner flow driven by $(\tilde{\rho}_s)$ is the solution  $s \mapsto g_s^{\tilde{\rho}} (z)$ to 
   $$ \partial_s g_s^{\tilde{\rho}} (z) = g_s^{\tilde{\rho}} (z)  H[\tilde{\rho}_s](g_s^{\tilde{\rho}}(z))  \quad g_0^{\tilde{\rho}}(z) = z.$$
   Using the definition of $\tilde{\rho}$, we therefore get
   $$ \partial_t g_s^{\tilde{\rho}} (z) = g_s^{\tilde{\rho}} (z)  H[\rho_t](g_s^{\tilde \rho}(z))$$
   and this implies $g_s^{\tilde{\rho}} = g_{t(s)}^\rho$. 
 \end{proof}
 
   One should thus view $\tilde{\rho}$ as a time-reparametrization of $\rho$ and  the solution to the Loewner equation associated with the measure $\tilde \rho$ is a reparametrization of the solution associated to $\rho$.
 This invariance property of the Loewner--Kufarev energy suggests that the energy is intrinsic to the foliation (once we know that finite energy measures generate foliations) and does not depend on the particular time-parametrization. This is further reflected in the energy duality Theorem~\ref{thm:main} expressing $S_+$ in terms of the winding function.

\begin{lemma} \label{lem:foliates}
If $S_+(\rho) <\infty$, then for every $z \in \overline{\m D} \smallsetminus \{0\}$ we have $\tau(z) < \infty$.
\end{lemma}

\begin{proof}
We claim that there exists $\vare > 0$ such that for all $t \ge 0$, $L(\rho_t) < \vare$ implies $\min_{\t} \nu_t^2(\t) > 1/4\pi$. Indeed, suppose without loss in generality that $\nu_t^2(0) \le 1/4\pi$. Then by the mean value theorem and that $\rho_t \in \mc M_1(S^1)$, there is $\t_0$ such that $\nu_t^2(\t_0) \ge  1/2\pi$ and  we obtain easily by interpolating linearly $\nu_t$ on  $[0,\t_0]$  that $L(\rho_t) \ge (3-2\sqrt 2)/16\pi^2 = : \vare$.
Let $E:=\{t \ge 0 : L(\rho_t) > \vare \}$. For $t \notin E$ the Poisson integral of $\rho_t$ is continuous on $\overline{\m D}$ and by the maximum principle, $P_{\m D}[\rho_t] \ge 1/4\pi$. 

Now fix $z \in \ad{\m D} \smallsetminus \{0\}$. Using \eqref{eq:ODE}, for $t < \tau(z)$,
\begin{align*}
    \log |g_t(z)/z| & = \int_0^t \partial_s \Re \log g_s(z) \, \dd s = \int_0^t \Re H_s(g_s(z)) \dd s\\
    &= 2\pi \int_0^t P_{\m D}[\rho_s](g_s(z)) \, \dd s  \ge \big|[0,t] \smallsetminus E\big|/2 
\end{align*}
since the Poisson integral is non-negative. 
Since $\log |g_t (z) / z|\le -\log |z|$ it follows that
$$\tau(z) \le -2 \log |z| + |E| \le -2 \log |z| + S_+(\rho)/\vare <\infty $$
where we also used Markov's inequality and that $S_+(\rho) = \int_0^\infty L(\rho_t) \dd t$.
\end{proof}

\subsection{Examples and energy minimizers}\label{sect:examples}
\begin{figure}
    \centering
    \includegraphics[scale=0.5]{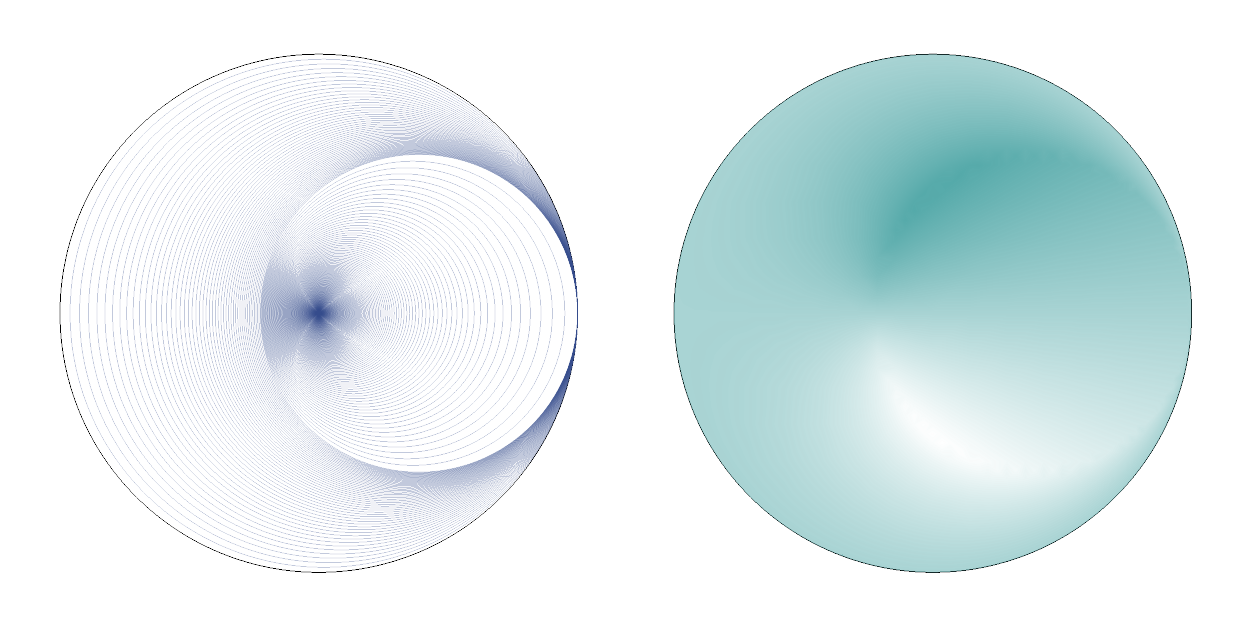}
    \caption{\emph{Left:} Leaves of the foliation corresponding to the measure that equals $\pi^{-1}\sin^2(\t / 2) \,\dd \t \dd t$ for $0 \le t < 1$ and is uniform for $t \ge 1$, drawn at equidistant times. 
    \emph{Right:} Density plot of the winding function corresponding to the same measure, where lighter color represents a larger value. The winding function is harmonic, but non-zero, in the part foliated after time $1$.}
    \label{fig:foliation_example}
\end{figure}
We discussed the simple but important example of a Loewner chain driven by the constant zero-energy measure in the introduction and Section~\ref{sect:Loewner--Kufarev-Equation}. Here we give another example of a time-homogeneous driving measure, this time vanishing at $\t = 0$, 
and minimizing the Loewner--Kufarev energy among all such measures. We compute explicitly the winding function of the corresponding non-injective foliation and verify energy duality by hand. Even for this simple example, while the computation is straightforward it is not entirely trivial.

Let $T > 0$ and $\nu \in W^{1,2} ([0,2\pi])$ with $\nu(0) = \nu (2\pi) =0$ and $\int_{S^1} \nu^2 (\t) \,\dd \t = 1$. Consider the measure $\rho \in \mc N_+$ such that $\dd \rho_t (\t) = \nu^2 (\t) \,\dd \t$ for $t \in [0,T]$, and that is equal to the uniform measure on $S^1$ otherwise. 
Among all such measures, $\nu^2 (\t) \dd\t$ minimizes $L$ if and only if $\nu$ is the first eigenfunction of the Dirichlet Laplacian on $[0,2\pi]$, namely, $ \nu^2(\t) = \sin^2(\t/2)/\pi$.
In this case, we have $L(\rho_t ) \equiv 1/8$ so $S_+(\rho) = \int_0^T L(\rho_t) \dd t= T/8$. 
 The corresponding Herglotz integral can be evaluated explicitly 
 and the Loewner--Kufarev equation in this case is simply
\[
\partial_t f_t(z) =- zf'_t(z) \left(1- z \right) , \quad z \in \mathbb{D},
\]
see, e.g.,~\cite{sola} where solutions to Loewner equations of this form were studied. We have
\[
f_t(z) = \frac{e^{-t}z}{1-(1-e^{-t})z}.
\]
The leaf $\partial D_t$ is a circle of radius $1/(2-e^{-t})$ centered at $(1 - e^{-t})/(2-e^{-t})$ and the hull $K_t$ at time $t$ is a 
crescent-shaped compact set, see Figure~\ref{fig:foliation_example}. In particular, $z=0$ and $z=1$ are fixed points of the evolution.
We now compute the corresponding winding function. Suppose $z \in K_T$ and let $t = \tau(z)$. Then
\[
\varphi(z) = \arg \frac{z g'_t(z)}{g_t(z)} = -\arg(z(1-e^{-t})+e^{-t})=\arg \frac{z}{(z-1)^2} + \pi.
\]
Moreover, for $z \in D_T$, we have
\[
\varphi(z) =  -\arg(z(1-e^{-T})+e^{-T}),
\]
which is harmonic in $D_T$. From these formulas one can verify directly that $\mc D_{\m D}(\varphi) = 
2T = 16 S_+ (\rho)$, which is the statement of energy duality in this case. We note that the computation of $\mc D_{\m D}(\varphi)$ is slightly technical as $\nabla \varphi$ has a singularity at $1$ (that 
is $L^2$-integrable).

\begin{rem}Conjugating $f_t$ by $z \mapsto z^m$, it is also possible carry out similar explicit computations for measures of the form $\pi^{-1} \sin^2(m\t/2) \dd \t \dd t$ that correspond to Laplace eigenfunctions on $S^1$ corresponding to higher eigenvalues. 
\end{rem}

\section{Weak energy duality and foliation by Weil--Petersson quasicircles}\label{sect:weak-WP}
In this section we prove energy duality in the disk, for measures satisfying a strong smoothness assumption.
We then use this result and an approximation argument to prove that any $\rho$ with $S_+(\rho)<\infty$, produces interfaces that are Weil--Petersson quasicircles which form a foliation of $\ad{\m D} \smallsetminus \{0\}$.

\subsection{Weak energy duality}
Let 
$\mc N_+^\infty$ denote the set of $\rho \in \mc N_+$ such that: 
\begin{enumerate}[itemsep=-2pt]
    \item There exists $T < \infty$, such that for $t > T$, $\rho_t = \dd \t /2\pi$ is the uniform measure on $S^1$.
    \item The mapping $t \mapsto \rho_t$ is piecewise constant (on finitely many time intervals).
    \item For all $t \ge 0$, $\rho_t$ has a $C^\infty$-smooth and strictly positive density with respect to $\dd\t$. 
\end{enumerate}

\begin{lemma}\label{lem:constant_rho}\label{cor:smooth_foliation}
Any $\rho\in \mc N^\infty_+$ generates a foliation of $\ad{\m D}\smallsetminus\{0\}$ in which each leaf is a smooth Jordan curve. The corresponding winding function $\varphi$ is continuous and piecewise smooth on $\overline{\m D}$.
\end{lemma}

\begin{proof}
We assume first that $\rho \in \mathcal{N}_+^\infty$ and that $ \rho_t$ is constant on $[0,T]$. Let $H : =H[\rho_t]$ which is constant for $t \le T$.  The family $(f_t)_{t \in [0,T]}$ solves the equation $\partial_t f_t(z) = -z f'_t(z) H(z)$ with $f_0(z) = z$ and forms a semigroup of conformal maps with fixed point $0$. (See Lemma~\ref{lem:domain_markov}.) The solution therefore enjoys the following representation:  there exists a starlike\footnote{Starlike here means that $D$ contains $0$ and for every $z \in D$, the line segment $[0,z] \subset D$.} domain $D$ and a conformal map $\psi:\m D \to D$ fixing $0$ such that 
\[
f_t(z) = \psi^{-1}(e^{-t}\psi(z)), \qquad \psi(z) = z\exp\left(\int_0^z \frac{1}{w}\left(\frac{1}{H(w)} -1\right) \dd w\right).
\]
See, e.g., \cite[Sec.\,3]{siskakis} and the references therein. Since $H$ is the Herglotz integral of a smooth, positive function, the maximum principle implies that it is non-zero on $\overline{\m D}$ and the smoothness implies that all derivatives exist and extend continuously to $\overline{\m D}$. (See, e.g., \cite[Cor.\,II.3.3]{GM}.) It follows that $\psi$ is smooth on $\overline{\m D}$. The foliation generated by $\rho$ is given by the family $(\g_t = \psi^{-1} (e^{-t} \psi(S_1)))$ for $t \le T$, and for $t > T$, the leaves $\g_t$ are equipotential curves in $D_T$ by Lemma~\ref{lem:harmonic}.
We now verify the smoothness of the winding function $z \mapsto \varphi(z)$ in $K_T$. For this, consider $(\t,t) \in S^1\times [0,T]$ and set $z(\t,t) = f_t (e^{i\t})$. Then $\varphi(z(\t,t)) = -\arg e^{i\t} f'_t(e^{i\t})/f_t(e^{i\t})$. 
We have 
 \begin{align*}
 \partial_\t  z & = i e^{i\t} f_t '(e^{i\t})  \\
 \partial_t  z & = \partial_t f_t(re^{i\t}) = e^{i\t} f_t '(e^{i\t}) H(e^{i\t}).
 \end{align*}
 Since $\Re H > 0$ the Jacobian is non-zero. Further differentiation using the Loewner equation shows that $z(\t,t)$ is smooth and by the inverse function theorem, the inverse is also smooth in $K_T$, and so is $\varphi$. 
 Since $\varphi|_{D_T} = \vartheta (g_T)$ by Lemma~\ref{lem:harmonic} and is continuous on $\ad{D_T}$, we obtain that $\varphi$ is continuous on $\ad{\m D}$ and smooth in $\ad{\m D} \smallsetminus \g_T$.
 
 Using Lemma~\ref{lem:domain_markov} and  the fact the smoothness is preserved under composition of smooth conformal mappings, the statement immediately generalizes to arbitrary $\rho \in \mc N^\infty_+$.
\end{proof}

\begin{lemma}\label{lem:Loewner-formulae}
 If $\rho \in \mc N_+$ and 
  $z \in \m D$, then
$- \partial_t \vartheta[f_t](z)  \mid_{t = 0} = \Im zH_0'(z).$
 \end{lemma}
\begin{proof}
Using \eqref{eq:loewner-pde}, we have for $t \ge 0$,
\[
\partial_t \log f_t(z) = -z\frac{f'_t(z)}{f_t(z)} H_t(z).
\]
and
\[
\partial_t \log f_t'(z) = -\frac{\partial_z (zf'_t(z) H_t(z))}{f_t'(z)}
=-\left(1+z \frac{f''_t(z)}{f'_t(z)}\right) H_t(z) - z H_t'(z).
\]
Hence,
\[\partial_t \log \frac{z f'_t(z)}{f_t(z)}  =  -\left(1+z \frac{f''_t(z)}{f'_t(z)}-z\frac{f'_t(z)}{f_t(z)}\right) H_t(z) - z H_t'(z). \]
Since $f_0(z)=z$, we have $f''_0(z) = 0$ and the claimed expression when evaluated at $t = 0$ follows by taking the imaginary part.
\end{proof}  

\begin{prop}[Weak energy duality]\label{prop:weak_duality}
Suppose $\rho \in \mc N_+^\infty$ and let $\varphi$ be the associated winding function. Then $\mc D_{\m D} (\varphi) = 16\, S_+(\rho)$. 
\end{prop}

\begin{proof}
By Lemma~\ref{lem:constant_rho} the winding function $\varphi$ is continuous and piecewise smooth on $\ad{\m D}$. Moreover, $\varphi|_{S^1} = 0$ and therefore $\varphi \in \mc E_0(\m D)$ \cite[Thm.\,9.17]{brezis} and we may apply the foliation disintegration isometry to $\varphi$. We claim that for $t \in [0,T]$, $\iota [\varphi
](\t,t) = - 2 \nu'(\t)/\nu(\t)$. Given this, we can apply Theorem~\ref{thm:bi_isometry} 
to conclude the proof. To prove the claim, set $\varphi_t = \varphi \circ f_t$ and recall that by definition
  \begin{align}\label{eq:oct17.1}
  \iota [\varphi] (\t,t) & = \frac{1}{2\pi}\int_{\m D} \D \varphi_t (z) P_{\m D} (z, e^{i\t}) \,\dd  A(z)  \nonumber \\
  &
  = \frac{1}{2\pi}\int_{\m D} \D  \varphi_t (z) \partial_{n(e^{i\t})}G_{\m D} (z, e^{i\t}) \,\dd  A(z).
  \end{align}
  On the other hand, by Lemma~\ref{lem:constant_rho},  $\Delta \varphi_t$ is smooth,  so the last term in \eqref{eq:oct17.1} equals
  \begin{align*}
   \frac{1}{2\pi }\partial_{n} \int_{\m D} \D  \varphi_t (z) G_{\m D} (z, e^{i\t}) \,\dd  A(z)  = \partial_{n} \varphi_t^0(e^{i \t}).
  \end{align*}
   Recall from Lemma~\ref{lem:harmonic} and Lemma~\ref{lem:zero_trace_varphi} that $z \mapsto \varphi_t^0(z)$ is the winding function generated by the measures $s \mapsto \rho_{t+s}$ for $s \ge 0$.  Hence, it suffices to consider $t=0$ and show that for a.e. $\t \in [0, 2\pi)$, 
   \begin{equation}\label{eq:normal_d_varphi}
   \partial_n \varphi_0^0(e^{i\t}) = \partial_n \varphi_0(e^{i\t})= -2\nu_0'(\t)/\nu_0(\t).  
   \end{equation}

  We know that $\partial D_t$ is $C^\infty$ for all $t$, and all (complex) derivatives of $f_t$ are continuous on $\overline{\m D}$ and since the Herglotz integral $H_t$ is also continuous on $\overline{\m D}$, so is $\partial_t f_t(z)= -zf'_t(z) H_t(z)$. It follows that the normal velocity (with respect to $\partial D_t$) of the interface at the point $f_t(e^{i \t})$ can be written 
  \begin{align}\label{eq:normal}
  \begin{split}
      \operatorname{vel}_n(t,\t) & = - \mathrm{Re} \left( \frac{\partial_t f_t(e^{i\t}) \overline{e^{i\t} f'_t(e^{i\t})}}{|f'_t(e^{i\t})|}\right) \\
&  =   \mathrm{Re} \left( \frac{e^{i\t} f_t'(e^{i\t}) H_t(e^{i\t})  \overline{e^{i\t} f'_t(e^{i\t})}}{|f'_t(e^{i\t})|}\right) \\
& = |f'_t(e^{i\t})| \mathrm{Re} \, H_t(e^{i\t}) = 2\pi |f'_t(e^{i \t})| \nu_t( \t)^2 > 0.
  \end{split}
\end{align}
  In particular, at $t = 0$ the normal velocity of the interface at $e^{i\t}$ equals  $2 \pi \nu_0(\t)^2$. On the other hand, using the chain rule \eqref{eq:theta_chain}, 
 \begin{align*}
  2 \pi \nu_0(\t)^2 \partial_n \varphi_0(e^{i\t})& = \partial_t [\varphi_0 (f_t (e^{i\t}))]_{t = 0} =  \partial_t [\vartheta [g_t] ( f_t (e^{i\t})  ) ] |_{t = 0} \\
  &= -\partial_t \left(\vartheta [f_t] (e^{i\t}) \right)|_{t = 0}. 
 \end{align*}
 It follows from Lemma~\ref{lem:Loewner-formulae}  that 
 $$-\partial_t \left(\vartheta [f_t] (e^{i\t}) \right)|_{t = 0} = \Im( e^{i\t} H_0'(e^{i\t}))  = -4\pi \nu_0(\t) \nu_0'(\t).$$
 The last identity is not hard to verify by hand, using integration by parts and the smoothness of $\nu_0$, see also Lemma~\ref{lem:H2W12}. This proves \eqref{eq:normal_d_varphi} and concludes the proof.
 \end{proof}

We will use the following special case of the generalized Grunsky inequality, which provides a useful bound to control $\mc D_{\m D}(\arg f'_t)$ in terms of $\mc D_{\m D} (\vartheta[f_t])$ (see Corollary~\ref{cor:smooth_arg_f_bound}). 
\begin{lemma}[{See \cite[P.\,70-71]{TT06}}]\label{lem:Grunsky_inequality}
  Suppose that $f: \m D \to  \m C$ and $h : \m D^* \to  \m C$ are univalent functions on $\m D$ and $\m D^*$ such that $f(0) = 0$ and $h(\infty) = \infty$, and $f(\m D) \cap h(\m D^*) = \emptyset$. Then we have 
  $$\int_{\m D} \abs{ \frac{f'(z)}{f (z)}  - \frac{1}{z} }^2  \dd A (z)+  \int_{\m D^*} \abs{\frac{h'(z)}{h(z)}  -\frac{1}{z} }^2  \dd A (z)  \le 2 \pi \log \abs{\frac{ h'(\infty)}{f'(0)}}. $$
  Equality holds if the omitted set   $\m C \smallsetminus \{f (\m D) \cup h (\m D^*)\}$ has Lebesgue measure zero.
\end{lemma}

\begin{lemma} \label{lem:grunsky_f}
Let $\rho \in \mc N_+$ and $(f_t)_{t \ge 0}$ is the associated Loewner chain. Then
  \begin{equation}\label{eq:Grunsky_f}
 \mc D_{\m D} \left(\arg [f_t(z)/z]\right) = \frac{1}{\pi} \int_{\m D} \abs{\frac{f_t'(z)}{f_t (z) }- \frac{1}{z}}^2 d A (z) \le 2 t.
  \end{equation}
\end{lemma}
\begin{proof}
Since $D_t$ and $\m D^*$ are disjoint, and since $f_t'(0)=e^{-t}$, Lemma~\ref{lem:Grunsky_inequality} applied to the pair $(f_t,\Id_{\m D^*})$ shows that
$$\int_{\m D} \abs{\frac{f_t'(z)}{f_t (z) }- \frac{1}{z}}^2 d A (z) \le - 2 \pi \log |f_t ' (0)| = 2 \pi t$$
as claimed. 
\end{proof}

\begin{cor}\label{cor:smooth_arg_f_bound}
Suppose $\rho \in \mc N_+^\infty$. Then for all $t \ge 0$,
$ \mc D_{\m D}(\arg f'_t) \le 32 S_{[0,t]}(\rho) + 4 t.  $
\end{cor}
\begin{proof}
Fix $t \ge 0$. Without loss of generality, we assume that $\rho_s$ is the uniform measure for all $s > t$.
Combining Proposition~\ref{prop:weak_duality}, Lemma~\ref{lem:grunsky_f}, and the Cauchy-Schwarz inequality,  we obtain
\begin{align*}
\mc D_{\m D}(\arg f'_t) & = \mc D_{\m D}(\vartheta [f_t] + \arg [f_t(z)/z])  \le 2 \mc D_{D_t}(\vartheta [g_t]) + 2 \mc D_{\m D}(\arg [f_t(z)/z]) \\
& \le 2 \mc D_{\m D}(\varphi) + 4 t = 32 S_{[0,t]}(\rho) + 4 t
\end{align*} 
as claimed.
\end{proof}

\subsection{Foliation by Weil--Petersson quasicircles} \label{A-priori-WP}
This section proves that the interfaces generated by a finite energy measure are  
Weil--Petersson quasicircles that together form a foliation (Corollary~\ref{cor:S_finite_foliation_WP}).
 We show first that each leaf $\partial D_t$ is a Weil--Petersson quasicircle. For this, we need the following quantitative upper bound for $\mc D_{\m D}(\arg f'_t)$ that depends only on $S_+(\rho)$ and $t$.

 \begin{lemma}\label{lem:arg-ft-energy}
 If $S_+(\rho)<\infty$, then for all $t \ge 0$, 
 $
 \mc D_{\m D}(\arg f'_t) \le 32 S_{[0,t]}(\rho) + 4 t.  
 $
 \end{lemma}
 \begin{proof}
 Fix $t > 0$. We will approximate $\rho$ by a sequence of  measures $\rho^{(k)} \in \mc N_+^\infty$ which converges to $\rho$ weakly, and such that $S_{[0,t]} (\rho^{(k)}) \le S_{[0,t]} (\rho)$.
 The corresponding sequence of conformal maps $f^{(k)}_t$ then converges uniformly on compacts to $f_t$ by Lemma~\ref{lem:cont-bij-loewner-transf}, which, by Corollary~\ref{cor:smooth_arg_f_bound}, implies 
 $$\mc D_{\m D} (\arg f_t') \le \liminf_{k \to \infty} \mc D_{\m D} (\arg (f^{(k)}_t)') \le 32\,S_{[0,t]} (\rho^{(k)}) + 4t \le 32\,S_{[0,t]} (\rho) + 4t.$$

We construct the approximation in two steps. We first let $\rho^{(n)}$ be the ``time-averaged'' measure, which is piecewise constant on dyadic time intervals $[jt/ 2^{n}, (j+1)t/2^{n})$ and is defined by 
$$\rho^{(n)}_s : = \sigma^j : = \frac{2^n}{t} \int_{jt/2^n}^{(j+1)t/2^n} \rho_r \ \dd r \in \mc M_1 (S^1), \quad \forall s \in \left[\frac{jt}{2^{n}}, \frac{(j+1)t}{2^{n}}\right).$$
To see that $S_{[0,t]}(\rho^{(n)}) \le S_{[0,t]}(\rho)$,
the key observation is that the map $\sigma \mapsto L(\sigma)$ from
$\mc M_1 (S^1)$ to $[0, \infty]$ is convex \cite{DonVar1975} (see also \cite[Thm.\,3.4]{APW}).
 We can therefore apply Jensen's inequality:
	\begin{align*}
	S_{[0,t]}(\rho^{(n)}) & = \frac{t}{2^n} \sum_{j= 0}^{2^n -1} L\left( \frac{2^n}{t} \int_{jt/2^n}^{(j+1)t/2^n} \rho_r \ \dd r  \right) \\
	& \leq \sum_{j= 0}^{2^n -1} \int_{jt/2^n}^{(j+1)t/2^n} L\left(  \rho_t  \right) \ dt = S_{[0,t]}(\rho). 
	\end{align*}
It is clear that $\rho^{(n)}$ restricted to $S^1 \times [0,t]$ converges weakly to $\rho$: integrating against a continuous function $u$ on $S^1 \times [0,t]$, which is then uniformly continuous, $\int u \,\dd\rho^{(n)}$ converges to $\int u \,\dd \rho$.

Since $\rho^{(n)}$ might not be strictly positive and smooth, the second step is to approximate each $\sigma = \sigma^j$ of the $2^n$ measures in $\mc M_1(S^1)$ on dyadics by 
$$\dd  \sigma_{r} (\t):= \frac{\dd \t }{2\pi}\int_{\xi \in S^1}  P_{\m D}(r e^{i\t} , e^{i \xi})   \dd \sigma (\xi)$$ 
where $r <1$, which has positive and smooth density with respect to $\dd \t$.

Now we prove that $L(\s_r) \le L(\s)$. Let $f \in C^0(S^1)$,
\begin{align*}
    \sigma_r (f) &: = \int_{S^1} f(\t) \dd \sigma_r (\t) = \frac{1}{2\pi} \int_{\t \in S^1} f(\t) \int_{\xi \in S^1} P_{\m D} (r e^{i\t}, e^{i \xi}) \,\dd \sigma (\xi) \dd \t \\
    &=\frac{1}{2\pi}  \int_{\xi \in S^1} \int_{\t \in S^1} f(\t) P_{\m D} (r e^{i(\t -\xi)}, 1) \, \dd \t \,\dd \sigma (\xi)  \\
    & = \frac{1}{2\pi}  \int_{\xi \in S^1} \int_{\eta \in S^1} f(\eta +\xi) P_{\m D} (r e^{i\eta}, 1) \,\dd \eta \,\dd \sigma (\xi) \\
    & = \frac{1}{2\pi}  \int_{\eta \in S^1} P_{\m D} (r e^{i\eta}, 1) \int_{\xi \in S^1} f(\eta +\xi)  \,\dd \sigma (\xi) \dd \eta\\
    & = \frac{1}{2\pi}  \int_{\eta \in S^1} P_{\m D} (r,  e^{i\eta})  \,\eta_*\sigma (f) \dd \eta,
\end{align*}
 where $\eta_* \sigma$ is the pull-back measure $\sigma$ by the rotation $\xi \mapsto \xi +\eta $. In particular, $L(\sigma) = L(\eta_* \sigma)$.
 We obtain 
 $$\sigma_r = \frac{1}{2\pi} \int_{\eta \in S^1} P_{\m D} (r,  e^{i\eta}) \eta_* \sigma \, \dd\eta.$$
 Since $\frac{1}{2\pi} \int_{\eta \in S^1} P_{\m D} (r,  e^{i\eta}) \dd\eta = 1$, $\sigma_r$ is a probability measure. 
Using convexity and Jensen's inequality once again,
 $$L(\sigma_r) \le \frac{1}{2\pi} \int_{\eta \in S^1} P_{\m D} (r,  e^{i\eta}) L( \eta_* \sigma) \dd\eta = L(\sigma).$$
 Finally, note that since $f$ is continuous, $\eta \mapsto \eta_*\sigma (f)$ is continuous on $S^1$ and equal to $\sigma(f)$ for $\eta = 0$. Therefore, since $\sigma_r(f)$ is the Poisson integral of $\eta \mapsto \eta_*\sigma (f)$ evaluated at $r$, it follows that $\lim_{r\to 1}\sigma_r (f) = \sigma(f)$. Hence $\sigma_r$ converges to $\sigma$ weakly and this completes the proof.
 \end{proof}
 
  We would now like to use Lemma~\ref{lem:arg-ft-energy} to conclude that $\partial D_t$ is a Weil--Petersson quasicircle. However, we cannot directly apply Lemma~\ref{thm_TT_equiv_T01} since we do not know \emph{a priori} that $\partial D_t$ is a Jordan curve. 
 In fact, it is not hard to construct an example of a simply connected domain for which the boundary is self-touching, while $\mc D_{\m D}(\log f') < \infty$ where $f$ is a conformal map onto the domain.
 In the present case, however, we can use the fact that $\partial D_t$ arises from Loewner evolution: we will consider the evolution for a small time interval and use estimates on the Schwarzian combined with a result of Ahlfors-Weill. We can then complete the proof using Lemma~\ref{lem:domain_markov}. 
 
 For a function $f$ holomorphic at $z$ such that $f'(z) \neq 0$, recall that the Schwarzian derivative of $f$ at $z$ is defined by
$$\mc Sf(z) = \frac{f'''(z)}{f'(z)} - \frac{3}{2} \left( \frac{f''(z)}{f'(z)} \right)^2 = \left(\frac{f''(z)}{f'(z)}\right)' - \frac{1}{2}\left( \frac{f''(z)}{f'(z)} \right)^2.  $$
\begin{lemma}[See {\cite[Lem.\,I.2.1, Lem.\,II.1.3 and Lem.\,II.1.5]{TT06}}] \label{lem:TT_bounded_S}
There exists $\d > 0$ such that 
if $\mc{D}_{\m D}(\log f') < \d$, then $f$ is univalent 
and $f (\m D)$ is a Jordan domain bounded by a 
Weil--Petersson quasicircle.
\end{lemma}

\begin{proof}
By Lemma~\ref{thm_TT_equiv_T01} and the assumption that $\mc{D}_{\m D}(\log f') < \d < \infty$, it is enough to prove that $f(\m D)$ is a Jordan domain. We will prove more and show that $f(\m D)$ is a quasidisk. By a theorem of Ahlfors-Weill (see \cite[Cor.\,5.24]{Pommerenke_boundary}), to show that $f(\m D)$ is a quasidisk it suffices to show that for small enough $\d$, 
\begin{equation}\label{jun12.1}
\sup_{z \in \m D} (1 - |z|^2)^2 \abs{\mc Sf  (z)} < 2.
\end{equation}
We will estimate the left-hand side of \eqref{jun12.1} in terms of $\mc{D}_{\m D}(\log f')$.
The required estimate is a combination of two bounds. First we claim that if $f$ is holomorphic on $\m D$ and $f' \neq 0$, then,
$$ \int_{\m D} |\mc S f (z)|^2  (1 - |z|^2)^2 \dd A(z) \le \pi \mc{D}_{\m D}(\log f')+ \frac{\pi}{8 } \left(\mc{D}_{\m D}(\log f')\right)^2.$$
Indeed, this follows from the proof of \cite[Lem.\,II.1.5]{TT06}  where the bound depends on a constant from \cite[Lem.\,II.1.3]{TT06}. 
On the other hand, it follows directly from 
\cite[Lem.\,I.2.1]{TT06} that
$$\sup_{z \in \m D} (1 - |z|^2)^2 \abs{\mc Sf (z)} \le \sqrt{\frac{12}{\pi}} \left(\int_{\m D} |\mc S f(z)|^2  (1 - |z|^2)^2 \dd A(z)\right)^{1/2}.$$
 Combining these bounds we see that \eqref{jun12.1} indeed holds provided $\delta$ is chosen sufficiently small. 
\end{proof}

\begin{prop}\label{prop:WP-QC-final}
If $S_+(\rho) <\infty$, then for all $t \ge 0$, $\partial D_t$ is a Weil--Petersson quasicircle.
\end{prop}
\begin{proof} Let $\d > 0$ be the small constant from Lemma~\ref{lem:TT_bounded_S}. 
Pick $t_1 > 0$ such that
$32 S_{[0,t_1]} (\rho)+ 4 t_1  < \d$.
For $t \le t_1$,
Lemma~\ref{lem:arg-ft-energy} and Lemma~\ref{lem:TT_bounded_S} show that $\partial D_t$ is a Weil--Petersson quasicircle.

  Now we consider the general case $t \ge 0$.  Lemma~\ref{lem:arg-ft-energy} shows that $\mc D_{\m D} (\arg f_t')$ 
  is finite, so it suffices to prove that $D_t$ is a Jordan domain to conclude that $\partial D_t$ is a Weil--Petersson quasicircle by Lemma~\ref{thm_TT_equiv_T01}. For this, let
   $0=t_0 <  t_1 < t_2< \ldots$ be a sequence tending to $\infty$ such that $32 S_{[t_j, t_{j+1}]} (\rho) + 4(t_{j+1} - t_j)  < \d$.
   We show by induction that $f_{t}$ is a homeomorphism from $\ad{\m D}$ onto $\ad D_{t}$ for all $t \le t_j$.
   We have already proved this in the case $j = 1$. Assume it is true for all $t \le t_j$. Using Lemma~\ref{lem:domain_markov} and the choices of $\delta$ and $|t_{j+1}-t_j|$ we obtain that $g_{t_j} (\partial D_t)$ is a Weil--Petersson quasicircle for each $t_j \le t \le t_{j+1}$, in particular a Jordan curve. Since by assumption $f_{t_j}$ is a homeomorphism of $\ad{\m D}$, $\partial D_t = f_{t_j} \circ g_{t_j} (\partial D_t)$ is also a Jordan curve and this completes the induction.
   \end{proof}

   We now show that $t \mapsto \g_t$ is continuous, where $\g_t$ is viewed as a parametrized curve $S^1 \to \g_t$ by $\theta \mapsto f_t (e^{i \theta})$. We will need the following lemma.  
   
\begin{lemma}\label{lem:derivative-estimate}
  Suppose $D$ is a simply connected domain containing $0$ and let $f: \m D \to D$ be a conformal map with $f(0)=0$
  and assume that $\mc D_{\m D}(\log f') < \infty$. There exists a constant $C < \infty$ depending only on $\mc D_{\m D}(\log f')$ such that
   \begin{equation}\label{derivative}
   |f'(r e^{i\t})| \le C |f'(0)| \exp \sqrt{C\log (1-r)^{-1}} .
  \end{equation}
\end{lemma}
\begin{rem}
Note that it is not assumed that $D$ is a Jordan domain. The estimate \eqref{derivative} easily implies that $f$ is continuous on $\overline{\m D}$.
\end{rem}
\begin{proof}
We may assume $f'(0) = 1$.  
Any $\phi$ that is holomorphic in $\m D$ and such that $\mc{D}_{\m D}(\phi) < \infty$ has non-tangential limits a.e. on $S^1$, and writing $\phi$ for that function as well we have the following weak-type estimate: there exist universal constants $C_1, C_2$ such that 
  $ \left|\{\t \in S^1 : |\phi| > \lambda\} \right| \le C_1 e^{-C_2\lambda^2/(|\phi(0)|^2+\mc{D}_{\m D}(\phi))}.$
See \cite[Cor.\,3.3.2]{primer} for a proof.
We apply this estimate with $\phi = \log f'$ which is normalized so that $\log f'(0) = 0$. Hence the upper bound in the weak-type estimate depends only on $\mc{D}_{\m D}(\log f')$ and 
implies there exist $C_1,C_2$ depending only on $\mc{D}_{\m D}(\log f')$ (but which, in what follows, are allowed to change from line to line) such that    \[
     \int_0^{2\pi} \exp\left( C_1| \log|f'(e^{i\t})||^2 \right) \,\dd\t \le C_2.
     \]
    Since $\log|f'(z)|$ is harmonic, $\exp\left(C_1|\log|f'(z)||^2\right)$ is subharmonic and it follows that for $0 \le r < 1$
    \[
    \int_0^{2\pi} \exp \left( C_1| \log|f'(re^{i\t})||^2 \right) \,\dd\t \le C_2.
    \]
    Therefore, if $r_n = 1-2^{-n}$ and $z_{k,n} = r_n e^{i2\pi k/2^n}$, using Koebe's distortion theorem (see, e.g., \cite[Ch.\,2.3]{Duren1983}), there is a universal constant $C_3 < \infty$ such that, taking $C_2$ larger if necessary,
    \begin{align*}
        \frac{1}{2^n}  \sum_{k=1}^{2^n}  \exp\left( C_1 (|\log|f'(z_{k,n})|)^2 \right)  &  \le   \int_0^{2\pi} \exp\left( C_1 (|\log|f'(r_ne^{i\t})|| + C_3)^2 \right) \,\dd\t
         \le C_2.
    \end{align*}
    Whence,
    \[
    |f'(z_{k,n})| \le \exp \sqrt{ C_1^{-1}\log(C_2 2^n)}.
    \]
    Using the distortion theorem again we deduce
    \[
    |f'(r e^{i\t})| \le C \exp\sqrt{ C \log  (1-r)^{-1} },
    \]
    where $C$ depends only on $\mc{D}_{\m D}(\log f')$,
    as claimed.
    \end{proof}

If $(f_t)_{t\ge 0}$ is the Loewner chain generated by a finite energy measure $\rho$ then by Lemma~\ref{lem:arg-ft-energy} we have
    $\mc D_{\m D}(\log|f'_t|) \le 32S_+(\rho) + 4T$ for all $t \le T$. Therefore, by  Lemma~\ref{lem:derivative-estimate},  if $\sigma (x) :  =  C \exp \sqrt {C\log (x)}$ 
    where $C$ depends  only on $S_+(\rho)$ and $T$, and
    \begin{align}\label{derivative2}
    |f'_t(re^{i\t})| \le |f'_t(0)| \sigma(1/(1-r)) \le \sigma(1/(1-r)).
    \end{align}
In the rest of the section we use the conformal parametrization of the leaves, namely $\gamma_t(\t) := f_t(e^{i \t})$.

\begin{prop}\label{prop:Lipschitz}
   Suppose $S_+(\rho)<\infty$. Then the function $t \mapsto (\gamma_t: S^1 \to \m C):  [0,\infty) \to (C^0, \|\cdot \|_{\infty})$ is continuous. 
 \end{prop}

\begin{proof}
Throughout $C$ denotes a constant whose value is allowed to change from line to line.
  Fix any $T < \infty$. Then if $0\le s \le t \le T$, for any $0 \le r < 1$,
  \begin{align}\label{feb19.0}
    |\gamma_s(\t) - \gamma_t(\t)|  \le   \, |\gamma_s(\t) - f_s(re^{i\t})|  
     +|\gamma_t(\t) - f_t(re^{i\t})|  +   |f_s(re^{i\t}) - f_t(re^{i\t})|  .
\end{align}
By integrating \eqref{derivative2} we have 
\begin{equation}\label{feb19.1} 
|\gamma_s(\t) - f_s(re^{i\t})| + |\gamma_t(\t) - f_t(re^{i\t})| \le C(1-r)\sigma(1/(1-r)),
\end{equation}
 where $\sigma$ is a subpower function depending only on $S_+(\rho)$ and $T$.
Using the Loewner equation and once again \eqref{derivative}, we have for $z=re^{i\t}$,
\[
     |f_s(re^{i \t}) - f_t(re^{i \t})|  = | \int_s^t zf'_u(z)H_u(z) \dd u|  \le C \sigma(1/(1-r)) \int_s^t|H_u(z)| \dd u.
     \]
Since for a.e. $t$, $\dd \rho_t(\t) = \nu_t(\t)^2 \dd \t$, we can estimate using the Cauchy-Schwarz inequality\[
|\nu_t(\t_1)^2 - \nu_t(\t_2)^2|  \le 2\|\nu_t\|_{\infty} \|\nu_t'\|_{L^2}|\t_1-\t_2|^{1/2} \le M_t|\t_1-\t_2|^{1/2},
 \]
 where 
 \[M_t := 2( 1/\sqrt{2\pi} + \sqrt{2 \pi} \|\nu_t'\|_{L^2} )\|\nu_t'\|_{L^2}.\]
 Indeed, since $\nu_t$ is continuous and $\int \nu_t^2 = 1,$ we can assume that $\nu_t(0) < 1/\sqrt{2\pi}$. 
 Then $|\nu_t(\t) - \nu_t(0)| \le \sqrt{\t}\|\nu_t'\|_{L^2} \le \sqrt{2 \pi} \|\nu_t'\|_{L^2}$, so $\|\nu_t\|_{\infty} \le 1/\sqrt{2\pi} + \sqrt{2 \pi} \|\nu_t'\|_{L^2}.$ We claim that for a.e. $t$, $\sup_{z \in \overline{\m D}}|H_t(z)| \le C M_t$. Since $\Re H_t(z) = 2\pi P_{\m D}[\rho_t](z)$, we have $|H_t'(r e^{i\t})| \le C M_t(1-r)^{-1/2}$. (See, e.g., \cite[Thm.\,5.8 and 5.1]{Duren_Hardy}.) So by integration, the claim follows. Consequently $\int_s^t|H_u(z)| \dd u \le C \int_{s}^t M_u \dd u$. Hence
\begin{equation}\label{feb19.2} 
|f_s(re^{i\t}) - f_t(re^{i\t})| \le C \sigma(1/(1-r))\int_{s}^t M_u \dd u.
\end{equation}
Now we choose $r$ so that $1-r = \int_s^tM_u \dd u \wedge 1$ and plug in \eqref{feb19.1} and \eqref{feb19.2} into \eqref{feb19.0} to conclude that
\[
 \sup_{\t \in [0,2\pi)} |\gamma_s(\t) - \gamma_t(\t)| =o(1)
\]
as $
|t-s| \to 0$. Since $T < \infty$ was arbitrary, this completes the proof. 
\end{proof}

\begin{rem}
Under the stronger assumption that $L(\rho_t)$ is uniformly bounded, the proof of Lemma~\ref{lem:derivative-estimate} shows that the mapping $t \mapsto (\gamma_t: S^1 \to \m C):  [0,\infty) \to (C^0, \|\cdot \|_{\infty})$ is weakly Lipschitz continuous, that is, it admits a modulus of continuity of the form $|\cdot|\sigma(1/|\cdot|)$, where $\sigma$ is a subpower function. 
\end{rem}

We are now ready to prove the main result of this section.

\begin{cor} \label{cor:S_finite_foliation_WP}
If $S_+(\rho) < \infty$, then $\rho$ generates a foliation of $\ad{\m D} \smallsetminus\{0\}$  by Weil--Petersson quasicircles.
\end{cor}

\begin{proof}
  Proposition~\ref{prop:WP-QC-final} shows $\g_t$ is a  Weil--Petersson quasicircle for each $t$ and $t\mapsto \g_t$ is continuous in the supremum norm for the conformal parametrization by Proposition~\ref{prop:Lipschitz}. Lemma~\ref{lem:foliates} shows that $\tau(z) < \infty$ for all $z \in \ad{\m D} \smallsetminus \{0\}$ and this completes the proof.
  \end{proof}

The next result is not used in the rest of the paper, but as it is an interesting and immediate consequence of Lemma~\ref{lem:derivative-estimate}, we choose to state it here. 
\begin{cor}
Suppose $D$ is a simply connected domain containing $0$ and let $f: \m D \to D$ be a conformal map with $f(0)=0$
  and assume that $\mc D_{\m D}(\log f') < \infty$. Then the conformal parametrization of $\partial D$ is weakly Lipschitz continuous on $S^1$ with subpower function depending only on $\mc{D}_{\m D}(\log f')$.
\end{cor}
\begin{rem}
The condition $\mc D_{\m D} (\log f') <\infty$ allows $f'$ to be unbounded and the conformal parametrization is not Lipschitz in general in this setting, so up to the exact form of the subpower function this modulus of continuity is sharp. 
\end{rem}
\begin{proof}
    We have already noted that $f$ is continuous on $\overline{\m D}$. By Lemma~\ref{lem:derivative-estimate}, we have for $0 < r< 1$,  
    \begin{align*}
         |f(e^{i \t_1}) - f(e^{i \t_1})| 
        & \le  |f(e^{i \t_1})  - f(r e^{i\t_1})| +  |f(e^{i \t_2})  - f(r e^{i\t_2})| + 
|f(r e^{i\t_1})- f(r e^{i\t_2})| \\
& \le 2 \int_r^1\sigma(1/(1-u)) \dd u+ \sigma(1/(1-r))|\t_1 - \t_2|\\
& \le ((1-r) + |\t_1-\t_2| )\tilde \sigma(1/(1-r)),
    \end{align*}
    where $\tilde \sigma$ is a subpower function that depends only on $\mc{D}_{\m D}(\log f')$. Now choose $r=1-|\t_1-\t_2|$ and we obtain the desired estimate. 
\end{proof}

\section{Disk energy duality: proof of Theorem~\ref{thm:main}}\label{sec:disk_duality}

The proof of Theorem~\ref{thm:main} is completed at the end of the section. 
We assume that $\rho \in \mathcal{N}_+$ generates a foliation of $\ad{\m D} \smallsetminus \{0\}$ 
throughout.
The proof is carried out in two steps: in Section~\ref{subsec:beta_energy} we assume $S_+
(\rho)<\infty$ and derive energy duality. Then in Section~\ref{subsec:converse} we
assume $\mc D_{\m D}(\varphi)<\infty$ and prove that this implies $S_+(\rho)<\infty$. 
An overview of the argument presented in this section was provided in Section~\ref{sect:core-argument}.

\subsection{$S_+(\rho) < \infty$  implies $\mc D_{\m D}(\varphi) = 16 S_+(\rho)$} \label{subsec:beta_energy}

For $\rho \in \mathcal{N}_+$, we define (for a.e. $t$)
\[
\a_t (z) =   \Im (z H_t'(z)), \quad z \in \m D,
\]
where $H_t$ is the Herglotz integral of $\rho_t$.
In this section, we assume $S_+(\rho) <\infty$, and write as before $\dd \rho_t = \nu_t^2 (\t )\dd \t$. In this case, we have 
$$H_t (z) = \int_{0}^{2\pi} \frac{e^{i\t} +z }{e^{i\t} - z}\nu_t^2(\t) \dd\t. $$

\begin{lemma}\label{cold-water-swim}
If $L(\rho_t) < \infty$, then $H'_t \in \mathcal{H}^2$.
\end{lemma}
\begin{proof}We have $(\nu_t^2)' = 2 \nu_t' \nu_t \in L^2(S^1, \dd \t)$ since $\nu_t$ is bounded and $\nu_t' \in L^2(S^1, \dd \t)$ by assumption. Writing the complex derivative in polar coordinates shows that $\Im zH'_t(z) = -\partial_\theta \Re H_t(z)$. 
Let 
$$P_r(\t -s) := P_{\m D}(re^{i(\t-s)}) = (1-r^2)/|1-re^{i(\t -s)}|^2.$$  Then using integration by parts
\begin{align*}
\partial_\t \Re H_t(z) & =   \int_0^{2\pi} \partial_\t P_r(\t -s) \nu_t(s)^2\dd s  = \int_0^{2\pi} (-\partial_s P_r(\t -s)) \nu_t(s)^2\dd s \\
& = \int_0^{2\pi} P_r(\t -s) [\nu_t(s)^2]'\dd s = 2\pi P_{\m D}[(\nu_t^2)'](z).
\end{align*}
Therefore, since $(\nu_t^2)' \in L^2$, we get that $\Im H'_t \in \mathfrak{h}^2$. By \cite[Thm.\,4.1]{Duren_Hardy} this in turn implies $\Re H'_t \in \mathfrak{h}^2$ and we conclude that $H'_t \in \mathcal{H}^2$. 
\end{proof}

\begin{lemma}\label{lem:H2W12}
If $H_t' \in \mathcal{H}^1$ then for a.e. $\t \in S^1$, 
\begin{equation}\label{eq:im_nu}
\Im (e^{i\t} H_t'(e^{i\t})) = -4\pi \nu_t(\t)\nu_t'(\t),
\end{equation} 
where the left-hand side is understood in terms of radial limits and we have
$\alpha_t = -4\pi P_{\m D}[\nu_t \nu_t']$. In particular, this holds if $L(\rho_t) < \infty.$
\end{lemma}
\begin{proof}
If $H_t' \in \mathcal{H}^1$, then by \cite[Thm.\,5.2]{Duren_Hardy}
(see in particular the last paragraph of the proof), $H_t$ is continuous on $\overline{\m D}$ and the boundary function $H_t(e^{i\t})$ is absolutely continuous on $S^1$. Since $2\pi \nu_t^2(\t) = \Re H_t(e^{i\t})$, $4\pi  \nu_t (\t) \nu_t'(\t) = \partial_\t \Re H_t(e^{i\t}) $ exists a.e. on $S^1$. Moreover, $H_t'$ has radial limits a.e. on $S^1$ and by \cite[Thm.\,3.11]{Duren_Hardy}, $\partial_\t H_t(e^{i\t}) = ie^{i\t}\lim_{r \to 1} H_t'(re^{i\t})$ a.e. on $S^1$. This gives the identity \eqref{eq:im_nu}. Since $H_t' \in \mathcal{H}^1$, $\a_t (z) = \Im z H_t'(z)$ is the Poisson integral of its boundary values and this gives the second assertion. The final statement follows directly using Lemma~\ref{cold-water-swim}.
\end{proof}

The following lemma holds for all $\rho \in \mc N_+$.

\begin{lemma}\label{lem:A-alpha}
  For all $z\in D_T$, $ \vartheta[g_T](z)  =  \int_0^T  \a_t(g_t (z))\,\dd t$.
\end{lemma}
\begin{proof}
Since $t \mapsto \vartheta[g_t](z)$ is absolutely continuous on $[0,T]$, we can use the Loewner equation \eqref{eq:ODE} to see that for a.e. $t$,
\begin{align*}
\partial_t  \vartheta [g_t](z) & = \Im \left(\frac{\partial_t g_t'(z)}{g_t'(z)} - \frac{\partial_t g_t (z)}{g_t(z)} \right) \\
& =\Im \left(\frac{g_t'(z) H_t(g_t(z)) + g_t(z) H_t'(g_t (z)) g_t'(z)}{g_t'(z)} - H_t(g_t (z))\right) \\
& = \Im (g_t (z) H'_t (g_t (z))) =  \a_t ( g_t (z) ).
\end{align*}
Since $g_0(z)=z$ we get the claim after integration.
\end{proof}

We define
\begin{equation}\label{eq:winding_integral}\beta (z)  = \int_0^{\tau(z)} \a_t ( g_t (z) )\, \dd t, \qquad z \in \m D.
\end{equation}
It is not obvious that this quantity is finite a.e. However, part of the conclusion of the next result is that $\beta 
\in \mc E_0 (\m D)$ and we shall later prove that $\beta$ is the unique $\mc E_0(\m D)$ extension of the winding function $\varphi$.

\begin{prop}\label{prop:psi_winding}
Suppose $S_+(\rho) < \infty$. Let $u (\t,t) : = - 2\nu_t'(\t)/\nu_t(\t)$ 
 if $\nu_t (\t) \neq 0$, and $u(\t, t) := 0$ otherwise. Then $u \in L^2  (2\rho)$ and $\varkappa[u] = \beta$. 
In particular, $\beta \in \mc E_0(\m D)$ and $\mc D_{\m D} (\beta) = 16 \, S_+(\rho).$
\end{prop}
\begin{proof}
We verify directly that $u \in L^2 (2\rho)$ and
\begin{equation}\label{eq:proof_beta_rho}
\norm {u}_{L^2 (2 \rho)}^2 =  2 \int_0^{\infty} \int_{S^1} 1_{\nu_t \neq 0}
\left[\frac{2 \nu_t'(\t)} {\nu_t(\t)}\right]^2 \,\nu_t^2(\t)\,\dd\t\dd t = 16 \, S_+ (\rho) < \infty. 
\end{equation}
Corollary~\ref{cor:S_finite_foliation_WP} shows that $\rho$ generates a foliation, therefore the  disintegration isometry of Section~\ref{subsec:hadamard_isometry} applies.  Using Proposition~\ref{prop:kappa_formula} and Lemma~\ref{lem:H2W12}, 
\begin{align*}
\varkappa [u] (z)& = 2 \pi \int_0^{\tau(z)} P_{\m D}[u_t \nu_t^2] (g_t(z)) \, \dd t  = -4 \pi  \int_0^{\tau(z)}  P_{\m D}[ \nu_t \nu_t'] (g_t(z))   \,\dd t \\
&=  \int_0^{\tau(z)} \a_t ( g_t (z)) \, \dd t = \b(z).
\end{align*}
Moreover, by Theorem~\ref{thm:bi_isometry} and \eqref{eq:proof_beta_rho}, we obtain
$16\,S_+ (\rho) = \norm {u}_{L^2 (2 \rho)}^2 = \mc{D}_{\m D} (\varkappa[u]) = \mc D_{\m D} (\beta)$
as claimed.
\end{proof}

\begin{cor}\label{cor:duality_beta}
For $T > 0$, let $\rho^T_t = \rho_t$ for  $t \le T$ and let $\rho^T_t$ be the uniform measure for $t > T$. Let $\beta^T$ be the associated function as in \eqref{eq:winding_integral}. 
  Then we have $\vartheta [g_T] = \beta^T$  on $D_T$. In particular,
  $$\mc D_{D_T} (\vartheta [g_T]) \le \mc D_{\m D} (\b^T) = 16 \, S_+(\rho^T) = 16 \, S_{[0,T]}(\rho).$$ 
\end{cor}

\begin{proof}
We only need to note that for $t \ge T$, $\a_t \equiv 0$. Therefore, for $z \in D_T$,
$$\b^T (z) = \int_{0}^{\min\{\tau(z),T\}} \a_t (g_t (z))\, \dd t = \int_{0}^{T} \a_t (g_t (z))\, \dd t,$$
since $\tau(z) > T$.
 Lemma~\ref{lem:A-alpha} implies $\vartheta[g_T] =\beta^T$ on $D_T$. 
\end{proof}

We now show that $\b$ is the unique extension of the winding function  \eqref{eq:def_winding} $\varphi$ in $\mc E_0 (\m D)$. 
\begin{lemma}\label{lem:beta_varphi}
If $S_+ (\rho) < \infty$, then for all $t \ge 0$, 
\[
\beta|_{\partial D_t} = \varphi|_{\partial D_t} \quad \text{arclength-a.e.}, 
\]
where $\beta$ is as in \eqref{eq:winding_integral} and its trace is taken in the sense of Jonsson--Wallin~\eqref{def:trace}.
\end{lemma}

In other words, $\beta$ is the unique extension of $\varphi$ in $\mc E_0 (\m D)$ by Proposition~\ref{prop:unique_extension}  (and from now on we will not distinguish $\b$ and $\varphi$).

\begin{proof} 
We will identify functions in $ \mc E_0(\m D)$ with their extension to $W^{1,2} (\m C)$ by $0$ in $\m D^*$. By Corollary~\ref{cor:duality_beta}, $\beta^t = \beta$ in $\m C \smallsetminus D_t$ and the Jonsson--Wallin traces  satisfy
  $$\beta|_{\g_t} = \beta^t|_{\g_t} = \vartheta[g_t]|_{\g_t} = \varphi|_{\g_t}  \quad \text{arclength-a.e.}$$
  Here, the first equality is a property of the Jonsson--Wallin trace, see \cite[Lem.\,A.2]{VW1}.
  The second equality follows from  Corollary~\ref{cor:duality_beta}, where we interpret $\vartheta[g_t]|_{\g_t}$ as the non-tangential limit from inside $D_t$ using \cite[Lem.\,A.5]{VW1}. The last equality is the definition of $\varphi$.
\end{proof}

\begin{cor}\label{cor:iota_varphi}
If $S_+(\rho)<\infty$, then $\varphi \in \mc E_0 (\m D)$.
For $\rho$-a.e. $(\t, t)$, 
\begin{equation*}
    \iota [\varphi] (\t, t)
    = -2\nu_t'(\t)/\nu_t(\t)
\end{equation*}
and $  \mathcal{D}_{\m D}(\varphi) = 16 S_+(\rho).$ \end{cor}
\begin{proof}
 The proof is immediate by combining Proposition~\ref{prop:psi_winding} and Lemma~\ref{lem:beta_varphi}.
\end{proof}

\subsection{$\mc D_{\m D}(\varphi)<\infty$ implies $S_+(\rho)<\infty$}\label{subsec:converse}

This section proves the following result.
\begin{prop}\label{prop:converse}
  Suppose $\rho \in \mc N_+$ generates a foliation
  and assume that the winding function $\varphi : \mathcal T \mapsto \m R$ can be extended to a function in $\mc E_0 (\m D)$ \textnormal{(}also denoted $\varphi$\textnormal{)}. Then $S_+(\rho) <\infty$.
\end{prop} 

Assuming this proposition, we can complete the proof of Theorem~\ref{thm:main}.
\begin{proof}[Proof of Theorem~\ref{thm:main}]
Let $\rho \in \mc N_+$. If $S_+(\rho) <\infty$, Corollary~\ref{cor:iota_varphi} shows that $\varphi \in \mc E_0 (\m D)$ and 
$\mc D_{\m D} (\varphi) = 16 \, S_+ (\rho).$
Conversely, Proposition~\ref{prop:converse} shows that if $\varphi \in \mc E_0 (\m D)$, then $S_+ (\rho) < \infty$. Therefore the identity also holds.
\end{proof}

To explain our present goal towards the proof of Proposition~\ref{prop:converse}, note that we do not know \emph{a priori} that $\rho_t$ is absolutely continuous with respect to Lebesgue measure, nor that its density (if it exists) is differentiable almost everywhere. We need to show that this is the case. 

 Under the assumptions of Proposition~\ref{prop:converse}, we let as before $\varphi_s^0$ be the zero trace part of the function $\varphi \circ f_s$ and recall that $\a_t (z) =   \Im (z H_t'(z))$ for $z \in \m D$.
Since $\rho$ generates a foliation and since $\varphi \in \mc E_0(\m D)$ by assumption, using Theorem~\ref{thm:bi_isometry} we may consider \[u := \iota [\varphi] \in L^2 (2\rho).\] 
\begin{lemma}\label{lem:alpha_h}
For a.e. $t \ge 0$, we have 
 \begin{equation}\label{eq:alpha_h}
     \a_t (z) = 2 \pi  P_{\m D}[u_t \rho_t](z), \qquad u_t(\cdot) := u(\cdot, t).
 \end{equation}
\end{lemma}

\begin{proof}
Proposition~\ref{prop:kappa_formula} shows that a.e. in $\m D$,
\begin{equation}\label{eq:winding_formula_u}
 \varphi (w) = \varkappa[u](w) = 
  2 \pi \int_{0}^{\tau(w)}  P_{\m D}[u_t \rho_t](g_t(w))\, \dd t.
 \end{equation}
 Lemma~\ref{lem:zero_trace_varphi} shows that the unique orthogonal decomposition  of $\varphi$ with respect to $D_s$ is given by the zero-trace part $ \varphi_s^0 \circ g_s \in \mc E_0 (D_s)$ and the harmonic part $\varphi - \varphi_s^0 \circ g_s$ which equals $\vartheta[g_s]$ in $D_s$.
On the other hand, Corollary~\ref{cor:ortho_decomp_formula} shows that this decomposition is also given by 
$\varkappa[u\1_{S^1 \times [0, s)}]$ and $\varkappa [u \1_{S^1 \times [s,\infty)}]$. Hence,
for all $w \in D_s$ (so that $\tau(w) > s$),
$$\vartheta [g_s] (w)  = 2 \pi \int_0^{s}  P_{\m D}[u_t \rho_t](g_t(w))  \, \dd t. $$

Taking a derivative in $s$ in the above expression we obtain from Lemma~\ref{lem:A-alpha}  that for a fixed $w \in \m D$ and  a.e. $t < \tau(w)$,  $ 2 \pi  P_{\m D}[u_t \rho_t](g_t(w))  = \a_t 
(g_t (w))$. Indeed, by choosing a countable and dense family $J$ of points in $\m D$, we have for a.e. $t \ge 0$ and all $w \in D_t \cap J$, $ 2\pi  P_{\m D}[u_t \rho_t](g_t(w)) = \a_t (g_t (w))$.
Since both $\a_t$ and $P_{\m D}[u_t \rho_t]$ are continuous (actually harmonic) in $\m D$, we therefore obtain a.e. $t$, $\a_t = 2\pi  P_{\m D}[u_t \rho_t]$ for all $w \in \m D$.
\end{proof}

We are now ready to prove the main result of this section.
\begin{proof}[Proof of Proposition~\ref{prop:converse}]

Since $u = \iota[\varphi] \in L^2 (2\rho)$, for a.e. $t \in \m R_+$, $u_t \in L^2(S^1, 2\rho_t)$.
By the Cauchy-Schwarz inequality and the fact that $\rho_t \in \mc M_1 (S^1)$ is a probability measure, we also know that $u_t \in L^1 (S^1,2\rho_t)$. Hence $\a_t = 2\pi  P_{\m D}[u_t \rho_t] \in \mathfrak{h}^1$ by Lemma~\ref{lem:alpha_h}.
This implies $H_t' \in \mathcal{H}^p$ for any $p < 1$. (See \cite[Thm.\,4.2]{Duren_Hardy},  and use that an analytic function is in $\mathcal{H}^p$ if and only if its real and imaginary parts are in $\mathfrak{h}^{p}$.) Using \cite[Thm.\,5.12]{Duren_Hardy} this in turn implies $H_t \in \mathcal{H}^p$ for every $p < \infty$. (More precisely, if $f' \in \mathcal{H}^p$ for some $p<1$, then $f \in \mathcal{H}^q$ with $q=p/(1-p)$.) 
Therefore, the radial limits of $H_t$ exist a.e. and define a function in $L^p(S^1, \dd \t)$ for all $p<\infty$. It follows that the positive function $\Re H_t/2\pi$ is the Poisson integral of a function in $L^p(S^1, \dd \t)$ for all $p < \infty$, and we denote this function by $\nu_t^2 (\t)$. (See, e.g., Corollary 2 to  \cite[Thm.\,3.1]{Duren_Hardy}.)
In other words we have shown that the measure $\rho_t$ is absolutely continuous and
$\dd \rho_t(\t) = \nu_t^2(\t) \dd \t$, where $\nu_t^2 \in L^p(S^1, \dd \t)$ for every $p < \infty$.

Since $u_t^2  \nu_t^2 \in L^1(S^1, \dd \t)$ and $\nu_t^2 \in L^p(S^1, \dd \t)$ for all $p < \infty$, H\"older's inequality with $p=2/(2-\eps), q=2/\eps$ implies $u_t \nu_t^2  \in L^{2-\eps}(S^1, \dd \t)$ for any $\eps \in (0, 2)$. This in turn implies that the Herglotz integral of $u_t \nu_t^2$ is in $\mathcal{H}^{2-\eps}$ if $\eps \in (0, 1)$.  Lemma~\ref{lem:alpha_h} then implies that $H_t' \in \mathcal{H}^{2-\eps}$. It follows that $H_t$ is continuous on $\overline{\m D}$ and absolutely continuous on $S^1$ (see \cite[Thm.\,3.11]{Duren_Hardy}) and consequently so is the density $\t \mapsto \nu_t(\t)$ and $\nu_t'(\t)$ is well-defined for Lebesgue-a.e. $(\t, t)$.

Moreover,  
by taking radial limits in \eqref{eq:alpha_h}, we obtain using Lemma~\ref{lem:H2W12}
and Fubini's theorem that for Lebesgue-a.e. $(\t, t) \in S^1 \times \m R_+$,
$$ u_t (\t) \nu_t^2(\t) =  \frac{1}{2\pi} \a_t(e^{i\t}) = - 2\nu_t'(\t) \nu_t (\t).$$
It follows that Lebesgue-a.e., when $\nu_t (\t) \neq 0$,
$u_t^2(\t) \nu_t^2(\t) = 4\nu_t'(\t)^2,$
and both sides are equal to $0$ otherwise. 
We conclude the proof by integrating over $S^1 \times \m R_+$, and we obtain $S_+(\rho) <\infty$ since $u = \iota [\varphi] \in L^2(2 \rho)$.
\end{proof}

\section{Whole-plane energy duality: proof of Theorem~\ref{thm:main0}} \label{sec:whole-plane} 

In this section we deduce whole-plane energy duality,  Theorem~\ref{thm:main0}, from disk energy duality, Theorem~\ref{thm:main}. 

\subsection{Whole-plane Loewner evolution}\label{subsec:whole_plane_Loewner_chain}

We now describe the whole-plane Loewner chain.
We define similarly as before a space of driving measures
$$\mc N := \{\rho\in \M (S^1 \times \m R): \rho(S^1\times I) = |I| \text{ for all intervals } I \}.$$  
The whole-plane Loewner chain driven by $\rho \in \mc N$, or equivalently by its measurable family of disintegration measures $\m R \to \mc M_1 (S^1): \, t \mapsto \rho_t$, is the unique family of conformal maps $(f_t : \m D \to D_t)_{t \in \m R}$ such that 
\begin{enumerate}[label= (\roman*), itemsep= -2pt]
    \item For all $s < t$, $0 \in D_t \subset D_s$. \label{it:whole_monotone}
    \item For all $t \in \m R$, $f_t (0) = 0$ and $f_t'(0) = e^{-t}$ (in other words, $D_t$ has conformal radius $e^{-t}$ seen from $0$). \label{it:whole_radius}
    \item For all $s \in \m R$, $( f_t^{(s)} : = f_{s}^{-1} \circ f_t :  \m D \to D_{t}^{(s)})_{t \ge s}$ is the Loewner chain driven by $(\rho_t)_{t\ge s}$, which satisfies \eqref{eq:loewner-pde} with the initial condition $f_s^{(s)} (z) = z$, as discussed in Section~\ref{sect:Loewner--Kufarev-Equation}.  \label{it:whole_Loewner}
\end{enumerate}
\begin{rem}
   If $\rho_t \in \mc M_1 (S^1)$ is the uniform measure for all $t \le 0$, then $f_t (z) = e^{-t} z$ for $t \le 0$ and $(f_t)_{t\ge 0}$ is the Loewner chain driven by $(\rho_t)_{t \ge 0} \in \mc N_+$ as in Section~\ref{sect:Loewner--Kufarev-Equation}. Indeed, we check directly that $(f_t)_{t\in \m R}$ satisfy the three conditions above. In other words, the Loewner chain in $\m D$ is a special case of the whole-plane Loewner chain.
\end{rem}

Note that the range $D_{t}^{(s)}$ of $f_t^{(s)}$ has conformal radius $e^{s-t}$ seen from $0$ and the family $(f_t^{(s)})_{t \ge s}$ is uniquely defined for all $s \in \m R$ and satisfies for $t \ge s$ and $z \in \m D$,
$$\partial_t f_t^{(s)} (z) =  -z f_t^{(s)}{}' (z) H_t(z), \qquad H_t(z) = \int_{S^1} \frac{ e^{i\t} + z}{ e^{i\t}-z } \,\dd  \rho_t (\t),$$
with the initial condition $ f_s^{(s)} (z) = z$ (see Section~\ref{sect:Loewner--Kufarev-Equation}). 
The condition \ref{it:whole_Loewner} is then equivalent to for all $t \in \m R$ and $z  \in \m D$,
\begin{equation} \label{eq:whole-plane-radial}
 \partial_t f_t (z) = -z f'_t (z) H_t(z)
\end{equation}
as described in the introduction.
For the purposes of our proofs later on it is convenient to give a slightly more explicit construction of the family $(f_t)_{t \in \m R}$ and explain why it is uniquely determined by $t \to \rho_t$, even though this statement is well-known. 
For this, note that once we determine $f_n$,  \eqref{eq:whole-plane-radial} gives $f_t$ for all $t \ge n$ by 
\begin{equation}
    \label{eq:whole-plane-radial-comp}
f_t = f_n \circ f_t^{(n)}.
\end{equation}

{\bf Existence:} Consider for $- \infty < s \le t$, the conformal map 
$$F_t^{(s)} (z) : = e^{-s}f_t^{(s)} (z)$$ 
which maps $\m D$ onto $e^{-s} D^{(s)}_{t}$. Since $F_t^{(s)}{}'(0) = e^{-t}$, $(F_t^{(s)})_{s \in (-\infty, t]}$ is a normal family for any $t$. We extract a sequence $(s_k)$ converging to $-\infty$, such that for all $n \in \m Z$,
$F_n^{(s_k)}$ converges uniformly on compacts to a univalent function that we call $F_n$. We construct $(f_t)$ by taking $f_n : = F_n$ and generate the other $f_t$ using \eqref{eq:whole-plane-radial-comp}.
We need to verify compatibility, that is, that we have that $F_{n+1} = F_n \circ f_{n+1}^{(n)}$.
To see this, notice that for $s < t_1 < t_2$, 
$$(F_{t_1}^{(s)})^{-1} \circ F_{t_2}^{(s)} = (f_{t_1}^{(s)})^{-1} \circ f_{t_2}^{(s)} = f_{t_2}^{(t_1)} $$
is independent of $s$. The last equality follows from the fact that as a function of $t_2$ on $[t_1,\infty)$, both terms satisfy the same differential equation with same initial condition. Hence,  we have 
$F_{n}^{-1} \circ F_{n+1} = f^{(n)}_{n+1}$ as claimed.

{\bf Uniqueness:} If there are two such families $(f_t : \m D \to D_t)_{t \in \m R}$ and $(\tilde f_t : \m D \to \tilde D_t)_{t \in \m R}$. Let $\psi_t : = \tilde f_t \circ f_t^{-1}: D_t \to \tilde D_t$.
Since $(f_t)$ and $(\tilde f_t)$ are driven by the same process of measures,
for all $s \le t$,
$$f_s^{-1} \circ f_t = \tilde f_s^{-1} \circ \tilde f_t.$$
And for $z = f_t (w) \in D_t$,
$$\psi_s (z) = \tilde f_s \circ f_s^{-1} \circ f_t(w) =  \tilde f_s \circ \tilde f_s^{-1} \circ \tilde f_t(w) = \tilde f_t \circ f_t^{-1} (z) = \psi_t(z).$$
That is, $\psi_s|_{D_t} = \psi_t$. Hence, $\psi_t$ extends to a conformal map $\cup_{t \in \m R} D_t \to \cup_{t \in \m R} \tilde D_t = \m C$. This shows that $\psi_t$ is the identity map and $D_t = \tilde D_t$, since $\psi_t(0) = 0$ and $\psi_t'(0) = 1$ and completes the proof of the uniqueness.

We remark that if $\rho \in \mc N$, then  $\cup_{t\in \m R} D_t = \m C$.   Indeed, $D_{t}$ has conformal radius $e^{-t}$, therefore contains the centered ball of radius $e^{-t}/4$ by Koebe's $1/4$ theorem. 
We define for all $z \in \m C$, 
$$\tau(z) : =  \sup\{ t\in \m R :  z \in D_t\} \in (-\infty, \infty].$$

Similar to the definition of foliations of $\ad {\m D} \smallsetminus \{0\}$, we say that $\rho \in \mc N$ generates a foliation $(\g_t : = \partial D_t)_{t\in \m R}$ of $\m C \smallsetminus \{0\}$ if
\begin{enumerate}[itemsep=-2pt]
    \item For all $t \in \m R$, $\g_t$ is a chord-arc Jordan curve.
    \item It is possible to parametrize each curve $\g_t, t \in \m R,$ by $S^1$ so that the mapping $t \mapsto \g_t$ is continuous in the supremum norm.
    \item For all $z \in \m C \smallsetminus \{0\}$, $\tau(z) <\infty$.
\end{enumerate}

We have the whole-plane version of Lemma~\ref{lem:tau_foliates}:
\begin{lemma}
  Assume that $\rho$ generates a foliation $(\g_t)_{t\in \m R}$ of $\m C \smallsetminus \{0\}$. For all $z\neq 0$, we have  $z \in \gamma_{\tau(z)}$. In particular, $\bigcup_{t \ge 0} \g_t = \m C \smallsetminus \{0\}$.
\end{lemma}
We associate similarly to a  foliation of $\m C \smallsetminus \{0\}$ its \emph{winding function} defined by $\varphi (0) = 0$,
$$\varphi (z) : = \vartheta[g_{t}] (z),  \quad \text{for arclength-a.e. } z \in \g_t \text{ and } \forall t \in \m R,$$
 where $g_t = f_t^{-1}$. We may also define $\varphi(z)$ by $\vartheta[g_{\tau(z)}] (z)$.
We say that $\varphi$ has an extension to $W^{1,2}_{\mathrm{loc}}$ if $\varphi|_{\g_t}$ coincides a.e. with the Jonsson--Wallin trace of its extension on $\g_t$, for all $t \in \m R$.

\subsection{Proof of Theorem~\ref{thm:WP-leaf} and Theorem~\ref{thm:main0}}\label{subsec:whole_plane}

Recall the definition of the Loewner--Kufarev energy in the whole-plane setting: for $\rho \in \mc N$,
$$S (\rho)  = \int_{\m R} L (\rho_t) \, \dd t \quad \text{and} \quad S_{[a,b]} (\rho) =  \int_{a}^b L (\rho_t) \,\dd t.$$
 We use the notation from Section~\ref{subsec:whole_plane_Loewner_chain} and write for $s \le  t$,
$$g_t^{(s)} : = (f_t^{(s)})^{-1} = g_t \circ f_s, \quad g_t = f_t^{-1}.$$ 
Then $(g^{(s)}_t)_{t \ge s}$ is the uniformizing Loewner chain driven by $\rho^{(s)} : = (\rho_t)_{t\ge s}$, mapping $D_t^{(s)} = g_s (D_t)$ onto $\m D$.
We are now ready to prove Theorem~\ref{thm:WP-leaf}.
\begin{proof}[Proof of Theorem~\ref{thm:WP-leaf}]
By Conditions~\ref{it:whole_monotone} and \ref{it:whole_Loewner} we have that
   $$\cap_{t\in \m R} D_t = \cap_{t \in \m R_+} D_t = f_0 (\cap_{t \ge 0} D_t^{(0)}) = \{0\},$$
 where we used Lemma~\ref{lem:foliates} and  $S_+(\rho) < \infty$.  Therefore $\tau(z) <\infty$ for all $z \neq 0$.

Now we show that $\partial D_t$ is a Weil--Petersson quasicircle. Note that on $D_t$,
$$g_t = (g_t \circ f_s) \circ g_s = g^{(s)}_t   \circ g_s, \quad \forall s \le t.$$

Since $S_{[s,\infty)}(\rho^{(s)}) < \infty$, Proposition~\ref{prop:WP-QC-final} shows that $g_s (D_t) = g_s (f_t (\m D)) = f^{(s)}_t (\m D) =  D_t^{(s)}$ is bounded by a Weil--Petersson quasicircle.  
 Moreover, from the proof of Lemma~\ref{lem:foliates}, we know that for all $t \in \m R$ there is $s_0 < t$, such that $D_t^{(s_0)} \subset (1/2) \m D$.
As $f_{s_0}: \m D \to D_{s_0}$ is conformal, $\g_t = f_{s_0} (\partial D_t^{(s_0)})$ is also a Weil--Petersson quasicircle (e.g. by Lemma~\ref{thm_TT_equiv_T01}).
We also obtain the continuity of $t \mapsto \g_t$ from the continuity of $t \mapsto  \partial D_t^{(s_0)}$ by Corollary~\ref{cor:S_finite_foliation_WP}.
\end{proof}

We are now ready to prove whole-plane energy duality.
 \begin{proof}[Proof of Theorem~\ref{thm:main0}] 
Assume first $S(\rho) < \infty$, we show $\mc D_{\m C} (\varphi) < \infty$.
For this, note that the winding function $\varphi^{(s)}$ of the foliation in $\ad{\m D}\smallsetminus\{0\}$ driven by $\rho^{(s)} = (\rho_t)_{t \ge s}$ is given by 
\begin{align*}
\varphi^{(s)} (w) = \vartheta[g^{(s)}_{\tau(w)}] (w) = \vartheta[g_{\tau(z)}](z) - \vartheta[g_s] (z) = \varphi (z) - \vartheta[g_s ](z)
\end{align*}
    where $z = f_s (w) \in D_s$ and we used the chain rule \eqref{eq:theta_chain}.
    We have
    \begin{align}\label{eq:proof_whole-plane_duality}
    \begin{split}
    \mc D_{D_s} (\varphi)& = \mc D_{D_s} (\varphi^{(s)} \circ g_s) + \mc D_{D_s} (\vartheta[g_s])=  \mc D_{\m D} (\varphi^{(s)}) + \mc D_{D_s} (\vartheta[g_s]) \\
    & 
    = 16 \,S_{[s,\infty)} (\rho) +  \mc D_{\m D} (\vartheta[f_s]).    
    \end{split}
        \end{align}
    The first equality follows from orthogonality of $\vartheta[g_s]$ and $\varphi^{(s)} \circ g_s$ in $\mc E_0 (\m D)$, since $\vartheta[g_s]$ is harmonic in $D_s$ and $\varphi^{(s)} \circ g_s \in \mc E_0 (D_s)$. The second equality follows from the conformal invariance of the Dirichlet energy, and the third from Theorem~\ref{thm:main}. We obtain immediately the lower bound
    $$\mc D_{\m C} (\varphi) \ge 16 \, S(\rho).$$
    
    For the opposite inequality, since $F_s^{(r)} := e^{-r}f_s^{(r)}$ (see Section~\ref{subsec:whole_plane_Loewner_chain}) converges uniformly on compact subsets of $\m D$ to $f_s$ as $r \to -\infty$, we have that
    $\vartheta[F_s^{(r)}] (z)= \vartheta[f_s^{(r)}] (z) $ converges uniformly on compact sets to $\vartheta[f_s] (z)$. The lower semicontinuity of the Dirichlet energy then shows that for all compact $K \subset \m D$,
    $$\mc D_{K} (\vartheta[f_s]) \le \liminf_{r \to -\infty} \mc D_{K} (\vartheta[f_s^{(r)}]).$$
    On the other hand, Corollary~\ref{cor:duality_beta} shows that 
    $$\mc D_{K} (\vartheta[f_s^{(r)}]) = \mc D_{f_s^{(r)} (K)} (\vartheta[g_s^{(r)}]) \le 16 \, S_{[r,s]} (\rho),$$
    letting $r \to -\infty$ yields
    $$\mc D_{\m D} (\vartheta[f_s]) \le 16 \int_{-\infty}^s L(\rho_t)\, \dd t. $$
    Combining this with \eqref{eq:proof_whole-plane_duality} shows that 
    $\mc D_{D_s} (\varphi) \le 16 \,S(\rho)$. Letting $s \to -\infty$ we obtain the upper bound and hence $\mc D_{\m C}(\varphi) = 16 \, S(\rho)$.
    
    For the converse, if $\rho$ generates a foliation of $\m C\smallsetminus \{0\}$ with $\mc D_{\m C} (\varphi) < \infty$, Proposition~\ref{prop:converse} and \eqref{eq:proof_whole-plane_duality} imply that 
    $$16\, S_{[s, \infty)} (\rho^{(s)}) = \mc D_{\m D} (\varphi^{(s)})  \le \mc D_{D_s} (\varphi) \le \mc D_{\m D} (\varphi).$$
    Letting $s \to -\infty$ we obtain $S(\rho) < \infty$ and this completes the proof. 
\end{proof}

\section{Applications of energy duality}\label{sec:application}
In this section we derive several consequences of Theorem~\ref{thm:main0}.
\subsection{Reversibility of the Loewner--Kufarev energy}\label{sec:reversibility}

Let $(\g_t)_{t \in \m R}$ be a foliation generated by $\rho \in \mathcal N$ and let $(D_t)_{t \in \m R}$ be the corresponding family of domains. We will consider the evolution of its time-reversal, that is the Loewner chain corresponding to the family $(\tilde D_{s(t)} : = j (\Chat \smallsetminus D_t))_{t \in \m R}$,  where $j(z):= 1/z$ and $e^{-s (t)}$ is the conformal radius of $j (\Chat \smallsetminus D_t)$. Let $\tilde f_{s}$ be the conformal map from $\m D$ onto $\tilde D_s$ with $\tilde f_s(0) = 0$ and $\tilde f_s'(0) = e^{-s}$.

\begin{lemma}\label{lem:capacity_complement}
  The function $t \mapsto s(t)$ is a decreasing homeomorphism of $\m R$.
\end{lemma}
\begin{proof}
To show that $t \mapsto s$ is decreasing, let $t_1 < t_2$. Since $D_{t_2} \varsubsetneq D_{t_1}$, we have $\tilde D_{s(t_1)} \varsubsetneq \tilde D_{s(t_2)}$. Schwarz' lemma then shows that $s (t_2) < s(t_1)$. 
We claim that $s(t) \to \infty$ as $t \to -\infty$. In fact, Koebe's $1/4$-theorem shows that $D_t$ contains the centered disk of radius $e^{-t}/4$, therefore $\tilde D_{s(t)}$ is contained in the centered disk of radius $4 e^{t}$. Schwarz' lemma shows $ s(t) \ge - t  - \log 4$ which goes to $\infty$ as $t \to -\infty$.
Since the diameter of $D_t$ tends to $0$ 
as $t \to \infty$ by Lemma~\ref{lem:foliates}, $\tilde D_{s(t)}$ has conformal radius tending to $\infty$. Therefore $s (t) \to -\infty$ as $t \to \infty$.  It remains to verify that $s$ is continuous. To see this, note that the continuity of $t \mapsto \g_t$ in the supremum norm shows that as $t \to t_0 \in \m R$, $\tilde D_{s(t)}$ converges in the Carath\'eodory kernel sense to $\tilde D_{s(t_0)}$. Therefore,
$\tilde f_{s(t)}'(0)$ tends to $\tilde f_{s(t_0)}'(0)$, and equivalently, $s(t)$ tends to $s(t_0)$.
\end{proof}

 Lemma~\ref{lem:capacity_complement} implies that $ (\tilde f_s)_{s\in \m R}$ is a  whole-plane Loewner chain as defined in Section~\ref{sec:whole-plane}. Note that in Section~\ref{sec:whole-plane} we took the measure $\rho$ as starting point for the definition whereas we have constructed the monotone family of domains here. However, using Pommerenke's theorem as discussed in Section~\ref{sect:Loewner--Kufarev-Equation} the existence of a measure $\tilde \rho$ generating the family follows easily. Energy duality, Theorem~\ref{thm:main0}, now implies the following reversibility of the Loewner--Kufarev energy.

\begin{thm} [Energy reversibility]\label{thm:energy_rev}Let $\tilde \rho \in \mc N$ be the measure generating $(\tilde D_s)_{s \in \m R}$. 
 Then $S (\rho) = S(\tilde \rho). $
\end{thm}

For the proof of Theorem~\ref{thm:energy_rev} we will use the following lemma which allows us to relate the winding function associated with $\rho$ to that with $\tilde \rho$. A weaker lemma sufficient for our purposes appeared in an earlier version of the paper and in \cite{Wang_survey}. The proof of the present stronger version was suggested by an anonymous referee. 
\begin{lemma}\label{lem:matching_trace}
Let $\g$ be a Jordan curve separating $0$ from $\infty$. Suppose $z \in \mathcal T_\g$. Then the winding functions $\varphi_\g$ and $\varphi_{1/\g}$ are well-defined at $z$ and $1/z$, respectively, and
$$\varphi_\g (z) = \varphi_{1/\g}(1/z). $$
\end{lemma}

 Geometrically, $\varphi_\g (z)$  measures the angle between  $\g$ and the circle centered at $0$ passing through $z$ modulo $2\pi$. The map $z\mapsto 1/z$ is conformal and maps centered circles to centered circles, so the angle is preserved. The difficulty lies in proving the equality without ``modulo $2\pi$'' with our choice for the branch of the argument. Before giving the proof, we first complete the proof of Theorem~\ref{thm:energy_rev} assuming the lemma.

\begin{proof}[Proof of Theorem~\ref{thm:energy_rev} assuming Lemma~\ref{lem:matching_trace}]
Assume that $S(\rho) < \infty$. Theorem~\ref{thm:WP-leaf} implies that every leaf of the associated foliation is a Weil--Petersson quasicircle, hence rectifiable and differentiable almost everywhere.
Lemma~\ref{lem:matching_trace} implies that the winding function $\tilde \varphi$ associated with the foliation generated by $\tilde \rho$ satisfies 
$$ \tilde \varphi (1/z) = \varphi (z)$$
for all $z \in \mc T_{\g_t}$, $t \in \m R$.
Since the Dirichlet energy is invariant under conformal mappings, we obtain $\mc D_{\m C} (\tilde \varphi) = \mc D_{\m C} ( \varphi)$ which implies $S (\rho) = S(\tilde \rho)$ by Theorem~\ref{thm:main0}. 

Finally, if $S(\rho) = \infty$, then $S(\tilde \rho) = \infty$. Indeed, if this does not hold we would have $S(\rho) < \infty$, since $\rho$ corresponds to the time-reversed foliation of that associated with $\tilde \rho$. This contradiction completes the proof.
\end{proof}

 It remains to prove Lemma~\ref{lem:matching_trace}. For this we will need some more notation and another lemma. Let $\gamma$ be a Jordan curve separating $0$ from $\infty$. Recall that the winding function  $\varphi_\g (z)$ at $z \in \mathcal{T}_\g$ is by definition the non-tangential limit at $z$ of $\vartheta [g]$ where $g : D \to \m D$ is a conformal map fixing $0$ (the fact that the limit exists follows from the remark after Lemma~\ref{lem:boundary-behavior}). We first give an equivalent expression for the winding function at a given $z \in \mathcal{T}_\g$ in terms of the turning of a smooth curve approaching $z$ normally.

We introduce the following notation. For a simply connected subset $V\subset \m C$ containing $0$, and $z \notin V$, let $w \mapsto \arg_V (z/(z - w))$ be the unique real-valued continuous function on $V$ such that $\arg_V(z/(z - 0)) = 0$ and 
$\exp\left(i \arg_V(z/(z - w))\right) = z\abs{z - w}/|z|(z - w)$.
If $v : [0,1] \to \m C$ is a smooth path joining $0$ to $z \in \m C, z \neq 0,$ we define the \emph{turning} of $v$ around $z$ as
$$\Theta_v (z) = \lim_{s \to 1^-} \arg_{v[0,1)} \left(\frac{z}{z - v(s)}\right) = \int_v \dd \arg \frac{z}{z - w}.$$
The limit exists since $v$ is smooth and has a well-defined tangent at $z$ and $z \neq 0$.

\begin{lemma} \label{lem:turning_winding}
    Let $\g$ be a Jordan curve separating $0$ and $\infty$, and let $D$ be the bounded connected component of $\m C \smallsetminus \g$ containing $0$. Suppose $z \in \mc T_\g$.
    Let $v : [0,1] \to \ad D$ be any smooth path joining $0$ to $z$ in $D$ \textnormal(with $v(1) = z \in \ad D$\textnormal). Then $|\Theta_v (z) - \varphi_\g (z)| \le \pi/2$. In particular, if $v$ approaches $z$ normally, then $\Theta_v (z) = \varphi_\g (z)$. 
\end{lemma}

\begin{proof}
    By \eqref{eq:vartheta} we have
    $$\varphi_\g (z) = \int_0^1 \dd \arg \frac{g(z) - g (v(s))}{ z - v(s)} = \lim_{s \to 1^-}  \left[\arg \frac{g(z) - g (v(s))}{ z - v(s)} - \arg \frac{g(z) }{ z } \right]$$
    where the branch of $\arg$ is chosen continuously along $v$ (and the difference vanishes when $s = 0$). This may be rewritten as
     \begin{align*}
     \varphi_\g (z) & = \lim_{s \to 1-} \left[ \arg_D \left( \frac{z}{ z - v(s)} \right)- \arg_{\m D} \left(\frac{g(z) }{  g(z) - g (v(s))} \right)\right] \\
     & = \Theta_v (z) - \lim_{s \to 1^-}  \arg_{\m D} \left(\frac{g(z) }{  g(z) - g (v(s))} \right).
     \end{align*}
     Since $g$ takes values in $\m D$ and $g (z) \in S^1$, we have that $w \mapsto  \frac{g(z) }{  g(z) - w} $ maps $\m D$ to $1/(1 + \m D)= \{z \in \m C \,|\, \Re (z) > 1/2\}$. Therefore, \[\arg_{\m D} \left(\frac{g(z) }{  g(z) - g (v(s))} \right)\in (-\pi/2, \pi/2).\] We obtain  $|\Theta_v (z) - \varphi_\g (z)| \le \pi/2$.

     If $v$ approaches $z$ normally, then from the geometric interpretation of $\varphi_\g$ (see Figure~\ref{fig:intro_fol}), we have $\Theta_v (z) = \varphi_\g (z)$ modulo $2\pi$. The last paragraph shows that $\Theta_v (z) = \varphi_\g (z)$.
 \end{proof}

\begin{proof}[Proof of Lemma~\ref{lem:matching_trace}]
Write $\tilde{\g} = 1/\g$. Let $D$ and $\tilde D$ be the bounded connected component of $\m C \smallsetminus \g$ and of $\m C\smallsetminus \tilde \g$ respectively. Given $z \in \mc T_\g$, let $v$ be a smooth path connecting $0$ to $z$ in $D$, approaching $z$ normally as in Lemma~\ref{lem:turning_winding}. Similarly, let $\tilde v$ be a smooth path connecting $0$ to $1/z$ in $\tilde D$, approaching $1/z$ normally. We assume furthermore that the paths $v$ and $1/\tilde v$ can be concatenated to form a smooth path from $0$ to $\infty$. That is, 
$$ u (s) : = \begin{cases} v(s), \qquad & s \in [0,1) \\
z, & s  = 1\\
1/\tilde v(1/s), & s = (1, \infty]
\end{cases}
$$
is also smooth in a neighborhood of $s = 1$.

Now we define the function
$$F_u (s) = \begin{cases}
    \dfrac{u(1)}{u(1) - u (s)} \abs{\dfrac{u(1) - u (s)}{u(1)}}, \qquad &  s \in [0,1) \\
    \dfrac{u(1)}{u'(1)} \abs{\dfrac{u'(1)}{u(1)}}, & s= 1\\
     \dfrac{u(s)}{u(s) - u (1)} \abs{\dfrac{u(s) - u (1)}{u(s)}}, \qquad & s \in (1, \infty) \\
     1, & s = \infty.
\end{cases}$$
We note that $F_u (0) = 1 = F_u(\infty) = \lim_{s \to \infty} F_u (s)$. Hence, $F_u$ defines a continuous map $[0,\infty]_{/0\sim \infty} \simeq S^1 \to S^1$. We note that for $s > 1$,
$$ F_u (s) = \frac{u(s)}{u(s) - u (1)} \abs{\frac{u(s) - u (1)}{u(s)}} = \frac{1/z}{ 1/z - \tilde v(1/s)} \abs{\frac{ 1/z - \tilde v(1/s)} {1/z}} \xrightarrow[]{s \to 1+}\exp \left(i \Theta_{ \tilde v} (1/z)\right). $$
By definition
$$\lim_{s \to 1-} F_u(s) = \exp (i \Theta_v (z)).$$

Now we scale and reparametrize the path $u$: consider for $r \in (0,1]$, 
$$u_r (s) : = r^{-1} u(rs)$$ and define $F_{u_r} : [0,\infty]_{/0\sim \infty} \simeq S^1 \to S^1$ analogously to $F_u$ (with $z$ replaced by $u_r (1) = r^{-1} u(r)$).  It is not hard to see that $F_{u_r}$ converges uniformly as $r \to 0+$ to the constant function $1$, so $r \mapsto F_{u_r}$ defines a homotopy from the constant function $1$ to $F_u$. 
Thus, there exists a continuous argument along $s \mapsto F_u(s)$ such that $\arg F_u (0) = \arg F_u (\infty) = 0$, which implies $$\Theta_v (z) = \arg F_u (1-) = \arg F_u (1+) = \Theta_{\tilde v} (1/z).$$ By Lemma~\ref{lem:turning_winding} we conclude that $\varphi_\g (z) = \varphi_{1/\g} (1/z)$.
\end{proof}

\subsection{Characterization of Weil--Petersson quasicircles}\label{subsec:jordan_curve}
Let $\g$ be a Weil--Petersson quasicircle separating $0$ from $\infty$. We will associate with $\g$ a particular measure $\rho = \rho^\g$ (and foliation) with $S(\rho) < \infty$ and in this way prove Theorem~\ref{thm:main-jordan-curve}. 
 We assume for notational simplicity that the bounded component of $\m C \smallsetminus \g$ has conformal radius $1$\footnote{This assumption is only made for convenience, so that the curve $\g$ corresponds to time-index $0$ in the foliation. All results in this section hold for a general Weil--Petersson quasicircle $\g$ separating $0$ from $\infty$.}. Let $f$ (resp. $h$) be the conformal map from $\m D$  (resp. $\m D^*$) to the bounded component $D_0$ (resp. unbounded component $D_0^*$) of $\m C \smallsetminus \g$ such that $f$ (resp. $h$) fixes $0$ (resp. $\infty$) and has derivative $f'(0) = 1$ (resp. $h'(\infty) >0$).
 Consider the foliation $(\g_t)_{t \in \mathbb{R}}$ that consists of $\g$ together with the family of equipotentials on both sides of $\g$. By equipotential which we mean image of a circle $r S^1$ under $f$ (resp. under $h$), and we include all equipotentials corresponding to $r < 1$ (resp. $r > 1$). The parametrization of  $(\g_t)_{t \in \m R}$ by $t$ is chosen so that the connected component $D_t$ of $\Chat \smallsetminus \g_t$ containing $0$ has conformal radius $e^{-t}$.
  
 Let $\rho^\g \in \mc N$ be the measure associated with $(\g_t)_{t \in \m R}$ and let 
$\varphi$ be the corresponding winding function. 
Along with the Loewner chain $(f_t: \m D \to D_t )_{t \in \m R}$, we consider also the conformal maps $h_t : \m D^*\to  D_t^* $, such that $h_t(\infty) = \infty$ and $h_t '(\infty) > 0$. 
In particular, $f = f_0$ and $h = h_0$.
We also set $g_t : = f_t ^{-1}$ and $k_t : = h_t ^{-1}$. For a conformal map $k$ fixing $\infty$, we define similarly $$\vartheta [k](z) = \arg \frac{z k'(z)}{k(z)},$$ 
where the argument is chosen so that $\vartheta[k](z) \to 0$ as $z \to \infty$.

\begin{lemma} \label{lem:harm_equi}
For a.e. $t \ge 0$, $\rho_t^\g$ is the uniform measure. Moreover, $\varphi|_{\m C \smallsetminus \g}$ is harmonic, and we have
$$\varphi|_{D_0} = \vartheta [g_0], \quad \text{and} \quad \varphi|_{D_0^*} = \vartheta [k_0].$$
\end{lemma}
\begin{proof}
By our definition of equipotential we have $\g_t = f (e^{-t} S^1)$ for $t \ge 0$. The flow $(g_t \circ f)_{t \ge 0}$, driven by $(\rho_t^\g)_{t \ge 0}$ equals $(z \mapsto e^t z)$, which shows that $(\rho_t^\g)_{t\ge 0}$ is a.e. the uniform measure on $S^1$. The identity  $\varphi = \vartheta [g_0]$  on $D_0$ and hence, the harmonicity of $\varphi|_{D_0}$ follows as in Lemma~\ref{lem:harmonic}.

Now we consider $D_0^*$.
Using the notation of Theorem~\ref{thm:energy_rev}, the winding function in $D_0^*$ is given by $\varphi (z) = \tilde \varphi \circ j (z)$ where $\tilde \varphi$ is the winding function associated with the foliation under the inversion map $j : z \mapsto 1/z$ by Lemma~\ref{lem:matching_trace}.
Since the family of equipotentials is preserved under the inversion, it follows that $\tilde \varphi$ is harmonic in $j (D_0^*)$ and
$$\tilde \varphi = \vartheta [j \circ k_0 \circ j] = \vartheta [k_0] \circ j,$$
using the chain rule \eqref{eq:theta_chain} and $\vartheta [j] \equiv 0$.
Therefore $\varphi|_{D_0^*} = \vartheta [k_0]$. 
\end{proof}

Recall that the Loewner energy of the curve $\g$ is
\begin{equation}\label{eq:def_Loewner_energy}
I^L(\g) = \mc D_{\m D} (\arg f') + \mc D_{\m D^*} (\arg h') + 4 \log |f'(0)/h'(\infty)|.
\end{equation}
(We have normalized so that $f'(0) = 1$ but we will keep it in the notation since \eqref{eq:def_Loewner_energy} holds more generally.)
\begin{thm} \label{thm:dual_Jordan_curve} Let  $\g$ be a Weil--Petersson quasicircle separating $0$ from $\infty$. Let $\rho^\g$ be the measure associated to $\g$. Then
  \[16 \,S(\rho^\g) = I^L(\g) - 2 \log |f'(0)/h'(\infty)| < \infty.\]
  Moreover, if $\rho \in \mc N$ generates $\g$ as a leaf, then $S(\rho^\g) \le S(\rho)$. 
\end{thm} 
\begin{proof}
By Lemma~\ref{lem:harm_equi}, the winding function associated to $\rho^\g$ satisfies
\begin{align*}
    \mc D_{\m C} (\varphi) & =  \frac{1}{\pi}\int_{D_0} \abs{\nabla \vartheta [g_0]}^2 \dd A (z) + \frac{1}{\pi} \int_{D_0^*} \abs{\nabla \vartheta [k_0]}^2 \dd A (z) \\
    & =  \frac{1}{\pi}\int_{\m D} \abs{\nabla \vartheta [f]}^2 \dd A (z) + \frac{1}{\pi} \int_{\m D^*} \abs{\nabla \vartheta [h]}^2 \dd A (z) \\
    & =\frac{1}{\pi} \int_{\m D} \abs{\frac{f''}{f'} - \frac{f'}{f} + \frac{1}{z}}^2 \dd A (z) +  \frac{1}{\pi} \int_{\m D*} \abs{\frac{h''}{h'} - \frac{h'}{h} + \frac{1}{z}}^2 \dd A (z) \\
 & =  \frac{1}{\pi} \int_{\m D} \abs{\frac{f''}{f'}}^2 \dd A (z)  + \frac{1}{\pi} \int_{\m D^*} \abs{\frac{h''}{h'}}^2 \dd A (z) -   \frac{1}{\pi} \int_{\m D} \abs{\frac{f'}{f} -\frac{1}{z}}^2  \dd A (z) \\&\qquad 
  - \frac{1}{\pi} \int_{\m D^*} \abs{\frac{h'}{h} -\frac{1}{z}}^2  \dd A (z) \\
& =  \frac{1}{\pi} \int_{\m D} \abs{\frac{f''}{f'}}^2 \dd A (z)  + \frac{1}{\pi} \int_{\m D^*} \abs{\frac{h''}{h'}}^2 \dd A (z) + 2 \log |f'(0)/h'(\infty)|\\
& = I^L(\g) - 2 \log |f'(0)/h'(\infty)|,
\end{align*}
where we used \eqref{eq:def_Loewner_energy} and Lemma~\ref{lem:Grunsky} proved just below.
Theorem~\ref{thm:main0} then implies the identity. 
The claim $S(\rho^\g) \le S(\rho)$ follows from the fact that the foliation formed by the equipotentials in $D$ is generated by the zero energy measure for $t \ge 0$, thus minimizes the energy among all foliations in $D$, and the reversibility Theorem~\ref{thm:energy_rev}.
\end{proof}

\begin{lemma}\label{lem:Grunsky}
  We have the identity
  \begin{align*}
    \Re & \left[ \int_{\m D}   \frac{f''}{f'}\left( \ad{\frac{f'}{f} - \frac{1}{z}}\right)\dd A (z) +\int_{\m D^*}  \frac{h''}{h'} \left(\ad{\frac{h'}{h} - \frac{1}{z}}\right) \dd A (z) \right] \\
      & =  \int_{\m D} \abs{ \frac{f'}{f }  - \frac{1}{z} }^2  \dd A (z)+  \int_{\m D^*} \abs{\frac{h'}{h }  -\frac{1}{z} }^2  \dd A (z)  = 2 \pi \log \abs{\frac{ h'(\infty)}{f'(0)}}. 
  \end{align*}
\end{lemma}
\begin{proof}
Consider $\tilde f : = j\circ h \circ j$ and $\tilde h : = j \circ f \circ j$, the conformal maps associated to $j (\g)$. We have $\tilde f' (0) = h'(\infty)^{-1}$  and 
$$\frac{\tilde f''(z)}{\tilde f' (z)} = - \frac{1}{z^2} \left( \frac{h''(1/z)}{h' (1/z)} - \frac{2 h'(1/z)}{h (1/z) }  +2z \right)$$
and similarly for $\tilde h$. We compute
\begin{align*}
     I^L(j (\g))  & =  \frac{1}{\pi}\int_{\m D} \abs{\frac{\tilde  f''}{\tilde f'}}^2 \dd A (z) + \frac{1}{\pi}\int_{\m D^*} \abs{\frac{\tilde h''}{\tilde h'}}^2 \dd A (z)  + 4 \log \abs{\frac{\tilde f'(0)}{\tilde  h'(\infty)}} \\
     & =  \frac{1}{\pi}\int_{\m D^*} \abs{ \frac{h''}{h'} - \frac{2 h'}{h }  + \frac{2}{z} }^2  \dd A (z) + \frac{1}{\pi}\int_{\m D} \abs{ \frac{f''}{f'} - \frac{2 f'}{f }  + \frac{2}{z} }^2  \dd A (z)  + 4 \log \abs{\frac{f'(0)}{h'(\infty)}}  \\
     & =  I^L(\g) -  \int_{\m D^*} 4 \Re \left[\frac{h''}{h'} \left(\ad{\frac{h'}{h} - \frac{1}{z}}\right) \right]\dd A (z) -   \int_{\m D} 4 \Re  \left[ \frac{f''}{f'}\left( \ad{\frac{f'}{f} - \frac{1}{z}}\right)\right]\dd A (z)  \\
      &\quad  + 4 \int_{\m D^*} \abs{\frac{h'}{h }  -\frac{1}{z} }^2  \dd A (z) + 4\int_{\m D} \abs{ \frac{f'}{f }  - \frac{1}{z} }^2  \dd A (z).
\end{align*}
Since the Loewner energy of a Jordan curve is M\"obius invariant, we have $I^L(\g) = I^L(j (\g))$ and we obtain the first equality. The second equality follows from Lemma~\ref{lem:Grunsky_inequality} since $\g$ has Lebesgue measure zero.
\end{proof}

Combining Theorem~\ref{thm:WP-leaf} and Theorem~\ref{thm:dual_Jordan_curve}, we obtain a new characterization of Weil--Petersson quasicircles. 
\begin{cor}\label{cor:WP-characterization}
  A Jordan curve $\g$ separating $0$ from $\infty$ is a Weil--Petersson quasicircle if and only if $\g$ can be realized as a leaf in the foliation generated by a measure $\rho$ with $S(\rho) < \infty$.
\end{cor}

\begin{cor}\label{cor:WP-LE-bound}
  The Loewner energies of the leaves generated by $\rho$ are uniformly bounded by $16\, S(\rho)$.
\end{cor}
\begin{proof}
  For all $t \in \m R$,  Theorem~\ref{thm:main0} 
   and  Theorem~\ref{thm:dual_Jordan_curve} imply 
   $$16 \,S(\rho)  \ge 16 \,S(\rho^\g) =   I^L (\g_t) - 2 \log |f_t'(0)| + 2 \log |h_t'(\infty)|,$$
   where
   $f_t, h_t$ are the conformal maps associated to $\g_t$. On the other hand, Lemma~\ref{lem:Grunsky_inequality} shows that $$ 2 \log |h_t'(\infty)| - 2 \log |f_t'(0)|  \ge 0.$$
   Thus,
   $I^L(\g_t) \le 16 \, S(\rho) $
 as claimed.
\end{proof}

\begin{rem}
Corollary~\ref{cor:WP-LE-bound} provides a way to generate and simulate Weil--Petersson quasicircles of bounded Loewner energy using a measure with controlled Loewner--Kufarev energy. In fact, we obtain infinitely many such quasicircles for any given measure.
\end{rem}

\begin{rem}
The Loewner energy $I^L(\g)$ is invariant under M\"obius transformations, and is known to be a K\"ahler potential for the Weil--Petersson metric defined on a subspace $T_0 (1)$ of the universal Teichm\"uller space, see \cite[Thm.\,II.4.1]{TT06}. On the other hand, $\log |f'(0)/h'(\infty)|$ is only invariant under scaling and rotation which is consistent with the fact that the Loewner--Kufarev equation has two marked points $0$ and $\infty$ on the Riemann sphere. We also point out that $\log |f'(0)/h'(\infty)|$ is a K\"ahler potential for the Velling-Kirillov metric on the universal Teichm\"uller curve, a complex fiber bundle over $T_0(1)$, see \cite[Thm.\,I.5.3]{TT06}.    
\end{rem}

\subsection{Complex identity: Proof of Proposition~\ref{prop:complex_id}} \label{sec:complex}
Recall that a Weil--Petersson quasicircle $\g$ separating $0$ from $\infty$ is said to be compatible with $\varphi \in W^{1,2}_{\textrm{loc}}$ if the winding function of $\g$ coincides with the trace $\varphi|_{\g}$ arclength-a.e.  Since $\g$ by assumption is compatible with $\Im \psi$, 
   $$\vartheta[f](z) 
  = -\vartheta[f^{-1}] (f(z)) = -\Im \psi^{h}( f(z) ), \quad \forall z \in \m D,$$
   and Lemma~\ref{lem:harm_equi} also shows
    $$\vartheta[h](z) 
   = -\Im \psi^{h}( h(z) ), \quad \forall z \in \m D^*,$$
   where $\Im \psi^h + \Im \psi^0$ is the orthogonal decomposition of $\Im \psi$ with respect to $\m C \smallsetminus \g$ as in \eqref{eq:orthogonal}. 
   Hence we can write, 
   \begin{align*}
       \zeta  &= \left( \Re \psi \circ f + \log \abs{\frac{zf' (z)}{f(z)}}\right) + i  \left( \Im \psi \circ f + \vartheta[f](z) \right) = u + i  \Im \psi^0 \circ f; \\
       \xi &= v + i \Im \psi^0 \circ h,
   \end{align*}
   where 
   $$u : = \Re \psi \circ f + \log \abs{\frac{zf' (z)}{f(z)}}, \quad v : = \Re \psi \circ h + \log \abs{\frac{zh' (z)}{h(z)}}.$$ 
   Notice that $\log |zf' (z)/f(z)|$ is a harmonic conjugate of $\vartheta[f]$, so they have the same Dirichlet energy.
   Therefore
   \begin{align*}
\mc D_{\m D} (\zeta) + \mc D_{\m D^*} (\xi) & =  \mc D_{\m D}(u) + \mc D_{\m D}(v) + \mc D_{\m C} (\Im \psi^0) \\
& = \mc D_{\m C} (\Re \psi) + \mc D_{\m C} (\Im \psi^{h}) + \text{``cross-terms''} + \mc D_{\m C} (\Im \psi^0) \\
& = \mc D_{\m C} (\Re \psi) +\mc D_{\m C} (\Im \psi)   + \text{``cross-terms''} 
   \end{align*}
   where the ``cross-terms'' come from expanding the Dirichlet integrals of $u$
    and $v$ and equal $2$ times
    \begin{align}\label{eq:cross-terms}
    \int_{\m D} \bigg\langle \nabla \Re \psi \circ f,  \nabla & \log \abs{\frac{f' (z)z}{f(z)}} \bigg\rangle   \, \dd  A(z) 
     + \int_{\m D^*} \bigg\langle \nabla \Re \psi \circ h, \nabla \log \abs{\frac{h' (z)z}{h(z)}} \bigg\rangle \, \dd  A(z).
    \end{align}
    It suffices to show that \eqref{eq:cross-terms} vanishes.
   We will only prove it assuming that $\g$ is smooth. The general case of a Weil--Petersson quasicircle can be deduced using an approximation argument following exactly the same proof as \cite[Thm.\,3.1]{VW1}.
    
    Using Stokes' formula, the first term in \eqref{eq:cross-terms} equals
    \begin{align*}
      &   \int_{\partial \m D} \Re \psi \circ f (z) \partial_n \log  \abs{\frac{f' (z)z}{f(z)}} |\dd z| \\
          & =  \int_{\partial \m D} \Re \psi \circ f (z) \partial_n \log |f' (z)| |\dd z| + \Re  \left[\int_{\partial \m D} \Re \psi \circ f (z) \partial_n \left(\log  \frac{z}{f(z)} \right)|\dd z|\right] \\
          & = :I_1 + I_2
    \end{align*}
    where $\partial_n$ is the normal derivative in the outward pointing direction. Using the formula $\partial_n \log |f'(z)| = k_{\O}\circ f(z)|f'(z)| - k_{\m D} (z)  = k_{\O}\circ f(z)|f'(z)| - 1 $, where $k_{\O}$ is the geodesic curvature of $\partial \O$ (see, e.g., \cite[Appx.\,A]{W2}), we obtain
    \begin{align*}
    I_1 &=  \int_{\partial \O} \Re \psi (w) k_{\O}(w) |\dd w| - \int_{\partial \m D} \Re \psi \circ f (z)|\dd z|.        
    \end{align*}
Using $z |\dd z| = -i \dd z$, we have
    \begin{align*}
    I_2 &=  \Re \int_{\partial \m D} \Re \psi \circ f (z) z \left( - \frac{f'(z)}{f(z)} + \frac{1}{z} \right)|\dd z| \\
    &= \Re \int_{\partial \m D} i \Re \psi \circ f (z) \frac{f'(z)}{f(z)} \dd z +\int_{\partial \m D} \Re \psi \circ f (z) |\dd z| \\
    & = \Re \int_{\partial \O}  \frac{i  \Re \psi (w)}{w} \dd w +\int_{\partial \m D} \Re \psi \circ f (z) |\dd z|.
    \end{align*}
    The sum of $I_1$, $I_2$ and those integrals coming from the second term of \eqref{eq:cross-terms} vanishes since $k_\O (y) = - k_{\O^*} (y)$ and the contour integral in $I_2$ winds in the opposite direction for $\O$ and $\O^*$. This concludes the proof. \qed

\section{Conformal distortion formula}\label{subsec:variation_LK}

The goal of this section is to compute an explicit formula for the change of the Loewner--Kufarev energy under conformal transformation of the foliation and prove
Theorem~\ref{thm:conformal-distortion} which combines Proposition~\ref{prop:conformal-distortion} and Corollary~\ref{cor:E_rho_variation}.

We consider the following set up. Let $\rho \in \mc N_+$ be such that 
$$S_{[0,1]} (\rho)  = \int_0^1 L(\rho_t)\,\dd t < \infty.$$ 
We write as before $\rho_t = \nu_t^2 (\t ) \dd \t$.
Suppose $\psi$ is a conformal transformation that maps $K_1$ onto another compact hull $\psi(K_1)=\tilde{K}_1 \subset \ad{\m D}$, and is defined on a neighborhood $U$ of $K_1$
in $\ad {\m D}$ 
which is mapped onto a neighborhood $\tilde U$ of $\tilde K_1$ in $\ad {\m D}$. Note that we always have $S^1 \subset K_1$ by definition.
See Figure~\ref{fig:distortion}. The image of the hulls $(\tilde  K_t: = \psi (K_t))$, is driven by a measure $t\mapsto \tilde \rho_t$. Note that $\tilde \rho_t$ may not be a probability measure since $\tilde D_t := \m C \smallsetminus \tilde K_t$ in general does not have conformal radius $e^{-t}$.

\begin{figure}[h]
    \centering
    \includegraphics[scale=0.55]{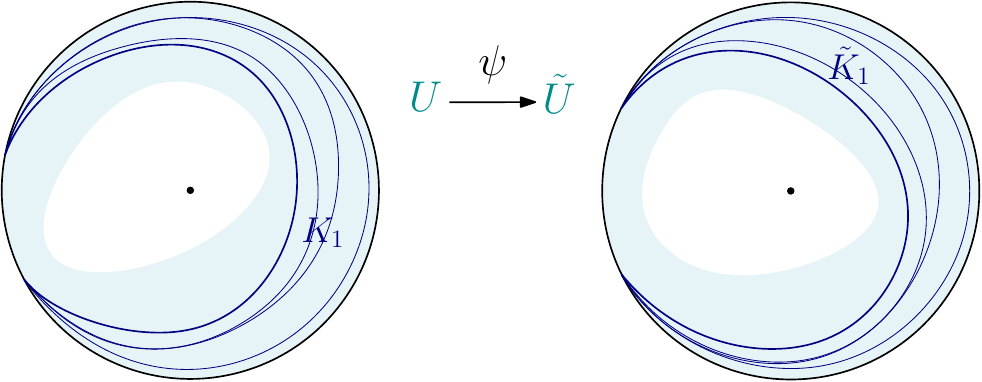}
    \caption{Conformal distortion of a foliation. The mapping $\psi$ is conformal in a neighborhood of the hull $K$ generated by $\rho$, which is mapped to another hull $\tilde K$, generated by a measure $\tilde \rho$. Proposition~\ref{prop:conformal-distortion} computes the difference of the Loewner--Kufarev energies of $\rho$ and $\tilde \rho$.}
    \label{fig:distortion}
\end{figure}

Let $(g_t = f_t^{-1})_{t \le 1}$ and $(\tilde g_t = \tilde f_t^{-1})_{t \le 1}$ be the corresponding uniformizing Loewner chains.   Let $\psi_t : = \tilde g_t \circ \psi \circ f_t$. Then $\psi_t$ is a conformal map of $U_t : = g_t (U \smallsetminus K_t)$ onto $\tilde U_t : = \tilde g_t (\tilde U \smallsetminus \tilde K_t)$. By the Schwarz reflection principle, $\psi_t$ extends to a holomorphic function in a neighborhood of $S^1$.

\begin{lemma}\label{lem:nu_transform}
   For a.e. $t \in [0,1]$, $\tilde \rho_t \ll \dd \t$. If we write $\tilde \rho_t = \tilde \nu_t^2(\t) \dd \t$, 
  then
 \begin{equation}\label{eq:rho_tilde_rho}
 \tilde \nu_t^2 (\t_t) = |\psi_t'(e^{i\t})| \nu_t^2(\t),
 \end{equation}
  where $\theta_t \in [0,2\pi]_{/0\sim 2\pi}$ satisfies $e^{i \theta_t} = \psi_t (e^{i\t})$. 
\end{lemma}
\begin{proof}
 Let $\tilde H_t$ be the Herglotz integral of $\tilde \rho_t$. Let $0<r < 1$ and consider for $t \in [0,1]$ the curve $\g_t^r = f_t (r S^1)$ passing through $f_t(re^{i\t})=:w$. 
 We can compute the normal velocity of $\psi \circ \g_t^r$ at $\psi(w)$ in two ways. First starting from the velocity at $w$ using that $\psi$ is conformal: this gives $r|\psi'(w)||f'_t(re^{i\t})| \mathrm{Re}  \, H_t(re^{i\t})$; and second, directly from the Loewner--Kufarev equation driven by $\tilde \rho$, which gives  
 $$|\tilde f'_t(\psi_t(re^{i\t}))|\mathrm{Re}   \,\overline{e^{i\t}\psi_t'(re^{i\t})}|\psi'_t(re^{i\t})|^{-1}\psi_t(re^{i\t}) \tilde H_t(\psi_t(re^{i\t})).$$ Indeed, to see why these formulas hold, first note that a unit normal at $w$ is
\[
 \mathbf{n}(\g_t^r(e^{i\t}))= -e^{i\t} \frac{f'_t(re^{i\t})}{|f'_t(re^{i\t})|}, \quad \t \in [0, 2\pi).
\]
Using the Loewner equation, we see that the normal velocity with respect to the curve $\g_t^r(\t)$ at time $t$ at the point $w$ is 
\begin{align*}
- \mathrm{Re}  \frac{\partial_t f_t(re^{i\t}) \overline{e^{i\t} f'_t(re^{i\t})}}{|f'_t(re^{i\t})|}  & =   \mathrm{Re}   \frac{re^{i\t} f_t'(re^{i\t}) H_t(re^{i\t})  \overline{e^{i\t} f'_t(re^{i\t})}}{|f'_t(re^{i\t})|}  \\
& = r|f'_t(re^{i\t})| \mathrm{Re} \, H_t(re^{i\t}).
\end{align*}

Next, by definition $\psi_t  = \tilde g_t \circ \psi \circ f_t$, so $\tilde{f}_t \circ \psi_t = \psi \circ f_t$. Since $w = f_t(re^{i\t})$ we have $\psi'(w) f'_t(re^{i\t}) = \tilde{f}_t'(\psi_t(r e^{i\t})) \psi_t'(re^{i\t})$ and the normal velocity of the image curve at $\psi(w)$ is
\begin{align*}
 - \Re  &  \frac{(\partial_t \tilde{f}_t)(\psi_t(re^{i\t})) \overline{e^{i\t} \psi'(w) f'_t(re^{i\t}) }}{|\psi'(w)||f'_t(re^{i\t})| } \\
&  = \Re \frac{ \psi_t(re^{i\t}) \tilde{f}_t'(\psi_t(re^{i\t})) \tilde H_t(\psi_t(re^{i\t})) \overline{e^{i\t} \psi'(w) f'_t(re^{i\t}) }}{|\psi'(w)||f'_t(re^{i\t})| } \\
& = \Re \frac{ \psi_t(re^{i\t}) \tilde{f}_t'(\psi_t(re^{i\t})) \tilde H_t(\psi_t(re^{i\t})) \overline{e^{i\t} \tilde{f}_t'(\psi_t(r e^{i\t})) \psi_t'(re^{i\t}) }}{|\tilde{f}_t'(\psi_t(r e^{i\t})) \psi_t'(re^{i\t})| } \\
& = |\tilde{f}_t'(\psi_t(re^{i\t})) |\Re \frac{ \psi_t(re^{i\t}) \tilde H_t(\psi_t(re^{i\t})) \overline{e^{i\t}  \psi_t'(re^{i\t}) }}{|\psi_t'(re^{i\t})| }
\end{align*}
We get 
\begin{align*}
& r|\psi'(w)||f'_t(re^{i\t})| \mathrm{Re}  \, H_t(re^{i\t}) =|\tilde f'_t(\psi_t(re^{i\t}))|\mathrm{Re} \left( \,\frac{\overline{e^{i\t}\psi_t'(re^{i\t})}}{|\psi'_t(re^{i\t})|}\psi_t(re^{i\t}) \tilde H_t(\psi_t(re^{i\t})) \right).
\end{align*}
Note that $H_t$ is continuous on $\overline{\m D}$ and $\psi_t$ extends to be holomorphic on a neighborhood of $S^1$. Moreover, as $r \to 1-$, $\overline{e^{i\t}\psi_t'(re^{i\t})}\psi_t(re^{i\t})  /|\psi'_t(re^{i\t})|\to 1$. 
 Since $\Re \tilde H_t = 2\pi P_{\m D}[\tilde \rho_t]$, we obtain that $\tilde \rho_t \ll \dd \t$ and \eqref{eq:rho_tilde_rho} by letting $r \to 1-$ and using the definition $\psi_t  = (\tilde f_t)^{-1} \circ \psi \circ f_t$ and the chain rule.
\end{proof}

For the next statement, recall that 
$\mc S f $ 
denotes the Schwarzian derivative of $f$.
\begin{prop}\label{prop:conformal-distortion}
 We have
 $$L(\tilde \rho_t) - L(\rho_t) =  \frac{1}{4} \int_{S^1} e^{2i\t} \mc S \psi_t (e^{i\t})   \dd \rho_t (\t) + \frac{1}{8}\left( |\tilde \rho_t| - | \rho_t| \right).$$
\end{prop}

Note that by conjugating $\psi_t$ by a M\"obius transformation mapping $\m D$ to $\m H$, one sees that $e^{2i\t} \mc S \psi_t (e^{i\t}) \in \m R$.

\begin{proof}
 We use the same notation as in Lemma~\ref{lem:nu_transform}. 
Differentiating \eqref{eq:rho_tilde_rho}
with respect to $\t$, we obtain using $\partial \t_t /\partial \t = |\psi_t'(e^{i\t})|$ that
 $$\tilde \nu_t'(\t_t) \sqrt {\abs{\psi_t'(e^{i\t})}} = \nu_t'(\t) + \frac{\partial_\t \abs{\psi_t'(e^{i\t})}}{2\abs{\psi_t'(e^{i\t})}} \nu_t (\t).$$
 
 Plugging this into the expression for $L (\tilde \rho_t)$,  we get
 \begin{align*}
   L (\tilde \rho_t) & = \frac{1}{2} \int_{S^1} \tilde \nu_t'(\t)^2 \,\dd \t =  \frac{1}{2} \int_{S^1} \tilde \nu_t' (\psi_t (e^{i\t}))^2 \abs{\psi_t'(e^{i\t})}  \, \dd \t  
   = \frac{1}{2} \int_{S^1} \left(\nu_t' + \frac{\partial_\t\abs{\psi_t'}}{2\abs{\psi_t'}} \nu_t \right)^2  \dd \t   \\
   & = L (\rho_t) + \frac{1}{2}\int_{S^1}  \frac{\partial_\t |\psi_t'| }{|\psi_t'|} \nu_t' \, \nu_t \, \dd\t + \frac{1}{8} \int_{S^1}  \left(\frac{\partial_\t |\psi_t'|}{|\psi_t'|}\right)^2 \nu_t^2 \, \dd \t.
 \end{align*}
 
 Integrating 
 \begin{align*}
 \partial_\t \left[\nu_t^2  \frac{\partial_\t |\psi_t'|}{ |\psi_t'|}\right] & = 2 \nu_t' \,\nu_t \frac{\partial_\t |\psi_t'|}{ |\psi_t'|} + \nu_t^2 \left[ \frac{\partial_\t^2 |\psi_t'|}{ |\psi_t'|} - \left(\frac{\partial_\t |\psi_t'|}{|\psi_t'|}\right)^2\right],
 \end{align*}
 over $S^1$ against $\dd \t$ gives $0$. It follows that $L(\tilde \rho_t) - L(\rho_t)$ equals
 \begin{align*}
 & \frac{1}{2}\int_{S^1}  \frac{\partial_\t |\psi_t'| }{|\psi_t'|} \nu_t' \, \nu_t \, \dd\t + \frac{1}{8} \int_{S^1}  \left(\frac{\partial_\t |\psi_t'|}{|\psi_t'|}\right)^2 \nu_t^2 \, \dd \t   =  - \frac{1}{4}\int_{S^1} \nu_t^2\left[\frac{\partial_\t^2 |\psi_t'|}{|\psi_t'|} - \frac{3}{2} \left(\frac{\partial_\t |\psi_t'|}{|\psi_t'|}\right)^2 \right] \,  \dd\t.
 \end{align*}

Using $|\psi_t' (z)| = \psi_t ' (z) z /\psi_t(z)$ and $\partial_\t = iz \partial_z$, we compute
\begin{align*}
\frac{\partial_\t |\psi_t' (z)|}{ |\psi_t' (z)| } & = iz \partial_z \log |\psi_t' (z)| = iz \left(\frac{\psi_t''}{\psi_t'}  - \frac{\psi_t'}{\psi_t} + \frac{1}{z}\right)
\end{align*}
and
\begin{align*}
-\frac{1}{2} \left( \frac{\partial_\t |\psi_t' (z)|}{ |\psi_t' (z)| }\right)^2 &= \frac{z^2}{2} \left[ \left(\frac{\psi_t''}{\psi_t'}\right)^2
+ \left(\frac{\psi_t'}{\psi_t}\right)^2 + \frac{1}{z^2} - 2 \frac{\psi_t''}{\psi_t} - 2 \frac{\psi_t'}{z\psi_t} + 2 \frac{\psi_t''}{z\psi_t'}
\right] \\
& = \frac{z^2}{2}  \left(\frac{\psi_t''}{\psi_t'}\right)^2
+ \frac{z^2}{2}\left(\frac{\psi_t'}{\psi_t}\right)^2 + \frac{1}{2} - z^2 \frac{\psi_t''}{\psi_t} -  \frac{z \psi_t'}{\psi_t} + \frac{z\psi_t''}{\psi_t'}.
\end{align*}
Moreover,
\begin{align*}
    \partial_\t \left(\frac{\partial_\t |\psi_t' (z)|}{ |\psi_t' (z)| } \right) & = iz \left[ i  \left(\frac{\psi_t''}{\psi_t'}  - \frac{\psi_t'}{\psi_t} + \frac{1}{z}\right) + iz \left(\left(\frac{\psi_t''}{\psi_t'}\right)'  - \left(\frac{\psi_t'}{\psi_t}\right)' - \frac{1}{z^2}\right) \right] \\
    & = - z^2 \left(\frac{\psi_t''}{\psi_t'}\right)' - z \left(\frac{\psi_t''}{\psi_t'}  - \frac{\psi_t'}{\psi_t} + \frac{1}{z}\right) + z^2 \left(\frac{\psi_t''}{\psi_t} - \frac{(\psi_t')^2}{\psi_t^2}\right) + 1.
\end{align*}
We obtain
\begin{align*}
    \frac{\partial_\t^2 |\psi_t'|}{|\psi_t'|} - \frac{3}{2} \left(\frac{\partial_\t |\psi_t'|}{|\psi_t'|}\right)^2 & = \partial_\t \left(\frac{\partial_\t |\psi_t' |}{ |\psi_t' | } \right) -\frac{1}{2} \left( \frac{\partial_\t |\psi_t'|}{ |\psi_t'| }\right)^2\\  
    & = - z^2 \mc S \psi_t - \frac{z^2}{2}\left(\frac{\psi_t'}{\psi_t}\right)^2 +\frac{1}{2} = -z^2 \mc S \psi_t +\frac{1 - |\psi_t'|^2}{2}. 
\end{align*}
Combining these computations,
 we get
 \begin{align*}
     L(\tilde \rho_t) - L(\rho_t) 
     & = - \frac{1}{4}  \int_{S^1} \nu_t^2 (\t)\left[\frac{\partial_\t^2 |\psi_t'|}{|\psi_t'|} - \frac{3}{2} \left(\frac{\partial_\t |\psi_t'|}{|\psi_t'|}\right)^2 \right]\dd\t \\
     & = \frac{1}{4}  \int_{S^1} \nu_t^2 (\t)\left[e^{2i\t} \mc S \psi_t +\frac{|\psi_t'|^2-1}{2} \right] \,\dd\t  \\
     & = \frac{1}{4} \int_{S^1} \nu_t^2 (\t) e^{2i\t} \mc S \psi_t (e^{i\t}) \,\dd\t + \frac{1}{8}\left( \int_{S^1}  \tilde \nu_t^2 (\t) \,\dd\t - \int_{S^1} \nu_t^2 (\t) \,\dd\t \right),
 \end{align*}
 where we used \eqref{eq:rho_tilde_rho} in the last equality and
 $$\int_{S^1} \nu_t^2 (\t) |\psi_t'|^2 \,\dd\t = \int_{S^1} \tilde \nu_t ^2 (\t_t) \, |\psi_t'| \,\dd\t = \int_{S^1} \tilde \nu_t ^2 (\t) \,\dd\t $$
 which completes the proof.
\end{proof}

\begin{cor} \label{cor:E_rho_variation} We have
  \begin{equation}\label{eq:integrated_distortion}
  S_{[0,1]} (\tilde \rho) - S_{[0,1]} (\rho) = \frac{1}{4} \int_0^1  \int_{S^1}  e^{2i\t} \mc S \psi_t (e^{i\t}) \rho_t (\t) \,\dd\t\dd t +
  \frac{1}{8}\left( \log \tilde g_1'(0) - 1 \right).
  \end{equation}
\end{cor}

\begin{proof}
The formula follows by integrating \eqref{eq:rho_tilde_rho}, using $\int_0^1 |\rho_t|\dd t = 1$ and  $\int_0^1 |\tilde \rho_t|\dd t = 
 \log \tilde g_1 '(0)$.
\end{proof}
\begin{rem}
The Brownian loop measure is a conformally invariant $\sigma$-finite measure on Brownian loops in the plane \cite{LSW_CR_chordal, LW2004loopsoup}. The conformal distortion formula of 
Theorem~\ref{thm:conformal-distortion} 
can be interpreted in terms of Brownian loop measures: namely, $S_+(\rho) - S_+(\tilde \rho)$ also equals the difference between the Brownian loop measure of those loops in $\m D$ that intersect both $K_1$ and $\m D \smallsetminus U$ and the measure of those loops in $\m D$ that intersect both $\tilde K_1$ and $\m D\smallsetminus \tilde U$. We will prove this in a forthcoming paper with Lawler. For now we just remark that this interpretation immediately implies that the energy difference depends only on the hulls at time $1$ and not on the foliation, a fact that is not immediately apparent from  \eqref{eq:integrated_distortion}.
\end{rem}

\section{Further comments and open problems}
\label{sec:further}

We will now indicate further implications and interpretations as well as open problems suggested by our results.

\vspace{10pt}

\textbf{Random conformal geometry and mating-of-trees.} We begin by discussing connections to ideas in random conformal geometry, see in particular \cite{IG4, MoT}, which served as inspiration for the formulation of our main result, Theorem~\ref{thm:main0}. We will also speculate on how to formulate stochastic versions of some of the results obtained in this paper but we do not make rigorous statements here.

Let  $(B_t)_{t\in \m R}$ be the rotation-invariant two-sided Brownian motion on $S^1$ and  
suppose $\k \ge 0$. Whole-plane 
SLE$_\k$ is the random (whole-plane) Loewner chain generated by the measure 
$\rho(\dd \t \dd t)  =  \d_{B_{\k t}} (\t) \dd t \in \mc N$, where $\d_{B_{\k t}}$ is the Dirac mass at $B_{\k t}$ and $\dd t$ is Lebesgue measure on $\m R$. When $\k \ge 8$, the hull is  generated by a space-filling curve growing in $\m C$ from $\infty$ towards $0$, see \cite{Rohde_Schramm}. The Loewner--Kufarev energy is infinite in the case of SLE, 
since a Dirac mass is not an absolutely continuous measure. However, an easy generalization of the main result of \cite{APW} shows that the Loewner--Kufarev energy is the large deviation rate function of whole-plane SLE$_{\infty}$. Roughly speaking, 
as $\k \to \infty$,
\begin{align}\label{eq:heurstic_infty}
\mathbb{P} \big\{\text{Whole-plane } \SLE_{\k}  \text{ domains stay close to }   & (D_t)_{t\in \m R}\big\}  
\approx \exp \left(- \k S(\rho)\right)
\end{align}
where $(D_t)_{t \in \m R}$ is the family of domains of any deterministic whole-plane Loewner chain with driving measure $\rho$. See \cite[Thm.\,1.2]{APW} for a precise statement in the unit disk setup. On the other hand, the Loewner energy $I^L$ of a Jordan curve can be expressed in terms of Dirichlet energies, see \eqref{eq:loop_LE_def}, and it is believed to be the large deviation rate function of the $\SLE_{0+}$ loop: as $ \kappa \to 0\!+\!$,
\begin{equation} \label{eq:heuristic_0}
 \mathbb{P} \left\{ \SLE_\k \text{ loop stays close to } \g \right\}  \approx \exp \left(- I^L(\g)/\k\right)
 \end{equation}
 where $\g$ is a given deterministic Jordan curve. See \cite{W1,peltola_wang} for precise statements in chordal settings. 
Furthermore, it is well-known that the Dirichlet energy is the large deviation rate function 
for the Gaussian free field, see \cite[Thm.\,3.4.12]{Deuschel-strook}.

 SLE processes enjoy a duality property with respect to replacing $\kappa$ by $16/\kappa$ 
 \cite{Dub_duality,Zhan_duality,IG1}. Roughly speaking, an $\SLE_{\k}$ curve describes locally the outer boundary of an $\SLE_{16/\k}$ hull when  $\k < 4$.
 The mating-of-trees theorem \cite{MoT} further explores this duality and the interplay with an underlying Liouville quantum gravity field. (See also, e.g., \cite{ang2020conformal} and the references therein for some recent progress in this direction.) An impressionistic picture is as follows. A pair of ``mated'' space-filling trees whose branches are formed by $\SLE_{\k}$-like curves are constructed as flowlines of an Gaussian field and a coupled whole-plane chordal space-filling SLE$_{16/\k}$ curve from $\infty$ to $\infty$ traces the interface between the pair of trees.  
 One can speculate that the union of the pair of trees degenerate to a foliation as $\k \to 0\!+$ (and $16/\kappa \to +\infty$) and the mating-of-trees coupling suggests that the large deviation rate functions of the coupled processes should match in this limit. Combining this with the heuristic formulas \eqref{eq:heurstic_infty} and \eqref{eq:heuristic_0} led us to guess Theorem~\ref{thm:main0} where the factor $16$ is consistent with the SLE $\k \leftrightarrow 16/\k$ duality.

 However, the setup of \cite{MoT} uses whole-plane chordal Loewner evolution so the space-filling SLE there runs from $\infty$ to $\infty$, whereas the whole-plane radial SLE runs from $\infty$ to $0$. 
A coupling of radial SLE with the Gaussian free field is described in \cite{IG4}, but we are not aware of results similar to the mating-of-trees theorem in the current literature for the setting we work in. 
Our results, in particular Theorem~\ref{thm:WP-leaf}, Theorem~\ref{thm:main0}, and Proposition~\ref{prop:complex_id} provide analytical evidence for a radial mating-of-trees theorem.

Using the dictionary we outlined in \cite[Sec.\,1.3 and 3.4]{VW1}, we may speculate that the following statements should hold. For small $\k > 0$, run a space filling whole-plane SLE$_{16/\k}$ on an appropriate ``quantum sphere'' assumed to be independent of the SLE. (A quantum sphere can be described by a Gaussian free field with additional logarithmic singularities and an attached transformation law.)
If the SLE process is run up to time $t$, the unvisited part and visited part of the quantum sphere form two independent ``quantum disks'' (which are also defined starting from a Gaussian free field) conformally welded along the frontier of the whole-plane SLE$_{16/\k}$, which itself is an SLE$_\k$-type loop. 
The two ``quantum disks'' are each decorated with an independent radial SLE$_{16/\k}$ curve, and the two SLE$_{16/\k}$ curves are independent conditionally on the position of the tip at time $t$. Varying $t$, we expect the separating SLE$_\k$-type loops to form a  (fractal) foliation-like family that sweeps out the twice punctured sphere. This foliation-like process also encodes the whole-plane SLE$_{16/\k}$ evolution. The real part of the complex field in Proposition~\ref{prop:complex_id} reflects the metric/measure structure of the quantum sphere, whereas the imaginary part encodes the fractal foliation-like process, hence the trajectory of the space filling SLE.

\vspace{10pt}

\textbf{Whole-plane radial SLE Reversibility.} Let us  next comment on  the reversibility of the Loewner--Kufarev energy, Theorem~\ref{thm:main_rev}.
An analogous result about the reversibility of the Loewner energy can be explained (and proved) by SLE$_{0+}$ large deviations considerations combined with the fact that chordal SLE is 
reversible \cite{Zhan_rev} for small $\k$, see  \cite{W1}.  Chordal SLE is however not reversible for $\kappa > 8$, and it is not known whether whole-plane SLE$_\k$ for $\k > 8$ is  reversible or not.
(For $\k \le 8$, reversibility was established in \cite{zhan_rev_whole,IG4}.) Therefore Theorem~\ref{thm:main_rev} cannot be predicted from the SLE point of view given currently known results, but it does on the other hand suggest that reversibility for whole-plane radial SLE might hold for large $\kappa$ as well.

\emph{Update:} In the time since our paper first appeared, proofs of large $\kappa$ reversibility for whole-plane SLE and a radial mating-of-trees theorem have appeared \cite{ang2023, ang232}.

\vspace{10pt}

\textbf{Whole-plane chordal Loewner--Kufarev energy.} One can ask about a version of our results in chordal settings. The most natural one is the whole-plane chordal version, where the family of curves all pass through $\infty$ and foliate the plane $\m C$ as $t$ ranges from $-\infty$ to $\infty$, as in the mating-of-trees theorem.
 When $\k \to \infty$, the whole-plane chordal SLE$_\k$ Loewner chain converges to the constant identity map for all time $t$. Therefore, renormalization is needed to obtain both a non-trivial limit and a meaningful large deviation result. (This is one reason to work in the radial setup here as well as in \cite{APW}.)
One way to proceed is to conformally map the two punctures ($0$ and $\infty$) in our whole-plane (radial) setup to $y$ and $\infty$ then let $y \to \infty$. The third complex degree of freedom (ranging in a non compact space) in the choice of conformal automorphism of the Riemann sphere needs to be chosen carefully to obtain a clean statement.

\vspace{10pt}

\textbf{Foliation loops in Weil--Petersson Teichm\"uller space.} Recall that any quasicircle $\g$ separating $0$ from $\infty$ can be identified with an element of universal Teichm\"uller space
 $T(1) \simeq \mob(S^1)\backslash \QS(S^1)$ via (the equivalence class of) its welding homeomorphism $\phi_\g  = h^{-1} \circ f|_{S^1}$. The subspace $T_0(1)$ corresponding to Weil--Petersson quasicircles 
 has an infinite dimensional K\"ahler-Einstein manifold structure when equipped with the Weil--Petersson metric, see \cite{TT06}.
Theorem~\ref{thm:WP-leaf} shows that any $\rho$ with $S(\rho) < \infty$ generates a foliation $(\g_t)_{t \in \m R}$ of Weil--Petersson quasicircles which can be considered as elements of $T_0(1)$ via their welding homeomorphisms. So the Loewner evolution of $\rho$ generates a dynamical process on $T_0(1)$. Theorem~\ref{thm:main-jordan-curve} shows that there exists such a family (obtained by interpolating by equipotentials) passing through any given element of $T_0(1)$. It is not too hard to show that it corresponds to a continuous loop $t \mapsto [\phi_{\g_t}]$ in $T_0(1)$ starting and ending at the origin $[\operatorname{Id_{S^1}}]$. 
We believe the class of loops in $T_0(1)$ coming from measures with $S(\rho)<\infty$ may be of interest to study. 
 For instance, a natural question concerns how the length of a loop is related to $S(\rho)$. The example in Section~\ref{sect:examples} shows that one can have $S(\rho) > 0$ while the corresponding loop is trivial, so one can only hope for an upper bound in terms of $S(\rho)$. One can further ask for properties of the minimal energy (equipotential) path to a given element.
Another interesting question concerns how transformations on $\rho$ affects a path in $T_0(1)$, and vice versa. 

\vspace{10pt}

\textbf{``Foliations'' in hyperbolic $3$-space.} Since M\"obius transformations of $\Chat$ extend to isometries of  the hyperbolic $3$-space $\m H^3$ (whose boundary at $\infty$ is identified with the Riemann sphere $\Chat$) and being a Weil--Petersson quasicircle is a M\"obius invariant property, it is natural to try to relate our foliations by Weil--Petersson quasicircles to objects in $\m H^3$.
In \cite{Anderson}, it is shown that every Jordan curve bounds at least one minimal disk in $\m H^3$, and \cite{bishop-WP} shows that a Jordan curve is a Weil--Petersson quasicircle if and only if  any such minimal disk in $\m H^3$ has finite total curvature. For example, when $\g$ is a circle, the unique minimal surface is the totally geodesic surface, namely the hemisphere bounded by $\g$. 
Although the minimal disk for a given boundary curve may not be unique in general,  \cite[Thm.\,B]{Seppi} and a bound on the quasiconformal constant of the Weil--Petersson quasicircle together imply uniqueness when the Loewner energy of $\g$ is small enough. 
 Hence, for sufficiently small $S(\rho)$, Theorem~\ref{thm:WP-leaf} and Corollary~\ref{cor:WP-LE-bound} imply that the foliation $(\g_t)_{t \in \m R}$ uniquely determines  a family of minimal disks of finite total curvature $(\Sigma_t)_{t \in \m R}$ in $\m H^3$, where $\Sigma_t$ is bounded by the leaf $\g_t$. We believe that the family $(\Sigma_t)_{t \in \m R}$ forms a smooth foliation of $\m H^3$ in this case. Such families of minimal surfaces seem interesting to study in their own right and, by embedding into a dynamical family, could be useful in the analysis of minimal surfaces in $\m H^3$ bounded by Weil--Petersson quasicircles and in deriving a rigorous AdS$_3/$CFT$_2$ holographic principle.

\bibliographystyle{abbrv}
\bibliography{ref}

\end{document}